\documentclass[12pt,oneside]{amsart}
\usepackage[parfill]{parskip}
\usepackage{tikz-cd}
\usepackage{amsmath}
\usepackage{geometry}
\usepackage{latexsym}
\usepackage{amsthm}
\usepackage{amssymb}
\usepackage{amscd}
\usepackage{bm}
\usepackage[font=small,labelfont=bf]{caption}
\usepackage{graphicx}
\usepackage{epsfig}
\usepackage{epstopdf}
\usepackage{placeins}
\usepackage{mathtools}
\usepackage{subfigure}
\usepackage{stackengine}
\usepackage{appendix}

\stackMath

\DeclarePairedDelimiter{\ceil}{\lceil}{\rceil}
\DeclarePairedDelimiter{\floor}{\lfloor}{\rfloor}
\providecommand{\N}{\mathbb{N}}
\providecommand{\R}{\mathbb{R}}

\providecommand{\Z}{\mathbb{Z}}
\providecommand{\vol}{\mbox{vol}}

\providecommand{\abs}[1]{\left\vert#1\right\vert}
\providecommand{\norm}[1]{\left\Vert#1\right\Vert}
\providecommand{\set}[1]{\left\{#1\right\}}

\providecommand{\floor}[1]{\lfloor\{#1\rfloor\}}

\providecommand{\paren}[1]{\left( #1 \right)}

\providecommand{\brac}[1]{\left [ #1 \right]}
\providecommand{\set}[1]{\left { #1 \right}}

\newcommand{\s}[1]{\begin{equation*} \begin{split} #1 \end{split} \end{equation*}}

\providecommand{\eqref}[1]{\left(\ref{#1}\right)}

\newtheorem{theorem}{Theorem}
\newtheorem{Lemma}[theorem]{Lemma}
\newtheorem{conj}[theorem]{Conjecture}

\newtheorem{Definition}[theorem]{Definition}
\newtheorem{Proposition}[theorem]{Proposition}

\newtheorem{corollary}[theorem]{Corollary}
\newtheorem{counterexample}[theorem]{Counterexample}

\newcommand{\PH}{\mathit{PH}}

\begin{document}
\title[Fractal Dimension and Persistent Homology]{Fractal Dimension and the Persistent Homology \\ of Random Geometric Complexes}
\author{Benjamin Schweinhart}
\date{June 2020}

\begin{abstract}
We prove that the fractal dimension of a metric space equipped with an Ahlfors regular measure can be recovered from the persistent homology of random samples. Our main result is that if $x_1,\ldots, x_n$ are i.i.d.~samples from a $d$-Ahlfors regular measure on a metric space, and  $E^0_\alpha\left(x_1,\ldots,x_n\right)$ denotes the $\alpha$-weight of the minimum spanning tree on $x_1,\ldots,x_n:$
\[E_\alpha^0\left(x_1,\ldots,x_n\right)=\sum_{e\in T\left(x_1,\ldots,x_n\right)} |e|^\alpha\,,\]
then there exist constants $0<C_1\leq C_2$ so that 
\[C_1\leq n^{-\frac{d-\alpha}{d}} E^0_\alpha\left(x_1,\ldots,x_n\right)\leq C_2\,\]
with high probability as $n\rightarrow \infty.$ In particular, $d$ can be recovered from the limit
\[\log\big(E^0_\alpha(x_1,\ldots,x_n)\big)/\log(n)\longrightarrow (d-\alpha)/d\,.\]
This is a generalization of a result of Steele~\cite{1988steele} from the non-singular case to the fractal setting. We also construct an example of an Ahlfors regular measure for which the limit $\lim_{n\rightarrow\infty} n^{-\frac{d-\alpha}{d}} E^0_\alpha\left(x_1,\ldots,x_n\right)$ does not exist with high probability, and prove analogous results for weighted sums defined in terms of higher dimensional persistent homology.
\end{abstract}
\maketitle

\section{Introduction}

The first precise notion of a fractal dimension was proposed by Hausdorff in 1918~\cite{2003edgar,1918hausdorff}. Since then, many  other definitions have been put forward, including the box~\cite{1928bouligand} and correlation~\cite{1983grassberger} dimensions. These quantities do not agree in general, but coincide on a class of regular sets. Fractal dimension was popularized by Mandelbrot in the 1970s and 1980s~\cite{1982mandelbrot,1977mandelbrot}, and it has since found a wide range of applications in subjects including medicine~\cite{2000baish,2009lopes}, ecology~\cite{2004halley}, materials science~\cite{1999davies,2008yu}, and the analysis of large data sets~\cite{2000barbara,2010traina}. It is also important in pure mathematics and mathematical physics, in disciplines ranging from dynamics~\cite{1980takens} to probability~\cite{2008beffara}.

Recently, there has been a surge of interest in applications of topology, and of persistent homology in particular. Several authors have proposed estimators of fractal dimension defined in terms of minimum spanning trees and higher dimensional persistent homology~\cite{2019adams, 2012macpherson, 2014mate,2000robins, 1992weygaert}, and provided empirical evidence that those quantities agreed with classical notions of fractal dimension. Here, we provide the first rigorous justification for the use of random minimum spanning trees and higher dimensional persistent homology to estimate fractal dimension. 

We define a persistent homology dimension for measures (Definition~\ref{defn_phdim}) and prove that it equals the Hausdorff dimension for a wide class of ``regular'' measures (Corollary~\ref{thm_dimension}). In concurrent, separate work with J. Jaquette~\cite{2019jaquette}, we implement an algorithm to compute this persistent homology dimension and provide computational evidence that it performs as well or better than classical dimension estimation techniques.

Informally, a set of ``fractal dimension'' $d$ is self-similar in the sense that its ``local properties'' measured at scale $\epsilon$ scale as $\epsilon^d$ or $\epsilon^{-d}$ for some positive real number $d$ that may be non-integral. This is not well-defined in general, and there exist ``multifractals'' for which different local properties give different values of $d.$ Here, we assume a standard regularity hypothesis that implies that the fractal dimension of a measure is well-defined in the sense that the various classical notions of fractal dimension --- including the Hausdorff, box-counting, and  correlation dimensions --- coincide. This is done by taking the volumes of balls centered at points in our set as the defining ``local property.''
\begin{Definition}[\cite{1956ahlfors,1993david}]
\label{defn:ahlfors}
A probability measure $\mu$ supported on a metric space $X$ is $d$-Ahlfors regular if there exist positive real numbers $c$ and $\delta_0$ so that
\begin{equation}
\label{ahlfors_equation}
\frac{1}{c}\, \delta^{d} \leq \mu\paren{B_\delta\paren{x}} \leq c\,\delta^{d}
\end{equation}
for all $x\in X$ and $\delta<\delta_0,$ where $B_\delta\paren{x}$ is the open ball of radius $\delta$ centered at $x.$
\end{Definition}
Ahlfors regularity is a common hypothesis used when studying geometry and analysis in the fractal setting~\cite{1993david,1997david,2010kozma,2010mackay,2017orponen}. If $\mu$ is $d$-Ahlfors regular on $X$ then it is comparable to the $d$-dimensional Hausdorff measure on $X,$ and the Hausdorff measure is itself $d$-Ahlfors regular. Examples of Ahlfors regular measures include the natural measures on the  Sierpi\'{n}ski triangle and Cantor set, and, more generally, on any self-similar subset of Euclidean space defined by an iterated function system whose correct-dimensional Hausdorff measure is positive~\cite{2015farkas} (a weaker requirement than the usual open set condition); a certain well-studied measure on the boundary of certain hyperbolic groups including the fundamental group of a compact, negatively curved manifold~\cite{1993coornaert,1997david}; and bounded probability densities on a compact manifold, either with the intrinsic metric or one induced by an embedding in Euclidean space (these are indeed ``self-similar'' sets). As such, our methods and results will  be more general than previous papers on the absolutely continuous case. Standard arguments used in proofs for the non-singular case do not work here, and laws of large numbers that follow from them are false for some Ahlfors regular measures (as we will see in Counterexample~\ref{prop:noLimit} below).  

We study the asymptotic behavior of random variables of the form 
\[E_{\alpha}^i\paren{x_1,\ldots,x_n}=\sum_{I \in \PH_i\paren{x_1,\ldots,x_n}} \abs{I}^{\alpha}\,,\]
where $\set{x_j}_{j\in\N}$ are i.i.d.~samples from a probability measure $\mu$ on a metric space, $\PH_i\paren{x_1,\ldots,x_n}$ denotes the $i$-dimensional reduced persistent homology of the \v{C}ech or Vietoris--Rips complex of $\set{x_1,\ldots,x_n},$ and $\abs{I}$ is the length of a persistent homology interval. Unless otherwise specified, our results apply to the persistent homology of either  the \v{C}ech or Vietoris--Rips complex, though the constants may differ. The case where $i=0$ and $\mu$ is absolutely continuous is already well-studied, under a different guise: if $T(x_1,\ldots,x_n)$ denotes the minimum spanning tree on $x_1,\ldots,x_n$ and
\[E_\alpha\paren{x_1,\ldots,x_n}=\sum_{e\in T\paren{x_1,\ldots,x_n}} \abs{e}^\alpha\,,\]
 then
\[E_\alpha\paren{x_1,\ldots,x_n}= E_\alpha^0\paren{x_1,\ldots,x_n}\]
where persistent homology is taken with respect to the Vietoris--Rips complex.

 In 1988, Steele~\cite{1988steele} proved the following celebrated result.
\begin{theorem}[Steele]
\label{theorem_steele}
Let $\mu$ be a compactly supported probability measure on $\R^m,$ $m\geq 2,$ and let $\set{x_n}_{n\in\N}$ be i.i.d.~samples from $\mu.$ If $0<\alpha<m,$
\[\lim_{n\rightarrow\infty} n^{-\frac{m-\alpha}{m}} E_\alpha^0\paren{x_1,\ldots,x_n} \rightarrow c\paren{\alpha,m}\int_{\R^m}f\paren{x}^{\paren{m-\alpha}/m}\;dx\]
with probability one, where $f\paren{x}$ is the probability density of the absolutely continuous part of $\mu,$ and $c\paren{\alpha,m}$ is a positive constant that depends only on $\alpha$ and $m.$
\end{theorem}
Steele wrote~\cite{1988steele}:
\begin{quote}
One feature of Theorem~\ref{theorem_steele} that should be noted is that if $\mu$ has bounded support and $\mu$ is singular with respect to Lebesgue measure, then we have with probability one that $E_\alpha^0\paren{x_1,\ldots,x_n}=o\paren{n^{\paren{d-\alpha}/d}}.$
Part of the appeal of this observation is the indication that the length of the minimum spanning tree is a measure of the \emph{dimension} of the support of the distribution. This suggests that the asymptotic behavior of the minimum spanning tree might be a useful adjunct to the concept of dimension in the modeling applications and analysis of fractals; see, e.g.,~\cite{1977mandelbrot}.
\end{quote}
 However, despite many subsequent sharper and more general results for non-singular measures on Euclidean space~\cite{1992aldous,1996kesten,2000yukich} and Riemannian manifolds~\cite{2004costa}, little is known about the asymptotic properties of random minimum spanning trees  for singular measures. As far as we know, the only previous result toward that end is that of Kozma, Lotker and Stupp~\cite{2010kozma}, who proved that if $\mu$ is a  $d$-Ahlfors regular measure with connected support, then the length of the longest edge of a minimum spanning tree on $n$ i.i.d.~points sampled from $\mu$ is $\approx \paren{\log\paren{n}/n}^{1/d},$ where the symbol $\approx$ denotes that the ratio between the two quantities is bounded between two positive constants that do not depend on $n.$ They also raised the possibility that analogous asymptotics hold for the alpha-weight of a minimum spanning tree, which we prove here in Theorem~\ref{thm_mst}.

More recently, as the field of stochastic topology has matured, several studies have examined the properties of the higher dimensional persistent homology of random geometric complexes defined by absolutely continuous measures on Euclidean space~\cite{2018bauer,2017bobrowski,2018bobrowski,2016duy,2017yogeshwaran}. Most relevantly, Divol and Polonik~\cite{2018divol} proved a strong law of large numbers akin to Steele's theorem for the persistent homology of points sampled from  bounded, absolutely continuous probability densities on $\brac{0,1}^m.$ In the non-persistent setting, several authors have investigated the homology of random geometric complexes on manifolds~\cite{2015bobrowski,2017bobrowski_b,2019kergorlay, 2008niyogi}. However, as far as we know, the current work is the first to study persistent homology of random geometric complexes beyond the world of absolutely continuous measures on $\mathbb{R}^m$ with convex support (with the exception of our unpublished manuscript~\cite{2018schweinhart_b}, which has largely been subsumed into the current work). With these broader hypotheses, we encounter difficult geometric issues related to non-locality and non-triviality of persistent homology, which we discuss below.

\begin{figure}
\centering
\includegraphics[width=.7\textwidth]{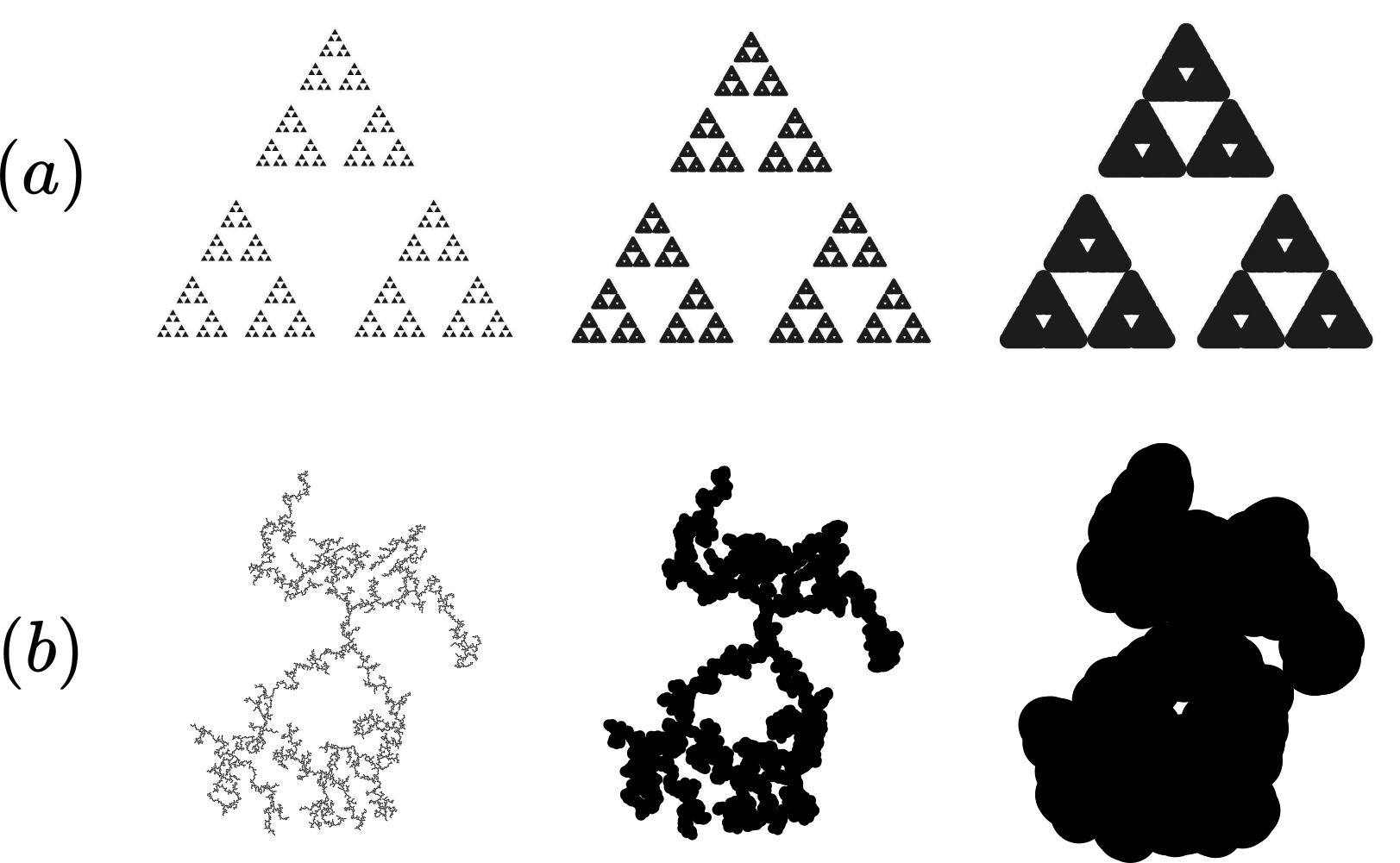}
\caption{\label{fig:epsilon_neighborhood} Two sets of fractional dimension, and their $\epsilon$-neighborhoods: (a) a modified Sierpi\'{n}ski triangle and (b) a branched polymer. Their complex geometry is reflected by their persistent homology.}
\label{fig:triangles}
\end{figure}

A relationship between persistent homology and fractal dimension has been observed in several experimental studies. In 1991, Weygaert, Jones, and Martinez~\cite{1992weygaert} proposed using the asymptotics of $E_\alpha^0\paren{x_1,\ldots,x_n}$ to estimate the generalized Hausdorff dimensions of chaotic attractors. The PhD thesis of Robins, which was arguably one of the first publications in the field of topological data analysis, studied the scaling of Betti numbers of fractals and proved results for the $0$-dimensional persistent homology of disconnected sets~\cite{2000robins}. In joint work with Robert MacPherson, we proposed a dimension for probability distributions of geometric objects based on persistent homology in 2012~\cite{2012macpherson}. Note that the quantities studied in that paper and in the thesis of Robins measure the complexity of a shape rather than the fractional dimension. Most recently, Adams et al.~\cite{2019adams} defined a persistent homology dimension for measures in terms of the asymptotics of $E_i^1\paren{x_1,\ldots,x_n}$. Their computational experiments helped to inspire this work. We study a modified version of their dimension here (Definition~\ref{defn_phdim}), and find hypotheses under which it agrees with the Ahlfors dimension (Corollary~\ref{thm_dimension}). 

In the extremal setting, Kozma, Lotker and Stupp~\cite{2006kozma} defined a minimum spanning tree dimension for a metric space $M$ in terms of the behavior of $E_\alpha^0\paren{Y}$ as $Y$ ranges over all subsets of $M,$ and proved that it equals the upper box dimension. In 2018, we generalized this concept to higher dimensional persistent homology and established hypotheses under which it agrees with the upper box dimension~\cite{2018schweinhart}. In the course of this work, we investigated extremal questions about the number of persistent homology intervals of a set of $n$ points; these questions are also important in the probabilistic context, as we describe below.

\subsection{Our Results for Minimum Spanning Trees}
Our main result is :
\begin{theorem}
\label{thm_mst}
Let $\mu$ be a $d$-Ahlfors regular measure on a metric space, and let $\set{x_n}_{n\in\N}$ be i.i.d.~samples from $\mu.$ If $0<\alpha<d,$ then
\[E^0_\alpha\paren{x_1,\ldots,x_n}\approx n^{\frac{d-\alpha}{d}}\]
with high probability as $n\rightarrow \infty,$ where the symbol $\approx$ denotes that the ratio of the two quantities is bounded between positive constants that do not depend on $n.$
\end{theorem}
We provide a proof of this result using the language of minimum spanning trees (rather than persistent homology) in Section~\ref{sec_MST}. The special case where $\mu$ is a measure on Euclidean space is also a consequence of either Theorems~\ref{thm_extremal} or~\ref{thm_probabalistic} below. 

The hypotheses we require to prove Theorem~\ref{thm_mst} and our other results below are somewhat weaker than Ahlfors regularity. In particular, the proofs of our upper bounds only require that $M_{\delta}\paren{\mu}=O\paren{\delta^{-d}},$ where $M_{\delta}\paren{\mu}$ is the maximal number of disjoint balls of radius $\delta$ centered at points of $\text{supp}\;\mu.$ Also, the proofs of our lower bounds require that the uniform bounds in Equation~\ref{ahlfors_equation} are satisfied on a set of positive measure, but not necessarily at every point in the support of $\mu.$ However, a regularity hypothesis on the underlying measure is necessary. Some definitions of fractals include the chaotic attractors studied in Section~4 of our computational paper~\cite{2019jaquette}. Our computations suggest that for several examples and each $\alpha>0$ there is a different value of $d_\alpha$ so that $E^0_\alpha\paren{y_1,\ldots,y_n}\approx n^{\frac{d_\alpha-\alpha}{d_\alpha}}$ (i.e. that the measure is ``multifractal''). In particular, we could not replace $d$ in the previous theorem with, say, the upper box or Hausdorff dimension of the support. 

\begin{figure}
\centering
\includegraphics[width=\textwidth]{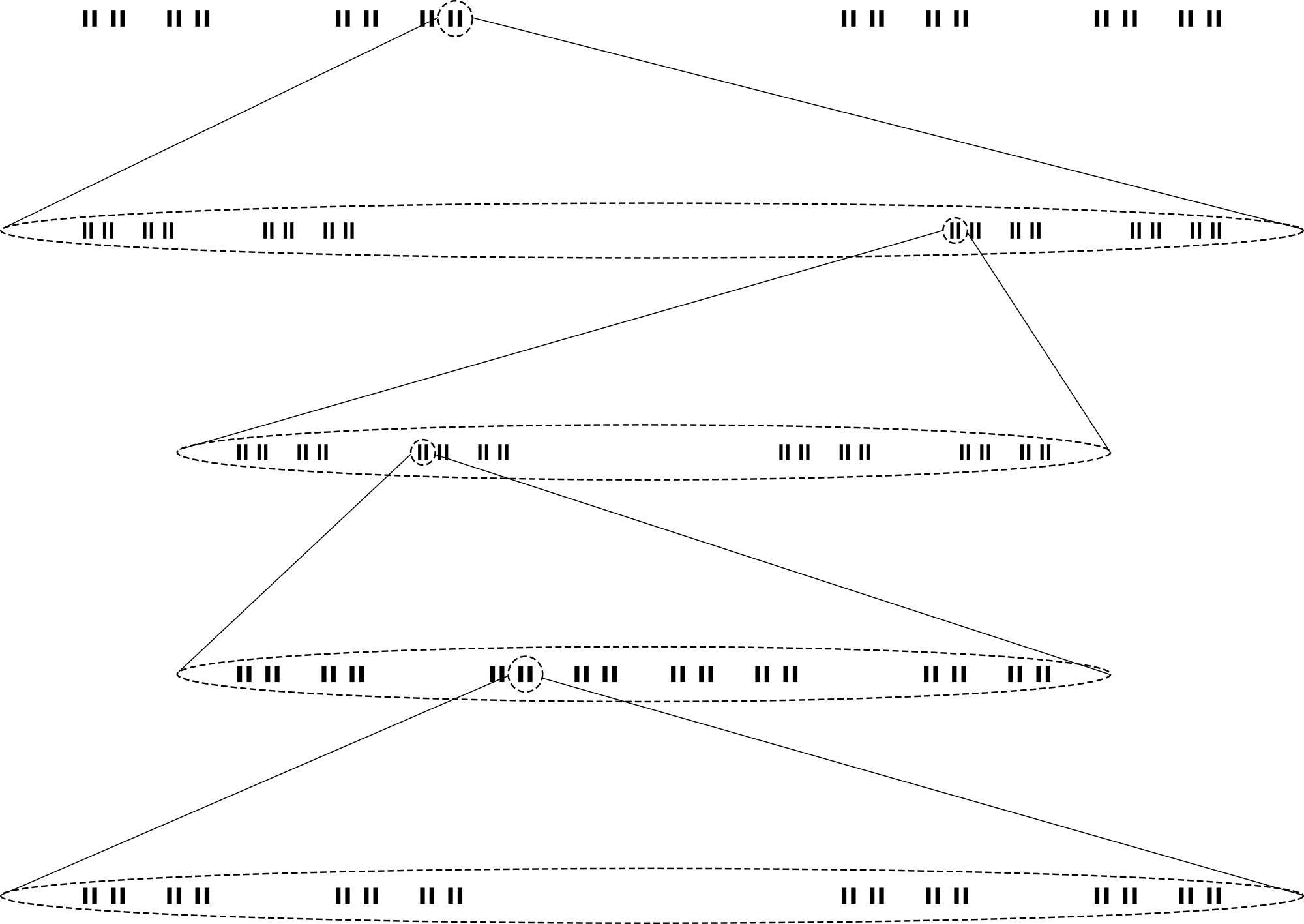}
\caption{A rough depiction of the construction of Counterexample~\ref{prop:noLimit}, where the size of each row is shown relative to the standard middle thirds Cantor set $C$ at the same scale. At larger scales, it resembles $C$ (first row). At a smaller scale, we shrink the pieces of the set by $\frac{5}{7}$ relative to $C$ (second row) resulting in structure that resembles $C$ rescaled by $\frac{5}{7}$ when zooming in further (third row). Even smaller, we re-expand pieces of the structure by $\frac{7}{5}$ relative to $C$ (fourth row) so that at finer scales it resembles $C$ again (fifth row). We repeat this process, resulting in a structure that alternates between resembling $C$ and $\frac{5}{7}C$ at different length-scales.}
\label{fig:CantorConstruction}
\end{figure}

Our next result shows that a sharper law of large numbers as in Theorem~\ref{theorem_steele} is false in general for Ahlfors regular measures. 
\begin{counterexample}
\label{prop:noLimit}
Let $d=\log\paren{2}/\log\paren{3}$ and $0<\alpha<d.$ There exists a $d$-Ahlfors regular measure $\mu$ on $\brac{0,1}$ so that $\lim_{n\rightarrow\infty}n^{-\frac{d-\alpha}{d}} E_\alpha^0\paren{x_1,\ldots,x_n}$ does not converge in probability.
\end{counterexample}
In particular, we construct an example of a $d$-Ahlfors regular measure where \newline $n^{-\frac{d-\alpha}{d}} E_\alpha^0\paren{x_1,\ldots,x_n}$ oscillates between two positive real numbers with high probability. We provide a brief description of the counterexample here, and a complete proof in \ref{sec:noLimit}. The construction can easily be modified to produce a counterexample of of dimension $d$ for any $d\in\paren{0,1}$ as described at the end of \ref{sec:noLimit}, but we concentrate on the case $d=\log\paren{2}/\log\paren{3}$ for clarity.

Recall that the standard middle-thirds Cantor set $C$ is constructed as the intersection of a nested sequence of closed sets $T_1\supset T_2 \supset\ldots,$ where $T_k$ consists of $2^k$ disjoint intervals of length $\paren{\frac{1}{3}}^k$. Our counterexample will resemble $C$ at some scales, and $C$ rescaled by $\frac{5}{7}$ at others (see Figure~\ref{fig:CantorConstruction}). It will be supported on another nested sequence of of closed sets $S_1\supset S_2 \ldots.$ To construct the counterexample, we suppose that $\lim_{n\rightarrow\infty}n^{-\frac{d-\alpha}{d}} E_\alpha^0\paren{y_1,\ldots,y_n}\coloneqq c$ exists in probability, where $\set{y_j}_{j\in\mathbb{N}}$ are i.i.d. samples from the natural measure on $C.$ We set $S_k=T_k$ for all sufficiently small $k$ which ensures that, to a certain point, $n^{-\frac{d-\alpha}{d}} E_\alpha^0\paren{x_1,\ldots,x_n}$ also approaches $c.$ At smaller length-scales $S_k$ instead consists of $2^k$ intervals of length $\frac{5}{7} \paren{\frac{1}{3}}^k,$ and we will show that  $n^{-\frac{d-\alpha}{d}} E_\alpha^0\paren{x_1,\ldots,x_n}$ will dip down toward $\paren{\frac{5}{7}}^\alpha c.$ We repeat this inductively, resulting in quantities that oscillate between $c$ and $\paren{\frac{5}{7}}^\alpha c.$

\subsection{Our Results for Higher Dimensional Persistent Homology}

As we noted in our earlier paper~\cite{2018schweinhart}, proving asymptotic results for higher dimensional persistent homology is challenging due to extremal questions about the number of persistent homology intervals of a finite point set. While a minimum spanning tree on $n$ points always has $n-1$ edges, a set of $n$ points may have trivial $\PH_i$ for all $i>0,$ and there exist families of finite metric spaces for which the number of persistent homology intervals grows faster than linearly in the number of points.

To prove upper bounds for the asymptotics of $E_\alpha^i$ for $i>0,$ we require either extremal or probabilistic control of the number of persistent homology intervals of a set of $n$ points. Families of point sets in Euclidean space with more than a linear number of persistent homology intervals exist~\cite{2011goff,2018schweinhart}, but are considered somewhat pathological. As far as we know, the Upper Bound Theorem~\cite{1975stanley} on the number of faces of a neighborly polytope provides the best extremal upper bound for the number of persistent homology intervals of the \v{C}ech complex of a finite subset of $\R^m$:
\[\abs{\PH_i\paren{x_1,\ldots,x_n}} =\begin{cases}
O\paren{n^{i+1}} & i<\floor{\frac{m}{2}}\\
O\paren{n^{\floor{\frac{m+1}{2}}}} & i\geq \floor{\frac{m}{2}}
\end{cases} \]
For the Vietoris--Rips complex of points in Euclidean space, we~\cite{2018schweinhart} showed that 
\[\abs{\PH_1\paren{x_1,\ldots,x_n}}= O\paren{n}\]
by modifying an argument of Goff~\cite{2011goff}.
A different extremal question arises in the process of proving lower bounds for $E_{\alpha}^{i}.$ In particular, a subset $\R^m$ must have dimension above a certain non-triviality constant $\gamma_i^m$ (defined in Section~\ref{sec_nontriviality}) to guarantee the existence of subsets with non-trivial $i$-dimensional persistent homology. Note that $\gamma_i^m$ may depend on whether persistent homology is taken with respect to the \v{C}ech complex or Vietoris--Rips complex. Unless otherwise noted, that dependence is left implicit. We showed that  $\gamma_1^m<m-1/2$ for the \v{C}ech complex in our previous paper~\cite{2018schweinhart}.

The proofs of the upper bounds in the next two theorems work for Ahlfors regular measures on arbitrary metric spaces, but the lower bound requires that the measure is defined on a subset of Euclidean space.
\begin{theorem}
\label{thm_extremal}
Let $\mu$ be a $d$-Ahlfors regular measure on $\mathbb{R}^m$ with $d>\gamma_i^m$, and let $\set{x_n}_{n\in\N}$ be i.i.d.~samples from $\mu.$  If there are positive real numbers $D$ and $a$ so that 
\[\abs{\PH_i\paren{x_1,\ldots,x_n}}<D n^a\]
for all finite subsets of $\text{supp}\;\mu,$ and $0<\alpha<ad,$ then there are real numbers $0<\zeta<Z$ so that

\[\zeta\, n^{\frac{d-\alpha}{d}} \leq E_{\alpha}^{i}\paren{x_1,\ldots,x_n}\leq Z\,n^{\frac{a d-\alpha}{d}}\]
with high probability, as $n\rightarrow \infty.$ In fact, the upper bound holds with probability one. 
\end{theorem}

The upper bound is shown in Proposition~\ref{prop_upper_extremal}, and the lower bound in Proposition~\ref{prop_lower}. The following is a corollary, using our previous results on $\gamma_1^2$ and the fact that the Alpha complex of a set of $n$ points in $\mathbb{R}^2$ has $O\paren{n}$ faces.

\begin{corollary}
\label{corollary_R2}
Let $\mu$ be a $d$-Ahlfors regular measure on $\mathbb{R}^2$, and let $\set{x_n}_{n\in\N}$ be i.i.d.~samples from $\mu.$ If $0<\alpha<d.$ If $d>1.5,$ $0<\alpha<d,$ and persistent homology is taken of the \v{C}ech complex, then 
\[ E_{\alpha}^{1}\paren{x_1,\ldots,x_n} \approx n^{\frac{d-\alpha}{d}}\]
with high probability as $n\rightarrow \infty.$ In fact, the upper bound holds with probability one. 
\end{corollary}

Another corollary, based on our results on the Rips complex~\cite{2018schweinhart}, is
\begin{corollary}
\label{corollary_rips}
Let $\mu$ be a $d$-Ahlfors regular measure on $\mathbb{R}^m$, and let $\set{x_n}_{n\in\N}$ be i.i.d.~samples from $\mu.$ If persistent homology is taken of the Rips complex, $d>\gamma_1^m,$ and $0<\alpha<d,$ then 
\[ E_{\alpha}^{1}\paren{x_1,\ldots,x_n} \approx n^{\frac{d-\alpha}{d}}\]
with high probability as $n\rightarrow \infty.$ In fact, the upper bound holds with probability one. 
\end{corollary}

For $i>0$ and $m>2,$ we show better upper bounds for $d$-Ahlfors regular measures for which the expectation and variance of $\abs{\PH_i\paren{x_1,\ldots,x_n}}$ scale linearly and sub-quadratically, respectively. These quantities can be measured in practice, allowing one to determine whether higher dimensional persistent homology would be suitable for dimension estimation in applications.

\begin{theorem}
\label{thm_probabalistic}

Let $\mu$ be a $d$-Ahlfors regular measure on $\mathbb{R}^m$ so that $d>\gamma_i^m,$ and let $\set{x_n}_{n\in\N}$ be i.i.d.~samples from $\mu.$ If

\[\mathbb{E}\paren{\abs{\PH_i\paren{x_1,\ldots,x_n}}}=O\paren{n} \quad \text{and}\quad \textit{Var}\paren{\abs{\PH_i\paren{x_1,\ldots,x_n}}}/n^2\rightarrow 0\]
and $0<\alpha<d,$ then there are real numbers $0<\lambda<\Lambda$ so that
\[\lambda\,n^{\frac{d-\alpha}{d}} \leq E_{\alpha}^{i}\paren{x_1,\ldots,x_n}\leq \Lambda \, n^{\frac{d-\alpha}{d}}\log\paren{n}^{\frac{\alpha}{d}}\]
with high probability, as $n\rightarrow \infty.$ 
\end{theorem}

The upper and lower bounds are shown in Propositions~\ref{prop_upper_prob} and~\ref{prop_lower}, respectively. Many of our other results can be viewed as special cases of this theorem, including Corollaries~\ref{corollary_R2} and~\ref{corollary_rips} and the particular case of Theorem~\ref{thm_mst} where the measure is supported on Euclidean space. More generally, although there are few rigorous results on the scaling of the number of persistent homology intervals in higher dimensions, computational results indicate that these hypotheses hold broadly --- see the Appendix. Also, Stemeseder~\cite{2014stemeseder} showed that any positive, continuous probability density on the $m$-dimensional Euclidean sphere satisfies the hypothesis on the expected number of intervals, and the uniform measure on the sphere satisfies the hypothesis on the variance. However, we think these are interesting hypotheses not because they are easy to prove but because they can be estimated in data analysis.

\subsection{Dimension Estimation}
\label{sec:dimension}
As we noted earlier, several authors have proposed to use persistent homology for dimension estimation.  Here, we provide the first proof that these methods recover a classical fractal dimension, under certain hypotheses. 

We define a family of $\PH_i$ dimensions of a measure, one for each real number $\alpha>0$ and $i\in\N:$

\begin{Definition}
\label{defn_phdim}
\[\text{dim}_{\PH_i^\alpha}\paren{\mu}=\frac{\alpha}{1-\beta}\,,\]
where
\[\beta=\limsup_{n\rightarrow\infty} \frac{\log\paren{\mathbb{E}\paren{E_{\alpha}^{i}\paren{x_1,\ldots,x_n}}}}{\log\paren{n}}\,.\]
\end{Definition}
That is, $\text{dim}_{\PH_i^\alpha}\paren{\mu}$ is the unique real number $d$ so that 
\[\limsup_{n\rightarrow\infty}\mathbb{E}\paren{E_{\alpha}^{i}\paren{x_1,\ldots,x_n}} n^{-\frac{k-\alpha}{k}}\]
equals $\infty$ for all $k<d,$ and is bounded for $k>d.$ The case $\alpha=1$ is very closely related to the dimension studied by Adams et al.~\cite{2019adams}, and agrees with it if defined.

Theorem~\ref{theorem_steele}~\cite{1988steele} implies that if $\mu$ is a compactly supported, non-singular probability measure on $\R^m,$ then $\text{dim}_{\PH_0^\alpha}\paren{\mu}=m$ for $0<\alpha<m.$ Similarly, the results of Divol and Polonik~\cite{2018divol} show that if $\mu$ is a bounded probability measure on the cube in $\mathbb{R}^m,$ then $\text{dim}_{\PH_i^\alpha}\paren{\mu}=m$ for $0<\alpha<m$ and $0\leq i < m.$ 

The following is a corollary of our theorems on the asymptotic behavior of $E_{\alpha}^{i}$:
\begin{corollary}
\label{thm_dimension}
If $\mu$ is a $d$-Ahlfors regular measure on a metric space and $0<\alpha<d$ then
\[\text{dim}_{\PH_0^\alpha}=d\,.\]
Furthermore, if $\mu$ is defined on $\R^m,$ $d> \gamma_i^m,$ and 
\[\mathbb{E}\paren{\abs{\PH_i\paren{x_1,\ldots,x_n}}}=O\paren{n} \quad \text{and}\quad \textit{Var}\paren{\abs{\PH_i\paren{x_1,\ldots,x_n}}}/n^2\rightarrow 0\,,\]
then 
\[\text{dim}_{\PH_i^\alpha}=d\,.\]
\end{corollary}
This result is weaker than our main theorems, and it can be proven with weaker hypotheses. For example, the upper bound $\text{dim}_{\PH_0^\alpha}\paren{\mu}\leq d$ holds if the hypothesis of $d$-Ahlfors regularity is replaced by the requirement that the upper box dimension of the support of $\mu$ is equal to $d.$

\begin{Proposition}
\label{prop:box}
Let $\mu$ be a measure on a bounded metric space $X,$ and let $\text{dim}_{\text{box}}\paren{X}$ be the upper box dimension of $X$ (defined below). If $\alpha<\text{dim}_{\text{box}}\paren{X}$ then
\[\text{dim}_{\PH_0^\alpha}\paren{\mu}\leq \text{dim}_{\text{box}}\paren{X}\,.\]
\end{Proposition}

In separate experimental work (joint with J. Jaquette), we implement an algorithm to compute the persistent homology dimensions and compare its practical performance  below to classical techniques for estimating fractal dimension, such as box--counting and the estimation of the correlation dimension. The persistent homology dimension (for $i=0$) performs about as well as the correlation dimension, both in terms of the convergence rate and speed of computation, 
and significantly better than the box dimension.~\cite{2019jaquette} Our results here imply that the computational estimates in~\cite{2019jaquette} will converge with high probability as the number of samples goes to infinity for several of the examples considered. These include the $\PH_0$ dimension of the Cantor dust, Cantor set cross an interval, Sierpi\'{n}ski triangle, and Menger sponge, and the $\PH_1$ dimensions of the Cantor set cross an interval and the Sierpi\'{n}ski triangle.

\subsection{A Conjecture}
We conjecture that if the persistent homology of the support of an Ahlfors regular measure is trivial, then the Lebesgue measure can be replaced with the $d$-dimensional Hausdorff measure $\mathcal{H}^d$ in Theorem~\ref{theorem_steele}. Note that this would exclude Counterexample~\ref{prop:noLimit}.
\begin{conj}
Let $\mu$ be a $d$-Ahlfors regular measure on a metric space $M$ and let $\set{x_n}_{n\in\N}$ be i.i.d.~samples from $\mu.$ If $\PH_0\paren{\text{supp}\;\mu }$ is trivial and $0<\alpha<d,$ then
\[\lim_{n\rightarrow\infty} n^{-\frac{d-\alpha}{d}} E_\alpha^0\paren{x_1,\ldots,x_n} \rightarrow c_0\paren{\alpha,d}\int_{M}f\paren{x}^{\paren{d-\alpha}/d}\;dx\]
with probability one, where $f\paren{x}$ is the probability density of the absolutely continuous part of $\mu$ with respect to the $d$-dimensional Hausdorff measure $\mathcal{H}^d$ and  $c_0\paren{\alpha,d}$ is a continuous function of $\alpha$ and $d.$

Furthermore, if $\mu$ is supported on $\mathbb{R}^m,$ $d>\gamma_i^m,$ and  $\PH_i\paren{\text{supp}\;\mu }$ is trivial then 
\[\lim_{n\rightarrow\infty} n^{-\frac{d-\alpha}{d}} E_\alpha^i\paren{x_1,\ldots,x_n} \rightarrow c_i\paren{\alpha,d}\int_{M}f\paren{x}^{\paren{d-\alpha}/d}\;dx\]
with probability one, where $f\paren{x}$ is the probability density of the absolutely continuous part of $\mu$ with respect to the $d$-dimensional Hausdorff measure $\mathcal{H}^d$ and  $c_i\paren{\alpha,d}$ is a continuous function of $\alpha$ and $d.$

\end{conj}

\section{Preliminaries}
We introduce notation and lemmas that will be used throughout the paper. Lemma~\ref{lemma_ballcount} controls the asymptotics of the maximal number of disjoint balls centered at points in the support of an Ahlfors regular measure, and will be applied in many of our arguments. In Section~\ref{sec:occupancy}, we define occupancy indicators in terms of collections of subsets of a metric space, and prove a weak law of large numbers for them. Later in the paper, we will use these occupancy indicators to define events implying the existence of a minimum spanning tree edge or persistent homology interval of a certain length.

\subsection{Notation}
In the following, $X$ will denote a metric space and $\textbf{x}$ will denote a finite point set with an unspecified number of elements. Furthermore, $\textbf{x}_n$ will be shorthand for a finite point set $\set{x_1,\ldots,x_n}\subset X$ containing $n$ points. If the measure $\mu$ is obvious from the context, $\set{x_j}_{j\in\N}$ will be a collection of independent random variables with common distribution $\mu.$ Finally, we will use symbols with the ``mathcal'' font (i.e. $\mathcal{A},\mathcal{B},\ldots$) for collections of sets. 

\subsection{Ahlfors Regularity and Ball-counting}

Let $X$ be a metric space, and let $M_{\delta}\paren{X}$ be the maximal number of disjoint open balls of radius $\delta$ centered at points of $X.$ The upper box dimension is defined in terms of  the asymptotic properties of $M_\delta\paren{X}.$
\begin{Definition}
\label{defn:upperbox}
Let $X$ be a bounded metric space. The upper box dimension of $X$ is
\s{\text{dim}_{box}\paren{X}=\limsup_{\delta\rightarrow 0}{\frac{\log{\paren{M_\delta\paren{X}}}}{-\log\paren{\delta}}}.}
\end{Definition}

If $X$ admits a $d$-Ahlfors regular measure, we can control the behavior of $M_\delta\paren{X}.$ 

\begin{Lemma}[Ball-counting Lemma]
\label{lemma_ballcount}
If $\mu$ is a is $d$-Ahlfors regular measure supported on a metric space $X$ then
\[\frac{1}{c} 2^{-d} \, \delta^{-d} \leq M_{\delta}\paren{X}\leq c \, \delta^{-d}\]
for all $\delta<\delta_0,$ where $c$ and $\delta_0$ are the constants given in Definition~\ref{defn:ahlfors}.
\end{Lemma}
\begin{proof}
Let $\set{x_j}_{j=1}^{M_\delta\paren{X}}$ be the centers of a maximal set of disjoint balls of radius $\delta$ centered at points of $X.$
\begin{align*}
1=&\;\;\mu\paren{X}\\
\geq & \;\; \sum_{j=1}^{M_\delta\paren{\mu}} \mu\paren{B_{\delta}\paren{x_j}}&&\text{by disjointness}\\
\geq & \;\; \frac{1}{c} \delta^d M_\delta\paren{\mu}&&\text{by Ahlfors regularity}\\
\implies & \;\; M_\delta\paren{\mu}\leq c\delta^{-d}\,.
\end{align*}

The maximality of $\set{B_\delta\paren{x_i}}_{i=1}^{M_\delta\paren{\mu}}$ implies that the balls of radius $2\delta$ centered at the points $\set{x_i}$ cover $X.$ It follows that 
\begin{align*}
1=&\;\;\mu\paren{X}\\
&\;\; \leq \sum_{j=1}^{M_\delta\paren{X}} \mu\paren{B_{2\delta}\paren{x_j}}\\
\leq & \;\; c 2^d \delta^d M_\delta\paren{X} &&\text{by Ahlfors regularity}\\
\implies & \;\; M_\delta\paren{X}\geq \frac{1}{c}2^{-d} \delta^{-d}\,,
\end{align*}
as desired. 
\end{proof}

\subsection{Occupancy Indicators}
\label{sec:occupancy}
Our strategy for proving lower bounds for the asymptotic behavior of $E_\alpha^i(x_1,\ldots,x_n)$ will be to define certain \textbf{occupancy indicators} that imply the existence of a persistent homology interval (or minimum spanning tree edge) whose length is bounded away from zero.

If $A$ and $B$ are sets define
\[\delta\paren{A,B}=\begin{cases}
0 & A\cap B=\varnothing\\ 
1 & A\cap B \neq \varnothing\\
\end{cases}
\,.\]

Also, If $A$ is a set and $\mathcal{B}$ is a collection of sets define the occupancy indicator
\[\Xi\paren{\textbf{x},A,\mathcal{B}}=\begin{cases}
 1 & \delta\paren{A,\textbf{x}}=0 \quad \quad \text{and}\quad \quad \delta\paren{B,\textbf{x}}=1\quad \forall\, B\in\mathcal{B}\\
 0 & \text{otherwise}
\end{cases}\,.\]
All occupancy indicators considered in this paper will satisfy $A\cap B=\varnothing$ for all $B\in \mathcal{B},$ and $B_1\cap B_2=\varnothing$ for all $B_1,B_2\in\mathcal{B}$ so that $B_1\neq B_2.$ We say that two occupancy indicators  $\Xi\paren{\textbf{x},A_1,\mathcal{B}}$ and $\Xi\paren{\textbf{x},A_2,\mathcal{C}}$ (where $\textbf{x}$ is the same sample for each) are \textbf{disjoint} if
\[\paren{A_1 \cup \bigcup_{B\in\mathcal{B}} B}\cap \paren{A_2 \cup \bigcup_{C\in\mathcal{C}} C}=\varnothing\,.\]

An \textbf{$n,p,q,r$-bounded occupancy indicator} is a random variable of the form 
\[\Xi\paren{\textbf{x}_n,A,\mathcal{B}}\,,\]
where $\mathcal{B}$ is a collection of at least $r$ sets, and $\textbf{x}_n$ is a collection of $n$ independent random variables with common distribution $\nu$ satisfying 
\[\nu\paren{A}\leq q/n\quad\text{and}\quad\nu\paren{B}\geq p/n\;\;\forall\;B\in\mathcal{B}\,.\]
If the above conditions on $\nu$ and the number of sets in $\mathcal{B}$ hold with equality, we say that $\Xi\paren{\textbf{x}_n,A,\mathcal{B}}$ is a \textbf{$n,p,q,r$-uniform occupancy indicator}.

Disjoint $n,p,q,r$-uniform occupancy indicators satisfy a weak law of large numbers as $n\rightarrow\infty.$ 
\begin{Lemma}
\label{lemma_occupancy_0}
Let $r,a>0,$ and $0<p,q<1.$ Also, for each $n\in\mathbb{N}$ let $X^n_1,\ldots,X^n_{\floor{a\,n}}$ be disjoint $n,p,q,r$-uniform occupancy indicators.  If $Y_n=\frac{1}{n}\sum_{j=1}^{\floor{a\,n}} X^n_j,$ then
\[\lim_{n\rightarrow\infty} Y_n = \gamma\]
in probability, where $\gamma=a e^{-q}\paren{1-e^{-p}}^r.$ 
\end{Lemma}
\begin{proof}
First, we compute the limiting expectation of the events $X_j^n$ as $n\rightarrow\infty$:

\[\mathbb{E}\paren{X_j^n}=\mathbb{P}\paren{X_j^n=1}=\paren{1-\frac{q}{n}}^n\sum_{j=0}^{r}\paren{-1}^{j}{r\choose j}\paren{1-j\frac{p/n}{1-q/n}}^n\]
by inclusion-exclusion. Therefore
\[\lim_{n\rightarrow\infty} \mathbb{E}\paren{X_j^n}=e^{-q}\paren{\sum_{k=0}^{r}\paren{-1}^{k}{r\choose k}e^{-k p}}=e^{-q}\paren{1-e^{-p}}^r\]
where we factored the second term in the middle equation using the binomial theorem. Thus $\lim_{n\rightarrow\infty} \mathbb{E}\paren{Y_n}=\gamma$ by linearity of expectation.

A similar computation shows that if $j\neq k,$ 
\[\lim_{n\rightarrow\infty} \mathbb{E}\paren{X_j^n X_k^n}=e^{-2q}\paren{1-e^{-p}}^{2r}\,.\]

It follows that
\[\lim_{n\rightarrow\infty}\text{Cov}\paren{X_j^n,X_k^n}=\;\; \lim_{n\rightarrow\infty}\paren{\mathbb{E}\paren{X_j^n X_k^n}-\mathbb{E}\paren{X_j^n}\mathbb{E}\paren{X_k^n}}=0\,.\]

Therefore
\begin{align*}
\text{Var}\paren{Y_n}=&\;\; \frac{1}{n^2}\paren{\sum_{j=1}^{\floor{a\,n}}\text{Var}\paren{X_j}+2\sum_{j=1}^{\floor{a\,n}}\sum_{i=1}^{j-1}\text{Cov}\paren{X^n_j,X^n_k}} \\
\sim&\;\; \frac{a}{n}\text{Var}\paren{X^n_1}+a\frac{n^2-n}{n^2}\text{Cov}\paren{X^n_1,X^n_2}\\
\leq & \;\; \frac{a}{n} + a \paren{1-\frac{1}{n}}\text{Cov}\paren{X^n_1,X^n_2}
\end{align*}
also converges to $0$ as $n$ goes to $\infty.$ 

Let $\epsilon>0$ and $0<\rho<1.$ Choose $N$ sufficiently large so that
\[\abs{\mathbb{E}\paren{Y_n}-\gamma}<\epsilon/2\quad \quad \text{and}\quad \quad \text{Var}\paren{Y_n}<\frac{\epsilon^2 \rho}{4}\]
for all $n>N.$ If $n>N,$
\begin{align*}
\mathbb{P}\paren{\abs{Y_n -\gamma} > \epsilon} \leq & \;\; \mathbb{P}\paren{\abs{Y_n-\mathbb{E}\paren{Y_n}} > \epsilon/2}\\
\leq  & \;\;  \mathbb{P}\paren{\abs{Y_n -\mathbb{E}\paren{Y_n}} > \frac{1}{\sqrt{\rho}}\sqrt{\text{Var}\paren{Y_n}}}\\
\leq & \;\; \rho
\end{align*}
by Chebyshev's Inequality.
\end{proof}

The occupancy indicators we define below will not be uniform, but we can use the previous lemma to bound them. To do so, we require a standard lemma on non-atomic measures~\cite{overflow_nonatomic,1958sikorski}.

\begin{Lemma}
If $\mu$ is a non-atomic measure on a metric space $Y,$ and $0<a<\mu\paren{Y}$ then there exists $Y_0\subset Y$ so that $\mu\paren{Y_0}=a.$
\end{Lemma}

\begin{Lemma}
\label{lemma_occupancy_2}
Let $r,a>0,$ $0<p,q<1,$ and $s_n\geq \floor{a n}$ for all $n\in \N.$ Also, for each $n\in\mathbb{N}$ let $X^n_1,\ldots,X^n_{s_n}$ be disjoint $n,p,q,r$-bounded occupancy indicators. Under these hypotheses, there is a $\gamma>0$ so that
\[\lim_{n\rightarrow\infty} \frac{1}{n}\sum_{j=1}^{s_n} X^n_j\geq \gamma\]
with high probability.
\end{Lemma}
\begin{proof}
Let $a_0=\min\paren{a,1/\paren{p+q}},$ $1\leq j \leq \floor{a_0 n},$ and
\[X_j^n=\Xi\paren{\textbf{x}_n,A_j^n,\mathcal{B}_j^n}\,.\]

$\nu$ is non-atomic, so by the previous lemma we can find a subset $\hat{B}$ of each set $B\in\mathcal{B}_j^n$ so that $\nu\paren{\hat{B}}=p/n$. Let 
\[\hat{\mathcal{B}}_j^n=\set{\hat{B}:B\in\mathcal{B}_j^n}\quad\quad\text{and}\quad\quad D_n=\bigcup_{j=1}^{\floor{a_0 n}}\bigcup_{\hat{B}\in\hat{\mathcal{B}}_
n}\hat{B}\,.\]

We will show that there are disjoint sets $\hat{A}_1^n,\ldots, \hat{A}_{\floor{a_0 n}}^n$ so that $A_j^n\subseteq \hat{A}_j^n \subset D_n^c$ and $\nu\paren{\hat{A}_j^n}=q/n$ for $j=1,\ldots,\floor{a_0 n}.$ Let $\hat{D}_n=D_n^c\setminus  \cup_{j} A_j^n.$ The maximum index of $j$ is $\floor{a_0 n}$  and  $a_0\paren{p+q}\leq 1$ so
\[\nu\paren{\hat{D}_0}\geq \sum_{j=1}^{\floor{a_0 n}}\paren{\frac{q}{n}-\nu\paren{A_j^n}}\,.\]
Applying the previous lemma to $\nu\mid_{\hat{D}_n}$ gives a $C_1\subset \hat{D}_n$ with $\nu\paren{C_1}=q/n-\nu\paren{A_1^n},$ so if $\hat{A}_1^n=C_1\cup A_1^n$ then $A_1^n \subseteq \hat{A}_1^n \subset \hat{D}_n$ and $\nu\paren{\hat{A}_1^n}=q/n.$ Assuming we have found $\hat{A}_1^n,\ldots,\hat{A}_k^n$  we can apply the same argument to $\nu\mid_{\hat{D}_n\setminus \cup_{j=1}^k \hat{A}_j^n}$ to find $ \hat{A}_{k+1}^n.$

Let
\[\hat{X}_j^n=\Xi\paren{\textbf{x}_n,\hat{A}_j^n,\hat{\mathcal{B}}_j^n}\,.\]
By construction, $X_j^n=1\implies \hat{X}_j^n=1$ so $\frac{1}{n}\sum_{j=1}^{\floor{s_n}} X^n_j$ stochastically dominates $\frac{1}{n}\sum_{j=1}^{\floor{a_0\,n}} \hat{X}^n_j.$ Applying the Lemma~\ref{lemma_occupancy_0} to the latter sum implies the desired result.
\end{proof}

\section{The Proofs for Minimum Spanning Trees}
\label{sec_MST}
We prove the upper and lower bounds in Theorem~\ref{thm_mst} in Sections~\ref{sec_MST_upper} and ~\ref{sec_MST_lower} below. First, we sketch both proofs. If $\textbf{x}$ is a finite metric space, let $T(\textbf{x})$ denote the minimum spanning tree on $\textbf{x},$ and let $p\paren{\textbf{x},\delta}$ be the number of edges of $T\paren{\textbf{x}}$ of length greater than $\delta.$

To prove the upper bound, we begin by controlling $p\paren{\textbf{x},\delta}$ in terms of the maximal number of disjoint balls of radius $\delta/2$ centered at points of $\textbf{x}_n$ (Lemma~\ref{edge_counting_lemma}). Combining this with the asymptotics we found in the ball-counting lemma above (Lemma~\ref{lemma_ballcount}) gives that $p\paren{\textbf{x}_n,\delta}\leq C\delta^{-d}$ for some constant $C>0$ and all $\textbf{x}_n\subset X.$ We convert this into a bound on $E_i^\alpha\paren{\textbf{x}_n}$ by integrating (Lemma~\ref{integral_lemma}) and using that a minimum spanning tree on $n$ points has $n-1$ edges, yielding the upper bound in Theorem~\ref{thm_mst} (Proposition~\ref{prop_upper_extremal_MST}). 

For the lower bound, we define an occupancy indicator that implies the existence of a minimum spanning tree edge of length at least $\delta,$ by requiring that a ball of radius $\delta$ is occupied and its annulus of radii $\paren{\delta,2\delta}$ is not (Lemma~\ref{lemma_ball_MST} and the preceding text).   Taking a collection of these indicators for a maximal set of disjoint balls of radius $2n^{-1/d}$ and applying Lemma~\ref{lemma_occupancy_2} gives that 
\[p\paren{\textbf{x}_n,n^{-1/d}}\geq \gamma n \]
with high probability as $n\rightarrow\infty$ for some $\gamma>0$ (Lemma~\ref{lemma_dominated_mst}). Summing over edges of length exceeding $n^{-/d}$ proves the lower bound in Theorem~\ref{thm_mst} (Proposition~\ref{prop_lower_MST}).

We use the next lemma in our proofs of both the upper and lower bounds. Let $G_{\textbf{x},\epsilon}$ be graph the with vertex set $\textbf{x}$ so that $x_1$ and $x_2$ are connected by an edge if and only if $d\paren{x_1,x_2}<\epsilon$ (this is the one-skeleton of the Vietoris-Rips complex on $\textbf{x}$). The following is a corollary of Kruskal's algorithm.
\begin{Lemma}
\label{MST_lemma}
\[p\paren{\textbf{x},\epsilon}=\beta_0\paren{G_{\textbf{x},\epsilon}}-1\]
where $\beta_0\paren{G_{\textbf{x},\epsilon}}$ is the number of connected components of $G_{\textbf{x},\epsilon}.$ 
\end{Lemma}

\subsection{Proof of the Upper Bound in Theorem~\ref{thm_mst}}
\label{sec_MST_upper}

\begin{Lemma}
\label{edge_counting_lemma}
Let $X$ be a metric space and suppose that there are positive real numbers $D$ and $d$ so that
\begin{equation}
\label{edge_counting_1}
M_\delta\paren{X}\leq D \, \delta^{-d}
\end{equation}
for all $\delta>0,$ where $M_\delta(X)$ is be the maximal number of disjoint open
balls of radius $\delta$ centered at points of $X$ (as defined in the previous section). Then
\[p\paren{\textbf{x},\delta}<D 2^{-d} \, \delta^{-d}\]
for all finite subsets $\textbf{x}$ of $X$ and all $\delta>0.$ 
\end{Lemma}
\begin{proof}
Let $\textbf{x} \subset X$ and $\delta>0.$ Also, let $\textbf{y}$ consist of the centers of a maximal set of disjoint balls of radius $\delta/2$ centered at points of $\textbf{x}.$ The maximality of $\textbf{y}$ implies that for every $x\in\textbf{x}$ there exists a $y\in\textbf{y}$ so that $d\paren{x,y}<\delta.$ In particular, every connected component of $G_{\textbf{x},\delta}$ has a vertex that is an element of $\textbf{y}.$ Therefore,

\begin{align*}
p\paren{\textbf{x},\delta}=&\;\; \beta_0\paren{G_{\textbf{x},\delta}}-1 &&\text{by Lemma~\ref{MST_lemma}}\\
\leq & \;\; \abs{\textbf{y}}-1\\
\leq & \;\; D \paren{\delta/2}^{-d} &&\text{by Eqn.~\ref{edge_counting_1}}\\
=& \;\; 2^{-d} D \delta^{-d}\,.
\end{align*}
\end{proof}

The previous lemma controls of the number of MST edges of length greater than $\epsilon.$ We can use this to prove an upper bound for $E_0^\alpha\paren{\textbf{x}_n}$ via the following lemma of Cohen-Steiner et al.~\cite{2010cohensteiner}.
\begin{Lemma}
\label{integral_lemma}
Let $J\subset \R^+$ be a bounded set of positive real numbers and let 
\[J_{\epsilon}=\set{j\in J:j>\epsilon}\,.\]
If 
\[\abs{J_{\epsilon}} \leq f\paren{\epsilon}< \infty\]
for all $\epsilon>0$ then
\[\sum_{j\in J_\epsilon}j^\alpha \leq \epsilon^\alpha f\paren{\epsilon}+\alpha \int_{\delta=\epsilon}^{\max J} f\paren{\delta}\delta^{\alpha-1}\;d\delta\]
for all $\alpha>0.$ Furthermore, if $\abs{J}\leq f\paren{0}<\infty$ then 
\begin{equation}
\label{eqn_integral_lemma}
\sum_{j\in J}j^\alpha \leq \alpha \int_{\delta=0}^{\max J} f\paren{\delta}\delta^{\alpha-1}\;d\delta\,.
\end{equation}

\end{Lemma}
\begin{proof}
For completeness, we reproduce the proof in~\cite{2010cohensteiner}. $\sum_{j\in J_\epsilon}j^\alpha$ can be expressed as an integral involving the distributional derivative of $\abs{J_{\epsilon}}.$ Applying integration by parts yields:
\begin{align*}
\sum_{j\in J_\epsilon}j^\alpha=&\;\; \int_{\delta=\epsilon}^{\infty} - \frac{\partial \abs{J_{\delta}}}{\partial \delta}\delta^\alpha\; d\delta\\
=&\;\; \Big[-\abs{J_{\delta}}\delta^\alpha\Big]_{\delta=\epsilon}^{\infty}+\alpha \int_{\delta=\epsilon}^{\infty}\abs{J_{\delta}}\delta^{\alpha-1}\; d\delta\\
=&\;\; \epsilon^{\alpha}\abs{J_{\epsilon}} +\alpha \int_{\delta=\epsilon}^{\sup J}\abs{J_{\delta}}\delta^{\alpha-1}\; d\delta\\
\leq & \;\; \epsilon^\alpha f\paren{\epsilon}+\alpha \int_{\delta=\epsilon}^{\sup J} f\paren{\delta}\delta^{\alpha-1} \; d\delta\,.
\end{align*}
\end{proof}

Combining the previous two lemmas gives an an extremal upper bound for $E_\alpha^0\paren{\textbf{x}_n}$ that, when combined with Lemma~\ref{lemma_ballcount}, implies the upper bound for Theorem~\ref{thm_mst}.

\begin{Proposition}
\label{prop_upper_extremal_MST}
Let $X$ be a metric space and suppose that there are positive real numbers $D$ and $d$ so that
\[M_\delta\paren{X}\leq D\,\delta^{-d}\]
for all $\delta>0.$ If $0<\alpha<d,$ then there exists a $D_\alpha>0$ so that 
\[E^0_\alpha\paren{\textbf{x}_n} \leq D_{\alpha} \, n^{\frac{d-\alpha}{d}}\]
for all $n$ and all collections $\textbf{x}_n$ of $n$ points in $X.$ Furthermore, there exists a $D_d>0$ so that 
\[E^0_d\paren{\textbf{x}_n} \leq D_d\,\log\paren{n}\]
for all $n$ and all collections $\textbf{x}_n$ of $n$ points in $X.$
\end{Proposition}
\begin{proof}
Rescale $X$ if necessary so that its diameter is less than $1,$ and let
\[\kappa=\frac{1}{2}\paren{\frac{D}{n-1}}^{1/d}\,.\]

The previous lemma implies that
\[p\paren{\set{\textbf{x}_n},\epsilon}\leq 2^{-d}D\epsilon^{-d}\,.\]
Furthermore, 
\[p\paren{\set{\textbf{x}_n},\epsilon}\leq n-1\]
because a minimum spanning tree on $n$ points has $n-1$ edges. Combining these yields that  $p\paren{\set{\textbf{x}_n},\epsilon}\leq f\paren{\epsilon}$ where
\begin{equation}
\label{eqn_puem_1}
f\paren{\epsilon}= \min\paren{n-1,2^{-d} D \epsilon^{-d}}=
\begin{cases}
n-1 & \epsilon \leq \kappa \\
2^{-d} D \epsilon ^{-d} & \epsilon \geq \kappa \\
\end{cases}\,.
\end{equation}

We have that 
Applying Lemma~\ref{integral_lemma} to the set of edge lengths of the minimum spanning tree on $\textbf{x}_n$ yields
\begin{align*}
E^0_\alpha\paren{\textbf{x}_n}=&\;\; \sum_{e\in T\paren{\textbf{x}_n}} \abs{e}^{\alpha}\\
\leq & \;\;\alpha \int_{\delta=0}^{1} f\paren{\delta}\delta^{\alpha-1}\;d\delta &&\text{by Eqn.~\ref{eqn_integral_lemma}}\\
=& \;\; \paren{n-1} \int_{\delta=0}^{\kappa}\alpha \delta^{\alpha-1}\;d\delta+\alpha 2^{-d} D \int_{\delta=\kappa}^{1}\delta^{\alpha-d-1}\;d\delta&&\text{by Eqn.~\ref{eqn_puem_1}}\\
=&\;\; \paren{n-1}\brac{\delta^\alpha}_{\delta=0}^{\kappa}-\frac{\alpha}{d-\alpha} 2^{-d} D \brac{\delta^{\alpha-d}}_{\delta=\kappa}^{1}\\
=&\;\; \paren{n-1}\kappa^\alpha+\frac{\alpha}{d-\alpha} 2^{-d} D\paren{\kappa^{\alpha-d}-1}\\
=&\;\; 2^{-\alpha}D^{\frac{\alpha}{d}}\paren{1+D \frac{\alpha}{d-\alpha}}\paren{n-1}^{\frac{d-\alpha}{d}}-\frac{\alpha}{d-\alpha}2^{-d}D\\
\leq & \;\; D_{\alpha} n^\frac{d-\alpha}{d}\,,
\end{align*}
where
\[D_{\alpha}=2^{-\alpha}D^{\frac{\alpha}{d}}\paren{1+D \frac{\alpha}{d-\alpha}}\,.\]

The result for $\alpha=d$ follows from a similar computation. 
\end{proof}

We now prove the upper bound in Theorem~\ref{thm_mst}.
\begin{proof}[Proof of the Upper Bound in Theorem~\ref{thm_mst}]
Let $\mu$ be a $d$-Ahlfors regular measure and $0<\alpha<d,$ and let $X$ be the support of $\mu.$ By Lemma~\ref{lemma_ballcount} there is a $c>0$ so that
\[ M_{\delta}\paren{X}\leq c \, \delta^{-d}\]
for all $\delta>0.$ Therefore, by the previous lemma there exists a $D_\alpha>0$ so that
\[E^0_\alpha\paren{\textbf{x}_n} \leq D_{\alpha} \, n^{\frac{d-\alpha}{d}}\]
for all collections of $n$ points $\textbf{x}_n\subset X.$
\end{proof}

Proposition~\ref{prop_upper_extremal_MST} also implies  Proposition~\ref{prop:box}, that the $\PH_0$ dimension of a measure is bounded above by the upper box dimension (Definition~\ref{defn:upperbox}) of its support, even if it is not Ahlfors regular. 

\begin{proof}[Proof of Proposition~\ref{prop:box}]
Let $\mu$ be a measure on a metric space $X,$ $d=\text{dim}_{\text{box}}\paren{X}$ be the upper box dimension of $X,$ and $\alpha<d<d_0.$ By Definition~\ref{defn:upperbox} there is a $D>0$ so that 
\[M_\delta\paren{X}\leq D\,\delta^{-d_0}\]
for all sufficiently small $\delta.$ Therefore, by Proposition~\ref{prop_upper_extremal_MST}, there is a $D_\alpha>0$ so that 
\[E^0_\alpha\paren{\textbf{x}_n} \leq D_{\alpha} \, n^{\frac{d_0-\alpha}{d_0}}\]
for all sets of $n$ points $\textbf{x}_n\subseteq X.$

Then
\[\beta\coloneqq \limsup_{n\rightarrow\infty}\frac{\log\paren{\mathbb{E}\paren{E_{\alpha}^{i}\paren{\textbf{x}_n}}}}{\log\paren{n}}\leq  \frac{\log\paren{ n^{\frac{d_0-\alpha}{d_0}}}}{\log\paren{n}}=\frac{d_0-\alpha}{d_0}\]
and, recalling Definition~\ref{defn_phdim},
\[\text{dim}_{\PH_i^\alpha}\paren{\mu}=\frac{\alpha}{1-\beta}\leq \frac{\alpha}{1-\paren{d_0-\alpha}/d_0}=d_0\,,\]
where we have used that $\frac{d_0-\alpha}{d_0}>0.$ This inequality holds for any $d_0>d,$ so  
\[\text{dim}_{\PH_i^\alpha}\paren{\mu}\leq d=\text{dim}_{\text{box}}\paren{X}\,,\]
as desired.
\end{proof}

\subsection{Proof of the Lower Bound in Theorem~\ref{thm_mst}}
\label{sec_MST_lower}
Our strategy to prove a lower bound for the asymptotics of $E^0_\alpha\paren{\textbf{x}_n}$ is to define random variables in terms of occupancy patterns of disjoint balls of radius $2r.$ This will in turn allow us to count the number of minimum spanning tree edges of lenght at least $r.$

Let $M$ be a metric space and let $\mu$ be a $d$-Ahlfors regular measure with support $M.$ If $B$ is a ball of radius $2 r$ centered at a point $y\in M$ and $\textbf{x}$ is a finite subset of $M,$ define
\[\omega\paren{B,\textbf{x}}=\Xi\paren{\textbf{x},B \setminus B_r\paren{y},\set{B_r\paren{y}}}\,.\]
That is, $\omega\paren{B,\textbf{x}}=1$ if  $\textbf{x}$ intersects $B_r\paren{y}$ but not the annulus centered at $y$ with radii $r$ and $2r.$

\begin{figure}
\centering
\includegraphics[width=.7\textwidth]{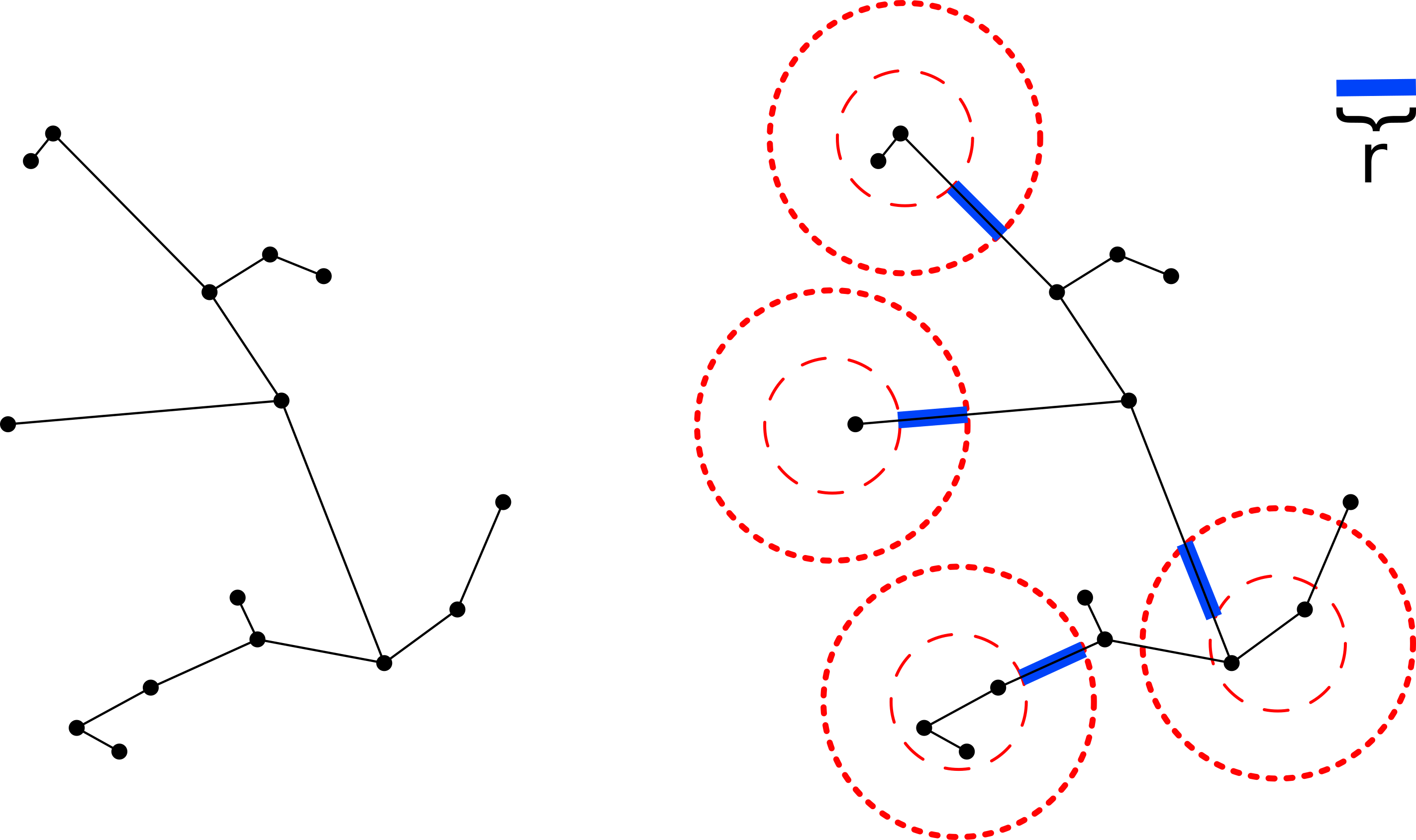}
\caption{The red balls on the right all satisfy $\omega\paren{B,\textbf{x}}=1,$ which guarantees that the minimum spanning tree on the left has at least three edges whose length exceeds $r.$}
\label{fig_mst_1}
\end{figure}

\begin{Lemma}
\label{lemma_ball_MST}
Let $\mathcal{B}$ be a set of disjoint balls of radius $2 r$ centered at points of $M,$ and let $\textbf{x}$ be a finite subset of $M.$ Then
\[p\paren{\textbf{x},r} \geq \sum_{B\in\mathcal{B}} \omega\paren{B,\textbf{x}}-1\,.\]
\end{Lemma}
\begin{proof}
This is an immediate consequence of Lemma~\ref{MST_lemma}. See Figure~\ref{fig_mst_1}. 
\end{proof}

Next, we take $\mathcal{B}$ to be a maximal set of disjoint balls of radius $2n^{-1/d}$ and use Lemma~\ref{lemma_occupancy_2} to provide a lower bound for $p\paren{\textbf{x}_n,n^{-1/d}}.$

\begin{Lemma}
\label{lemma_dominated_mst}
There is a positive real number $\gamma>0$ so that 
\[p\paren{\textbf{x}_n,n^{-1/d}} \geq \gamma n\]
with high probability as $n\rightarrow\infty.$ 
\end{Lemma}
\begin{proof}
Fix $n\in\N$ and let $\epsilon=n^{-1/d}.$ Let $B^n_1,\ldots, B^n_{s_n}$ be a maximal collection of disjoint balls of radius $2\epsilon$ centered at points of $X,$ and let $y^n_j$ be the center of $B^n_j$ for $j=1,\ldots,s_n.$

Set
\[p=\frac{1}{c}\quad\text{and}\quad q=2^d c - \frac{1}{c}\,.\]
where $c$ is the constant appearing in the definition of Ahlfors regularity (Definition~\ref{defn:ahlfors}). By that definition,
\[\mu\paren{B_{\epsilon}\paren{y^n_j}}\geq p \epsilon^{d}=\frac{p}{n}\]
and
\[\mu\paren{B^n_j\setminus {B_{\epsilon}\paren{y^n_j}}}\leq c\paren{2\epsilon}^d-\frac{1}{c}\epsilon^d=\frac{q}{n}\,.\]
Also, Lemma~\ref{lemma_ballcount} implies that
\[s_n \geq \frac{1}{c} 2^{-d} {2\epsilon}^{-d}=\frac{1}{c} 2^{-2d}n\,.\]

Therefore, the occupancy indicators $\omega\paren{B^n_1,\textbf{x}_n} ,\ldots, \omega\paren{B^n_{s_n},\textbf{x}_n}$ satisfy the hypotheses of Lemma~\ref{lemma_occupancy_2} and

\[ \lim_{n\rightarrow\infty }\frac{1}{n}\sum_{j=1}^{s_n} \omega\paren{B^n_j,\textbf{x}_n} \geq \gamma\]
with high probability as $n\rightarrow\infty.$ Combining this with Lemma~\ref{lemma_ball_MST} gives the desired result.

\end{proof}

The lower bound in Theorem~\ref{thm_mst} follows quickly.
\begin{Proposition}
\label{prop_lower_MST}
Let $\mu$ be a $d$-Ahlfors regular measure on a metric space $M.$ If $\set{x_j}_{j\in\N}$ are i.i.d.~samples from $\mu,$ and $\gamma$ is as given in the previous lemma, then
\[\lim_{n\rightarrow\infty} n^{-\frac{d-\alpha}{d}} E^0_\alpha\paren{\textbf{x}_n}\geq  \gamma \]
with high probability.
\end{Proposition}

\begin{proof}
We have that
\begin{align*}
\lim_{n\rightarrow\infty}  n^{-\frac{d-\alpha}{d}} E_\alpha^0\paren{\textbf{x}_n}\geq& \;\; \lim_{n\rightarrow\infty} n^{-\frac{d-\alpha}{d}} \paren{n^{-1/d}}^\alpha p\paren{\textbf{x}_n,n^{-1/d}}\\
\geq & \;\; \lim_{n\rightarrow\infty} n^{-\frac{d-\alpha}{d}} n^{-\alpha/d} \paren{\gamma n} &&\text{by Lemma~\ref{lemma_dominated_mst}}\\
= & \;\; \gamma 
\end{align*}
with high probability as $n\rightarrow \infty.$ 
\end{proof}

\section{Persistent Homology}

We provide a brief introduction to the persistent homology~\cite{2002edelsbrunner} of a filtration, loosely following~\cite{2018schweinhart}. For a more in-depth survey refer to, e.g.,~\cite{2009carlsson,2016chazal,2008edelsbrunner,2013edelsbrunner,2008ghrist}. A \textbf{filtration} of topological spaces is a family $\set{X_\epsilon}_{\epsilon\in I}$ of topological spaces indexed by an ordered set $I$, with inclusion maps $X_{\epsilon_1}\hookrightarrow X_{\epsilon_2}$ for all pairs of indices $\epsilon_1<\epsilon_2.$ For example, if $X$ is a subset of a metric space $M$, the $\epsilon$--neighborhood filtration of $X,$ $\set{X_\epsilon}_{\epsilon\geq 0},$ is the family of $\epsilon$-neighborhoods of $X,$ where
\[X_{\epsilon}=\set{m\in M:d\paren{m,X}<\epsilon}\,,\]
together with inclusion maps $X_{\epsilon_1}\hookrightarrow X_{\epsilon_2}$ for $\epsilon_1<\epsilon_2.$ See Figure~\ref{fig:epsilon_neighborhood}. If $X$ is a subset of Euclidean space, this construction is homotopy equivalent to the \v{C}ech complex of $X.$ The \v{C}ech complex of a subset $X$ of a metric space is the simplicial complex defined by 
\[\paren{x_1,\ldots,x_n}\in C\paren{X,\epsilon}\quad\text{if}\quad \cap_{j=1}^n B_{\epsilon}\paren{x_j}\neq \varnothing\,.\]
Note that the \v{C}ech complex depends on the ambient metric space. For example if $p_1,p_2,p_3$ are the vertices of an acute triangle  $\mathbb{R}^m$ and the ambient space is $\mathbb{R}^m$ then the $2$-simplex $\paren{p_1,p_2,p_3}$ will enter  $C\paren{X,\epsilon}$ when $\epsilon$ equals the circumradius of the triangle. If the ambient space is $\set{p_1,p_2,p_3},$ the simplex $\paren{p_1,p_2,p_3}$ will enter the complex when $\epsilon$ equals the maximum pairwise distance between the three points. 

In Euclidean space, the Alpha complex of a finite set $\textbf{x}$ is filtration on the Deluanay triangulation on $\textbf{x}.$ We do not define the Alpha complex here; see~\cite{2002edelsbrunner} for a definition.

 Another common construction is the Vietoris--Rips complex: if $Y$ is a metric space, let $V\paren{Y,\epsilon}$ be the simplicial complex defined by
\[\paren{y_1,\ldots,y_n}\in V\paren{Y,\epsilon}\quad\text{if}\quad d\paren{y_j,y_k}<\epsilon\text{ for }j,k=1,\ldots,n\,.\]
The family $\set{V\paren{Y,\epsilon}}_{\epsilon>0}$ together with inclusion maps for $\epsilon_1<\epsilon_2$ is a filtration indexed by the positive real numbers. As noted earlier, all of our results apply to both the \v{C}ech and Vietoris--Rips complexes except for Corollaries~\ref{corollary_R2} and~\ref{corollary_rips}, though the constants may differ. We will suppress the dependence of persistent homology on the underlying filtration, unless otherwise noted.

The \textbf{persistent homology module} of a filtration is the product $\prod_{\epsilon\in I} H_i\paren{X_{\epsilon}},$ together with the homomorphisms $j_{\epsilon_0,\epsilon_1}:H_i\paren{X_{\epsilon_0}}\rightarrow H_i\paren{X_{\epsilon_1}}$ for $\epsilon_0<\epsilon_1,$ where $H_i\paren{X_\epsilon}$ denotes the reduced homology of $X_\epsilon$ with coefficients in a field. If the rank of $i_{\epsilon_0,\epsilon_1}$ is finite for all $\epsilon_0<\epsilon_1,$ --- a hypothesis satisfied by all filtrations considered in this paper~\cite{2014chazal,2016chazal} --- the persistent homology module decomposes canonically into a set of interval modules~\cite{2009chazal,2005zomorodian}. We denote the colletion of these intervals as  $\PH_i\paren{X};$  each interval $\paren{b,d}\in\PH_i\paren{X}$ corresponds to a homology generator that is ``born'' at $\epsilon=b$ and ``dies'' at $\epsilon=d.$

If $\textbf{x}$ is a finite metric space and persistent homology is taken with respect to the Vietoris--Rips complex, Kruskal's algorithm implies that there is a length-preserving bijection between intervals of $\PH_0\paren{\textbf{x}}$ and the edges of the minimum spanning tree on $\textbf{x}.$ The same is true if persistent homology is taken with respect to the \v{C}ech complex if the ambient space is $\mathbb{R}^m,$ except that an interval is matched with an edge of twice its length.

\subsection{Properties of Persistent Homology}
\label{sec_properties}
Let $X$ be a metric space. For each $\epsilon>0,$ let $\PH_i^\epsilon\paren{X}$ denote the set of intervals of $\PH_i\paren{X}$ of length greater than $\epsilon$: 
\[\PH_i^\epsilon\paren{X}=\set{I\in\PH_i\paren{X}:\abs{I}>\epsilon}\,.\]
Also, define
\begin{equation}
\label{eqn:pi}
p_i\paren{X,\epsilon}=\abs{\PH_i^\epsilon\paren{X}}\,.
\end{equation}

If $X,Y\subset X$, let $d_H\paren{X,Y}$ denote the Hausdorff distance between $X$ and $Y$:
\[d_H\paren{X,Y}=\inf\set{\epsilon\geq 0 :  Y\subseteq X_\epsilon \quad \text{and} \quad X\subset Y_\epsilon}\,.\]
Also, let $d\paren{X,Y}$ be the infimal distance between pairs of points, one in each set:
\[d\paren{X,Y}=\inf_{x\in X, y\in Y}d\paren{x,y}\,.\]

We use the following properties of persistent homology in our proofs:

\begin{enumerate}
\item \textbf{Stability}: If $d_H\paren{X,Y}<\epsilon,$ there is an injection
\[\eta:\PH_i^{2\epsilon}\paren{X}\rightarrow \PH_i\paren{Y}\]
so that if $\eta\left(\left(b_0,d_0\right)\right)=\left(b_1,d_1\right)$ then
\[\max\paren{\abs{b_0-b_1},\abs{d_0-d_1}}<\epsilon\,.\]
 In particular, 
\[p_i\paren{X,2\epsilon+\delta}\leq p_i\paren{Y,\delta}\]
for all $\delta\geq 0.$ ~\cite{2014chazal,2007cohensteiner}
\item \textbf{Additivity for well-separated sets}: If $X_1,\ldots,X_n\subset M$ and

\[d\paren{X_j,X_k} > \text{max}\paren{\text{diam}\,X_j,\text{diam}\,X_k}\paren{1-\delta_{j,k}}\quad \forall j,k\]
then 
\[p_i\paren{\cup_j X_j,\epsilon}\geq \sum_j p_i\paren{X_j,\epsilon}\,.\]
\item \textbf{Translation invariance}: $\PH_i\paren{X}=\PH_i\paren{X+t}$ for all $t\in \R^m.$ 
\item \textbf{Scaling}: For all $\rho>0,$ 
\[\PH_i\paren{\rho\,X}=\set{\paren{\rho b,\rho d}: \paren{b,d}\in \PH_i\paren{X}}\,.\]
\end{enumerate}

We use property (1) in our proofs of both the upper and lower bounds in Theorems~\ref{thm_extremal} and~\ref{thm_probabalistic}, and property (2) for our proof of the lower bound. For these results, we also require a non-triviality property (as in Definition~\ref{defn_nontrivial}) and an upper bound for the number of $i$-dimensional persistent homology intervals of a set of $n$ points.

\subsection{A Lemma}
\label{sec_truncated}
If $X$ is a metric space, let $F_\alpha^i\paren{X,\epsilon}$ denote the $\alpha$-weighted sum of the persistent homology intervals of $X$ of length greater than $\epsilon:$

\[F_\alpha^i\paren{X,\epsilon}=\sum_{I\in\PH_i^\epsilon\paren{X}}\abs{I}^\alpha\,.\]
We will use the following lemma in the next section.
\begin{Lemma}
\label{lemma_HB1}
If $d_H\paren{X,Y}<\epsilon/4$
then 
\[F_\alpha^i \paren{X,\epsilon}< 2^\alpha F_\alpha^i \paren{Y,\epsilon/2}\,.\]
\end{Lemma}
\begin{proof}
By stability, there is an injection
\[\eta:\PH_i^\epsilon\paren{X}\rightarrow \PH_i^{\epsilon/2}\paren{Y}\]
so that for all  $I\in \PH_i^\epsilon\paren{X}$ 
\[\abs{I}<\abs{\eta\paren{I}}+\epsilon/2\leq 2\abs{\eta\paren{I}}\,.\]
It follows that
\begin{align*}
F_\alpha^i\paren{X,\epsilon}=&\;\;\sum_{I \in \PH_i^\epsilon\paren{X}} \abs{I}^\alpha\\
< & \;\; \sum_{I \in \PH_i^\epsilon\paren{X}} 2^\alpha \abs{\eta\paren{I}}^\alpha\\
\leq &\;\; 2^\alpha \sum_{J \in \PH_i^{\epsilon/2}\paren{Y}}\abs{J}^\alpha\\
=&\;\; 2^\alpha F_\alpha^i \paren{Y,\epsilon/2}\,.
\end{align*}
\end{proof}

\section{Upper Bounds}

In this section, we prove the upper bounds for in Theorems~\ref{thm_extremal} and~\ref{thm_probabalistic}, where the former assumes extremal hypotheses on the number of intervals of $\PH_i\paren{\textbf{x}_n}$ and the latter assumes that the expectation and variance of the number of intervals behave nicely.

In the extremal case, we closely follow the approach of Section~\ref{sec_MST_upper}. Instead of using that a minimum spanning tree on $n$ points has $n-1$ edges, we assume that the number of intervals of $\PH_i\paren{\textbf{x}_n}$ is bounded above by $D n^a$ for some constant $D.$ We control $p_i\paren{\textbf{x}_n,\delta}$ by approximating $\textbf{x}_n$ by with the centers of a maximal set of disjoint balls of radius $\delta/2$ and applying stability (Lemma~\ref{interval_counting_lemma}). The asymptotics we found in the ball-counting lemma above (Lemma~\ref{lemma_ballcount}) then imply that $p_i\paren{\textbf{x}_n,\delta}\leq C\delta^{-ad}$ for some constant $C>0$ and all $\textbf{x}_n\subset X.$ We convert this into a bound on $E_i^\alpha\paren{\textbf{x}_n}$ by integrating (using Lemma~\ref{integral_lemma}) and again using our assumption on the number of intervals, yielding the upper bound in Theorem~\ref{thm_extremal} (Proposition~\ref{prop_upper_extremal}). 

While the extremal hypotheses  allow us to prove the desired upper bound in Corollary~\ref{corollary_R2}, they are inadequate to show a similar upper bound for subsets of higher dimensional Euclidean space. Instead, we show that we can obtain a better upper bound on the scaling of $E_i^\alpha\paren{\textbf{x}_n}$ by assuming that  
\begin{equation}
\label{eqn_expecation}
\mathbb{E}\paren{\abs{\PH_i\paren{\textbf{x}_n}}}=O\paren{n}
\end{equation}
and
\[\textit{Var}\paren{\abs{\PH_i\paren{\textbf{x}_n}}}/n^2\rightarrow 0\,,\]
which are quantities that can be estimated in practice during the course of data analysis. We use Equation~\ref{eqn_expecation} to control the persistent homology of the support of the measure ($X$)  by approximating $X$ by a point sample in Lemma~\ref{lemma_hausdorff_probability} and applying stability in Lemma~\ref{lemma_bootstrap}, resulting in Proposition~\ref{prop_selfHomology} on the asymptotics of truncated $\alpha$-weighted sums for $\PH_i\paren{X}.$ With that, we write  $\PH_i\paren{\textbf{x}_n}$ a sum of two terms, one which approximates $\PH_i\paren{X}$ and one which corresponds to ``$d$-dimensional noise'' at a certain scale. Controlling both terms gives a proof of the upper bound in Theorem~\ref{thm_probabalistic} (Proposition~\ref{prop_upper_prob}).

\subsection{Extremal Hypotheses}
\label{sec_extremal}

First, we prove the upper bound in Theorem~\ref{thm_extremal}, which implies the upper bound for our result on measures supported on a subset of $\mathbb{R}^2$ (Corollary~\ref{corollary_R2}). The next lemma uses bottleneck stability to convert an extremal bound on the number of persistent homology intervals of a set of $n$ points in a metric space $X$ into a bound on the number of intervals of length greater than $\epsilon$ for any $Y\subseteq X.$ It is the analogue of Lemma~\ref{edge_counting_lemma} for higher dimensional persistent homology.

\begin{Lemma}[Interval Counting Lemma]
\label{interval_counting_lemma}
If $X$ is a bounded metric space so that 
\[\abs{\PH_i\paren{x_1,\ldots,x_n}}<D n^a\,.\]
for some positive real numbers $a$ and $D$ and all finite subsets $\set{x_1,\ldots,x_n}$ of $X,$ then
\[p_i\paren{Y,\epsilon}<D' \epsilon^{-a d}\]
for some $D'>0,$ all $Y\subseteq X,$ and all $\epsilon>0.$
\end{Lemma}
\begin{proof}
Recall from Equation~\ref{eqn:pi} that $p_i\paren{Y,\epsilon}$ is the number of intervals of $\PH_i\paren{Y}$ of length greater than $\epsilon.$

Let $Y\subseteq X,$ $\epsilon>0,$ and $\set{y_j}$ be the centers of a maximal set of disjoint balls of radius $\epsilon/4$ centered at points of $Y.$ The balls of radius $\epsilon/2$ centered at the points $\set{y_j}$ cover $Y$ so
\[d_H\paren{\set{y_i},Y}<\epsilon/2\]

It follows that
\begin{align*}
p_i\paren{Y,\epsilon}\leq &\;\;p_i\paren{\set{y_i},0} &&\text{by stability} \\
\leq & \;\; D \abs{y_i}^a &&\text{by hypothesis}\\
\leq & \;\; D M_{\epsilon/4}\paren{X}^a \\
\leq &  \;\; D c^a 4^{-a/d} \epsilon^{-ad} &&\text{by Lemma~\ref{lemma_ballcount}}\\
\end{align*}
as desired.
\end{proof}

The next proposition is the analogue of Proposition~\ref{prop_upper_extremal_MST}. The proof is nearly identical, and we do not repeat it here. 

\begin{Proposition}
\label{prop_upper_extremal}
If $X$ satisfies the hypotheses of the previous lemma and $\alpha < ad,$ then there exists a $D_\alpha>0$ so that
\[E_\alpha^i\paren{x_1,\ldots,x_n} \leq D_\alpha n^{\frac{a d-\alpha}{d}}\]
for all finite subsets $\set{x_1,\ldots,x_n}\subset X$ and all $n\in \N.$ Furthermore there exists a $D_d>0$ so that 
\[E_{ad}^{i}\paren{x_1,\ldots,x_n}\leq D_d \log\paren{n}\]
for all finite subsets $\set{x_1,\ldots,x_n}\subset X$ and all $n\in \N.$

\end{Proposition}

We now prove the upper bound in Theorem~\ref{thm_extremal}.
\begin{proof}[Proof of the Upper Bound in Theorem~\ref{thm_extremal}]
Let $\mu$ be a $d$-Ahlfors regular measure and let $X$ be the support of $\mu.$   Assume that there are positive real numbers $D$ and $a$ so that 
\[\abs{\PH_i\paren{\textbf{x}_n}}<D n^a\]
for all finite subsets of $X,$ an let $0<\alpha<ad.$

By Lemma~\ref{lemma_ballcount} there is a $c>0$ so that
\[ M_{\delta}\paren{X}\leq c \, \delta^{-d}\]
for all $\delta>0.$ Therefore, by the previous lemma there exists a $D_\alpha>0$ so that
\[E^0_\alpha\paren{\textbf{x}_n} \leq D_{\alpha} \, n^{\frac{ad-\alpha}{d}}\]
for all collections of $n$ points $\textbf{x}_n\subset X.$
\end{proof}

\subsection{Probabilistic Hypotheses}
While the extremal hypotheses of the previous section allow us to prove the desired upper bound in Corollary~\ref{corollary_R2}, they are inadequate to show a similar upper bound for subsets of higher dimensional Euclidean space.  Here, we show that hypotheses on the the expectation and variance of the number of $\PH_i$ intervals of a set of $n$ points imply better asymptotic upper bounds (the upper bound in Theorem~\ref{thm_probabalistic}). The idea of the proof is to control the behavior of the support of the measure ($\PH_i\left(X\right)$) in terms of the persistent homology of point samples from $X.$ With that, we write  $\PH_i\paren{\textbf{x}_n}$ a sum of two terms, one which approximates $\PH_i\paren{X}$ and one which corresponds to ``$d$-dimensional noise'' at a certain scale.

 First, we require the following lemma, which follows from a standard argument using the union bound; see~\cite{2008niyogi} for a proof.
\begin{Lemma}
\label{lemma_union}
Let $\mu$ be a probability measure on a metric space $X,$ and $\set{B_j}_{j=1}^l\subset X$ be a collection of balls so that so that $\mu\paren{B_j}\geq a$ for all $j.$ Then 
\[\mathbb{P}\paren{\textbf{x}_n \cap B_j \neq \varnothing\quad\text{for}\quad j=1,\ldots,l}\geq 1- l e^{-a n}\,.\]
\end{Lemma}

Next, we apply the previous lemma to control the Hausdorff distance between $X$ and finite samples from an Ahlfors regular measure on $X.$

\begin{Lemma}
\label{lemma_hausdorff_probability} 
If $\mu$ is a $d$-Ahlfors regular measure with support $X$ then there exists a positive real number $A_0$ depending only on the constants $c$ and $d$ appearing in the definition of Ahlfors regularity so that
\begin{equation}
\label{eqn_lemma_hausdorff_probability}
\mathbb{P}\paren{d_H\paren{\set{\textbf{x}_n},X}<\epsilon}\geq 1- c \epsilon^{-d} e^{-A_0 \epsilon^{d}n}
\end{equation}
for all $\epsilon>0.$
\end{Lemma}
\begin{proof}
Let  $\textbf{y}=\set{y_1,\ldots,y_{M_{\epsilon/3}\paren{X}}}$ be the centers of a maximal set of disjoint balls of radius $\epsilon/3$ centered at points of $X.$ By the definition of Ahlfors regularity,
\[\mu\paren{B_{\epsilon/3}\paren{y}}\geq A_0 \epsilon^d\]
for all $y\in\textbf{y},$ where $A_0=3^{-d}/c.$ 

 The balls of radius $2\epsilon/3$ centered at the points of $\textbf{y}$ cover $X$ so
\[d_H\paren{\textbf{y},X}<2\epsilon/3\,.\]

Therefore, if $\set{\textbf{x}_n}\cap B_{\epsilon/3}\paren{y}\neq \varnothing$ for all $y\in\textbf{y}$ then
\[d_H\paren{\set{\textbf{x}_n},X}<\epsilon/3+2\epsilon/3=\epsilon\,.\]
It follows that
\begin{align*}
\mathbb{P}\paren{d_H\paren{\set{\textbf{x}_n},X}<\epsilon}\geq &\;\;\mathbb{P}\paren{\set{\textbf{x}_n}\cap B_{\epsilon/3}\paren{y}\neq \varnothing \quad \text{for all }y\in\textbf{y}}\\
\geq & \;\; 1-M_{\epsilon/3}\paren{X} e^{-  A_0 \epsilon^d n} &&\text{by Lemma~\ref{lemma_union}}\\
\geq & \;\; 1- c \epsilon^{-d}  e^{-A_0\epsilon^{d}n} &&\text{by Lemma~\ref{lemma_ballcount}}\,.\\
\end{align*}
\end{proof}

In the next lemma, we show that if the expected number of persistent homology intervals of $\textbf{x}_n$ is $O(n),$ then we can control the number of ``long'' persistent homology intervals of $X$ itself.

\begin{Lemma}
\label{lemma_bootstrap}
Let $X$ be a bounded metric space that admits a $d$-Ahlfors regular measure $\mu$ satisfying
\[\mathbb{E}\paren{\abs{\PH_i\paren{\textbf{x}_n}}}=O\paren{n}\,.\]
Then there are positive real numbers $A_1$ and $\epsilon_0$ so that
\[p_i\paren{X,\epsilon} \leq A_1 \epsilon^{-d} \log\paren{1/\epsilon}\]
for all $\epsilon<\epsilon_0.$ 
\end{Lemma}
\begin{proof}
By hypothesis, there are positive real numbers $D_1$ and $N_1$ so that 
\[\mathbb{E}\paren{\abs{\PH_i\paren{\textbf{x}_n}}}\leq n D_1/2  \]
for all $n>N_1.$ By Markov's inequality,
\begin{equation}
\label{eqn_bootstrap1}
\mathbb{P}\paren{\abs{\PH_i\paren{\textbf{x}_n}}\leq  n D_1 } \geq 1/2\,.
\end{equation}

Manipulating the inequality in Equation~\ref{eqn_lemma_hausdorff_probability} to solve for the number of points samples required to approximate $X$ within a distance of $\epsilon/2$ with probability exceeding $1/2$ gives that
\begin{equation}
\label{eqn_bootstrap2} 
\mathbb{P}\paren{d_H\paren{\set{x_1,\ldots,x_{m\paren{\epsilon}}},X}<\epsilon/2}\geq 1/2
\end{equation}
where
\begin{equation}
\label{eqn_bootstrap25}
m\paren{\epsilon} = \ceil{ \frac{2^d}{A_0} \epsilon^{-d}\log\paren{2^{d+1}c \epsilon^{-d}}}\,.
\end{equation}
 Note that $m\paren{\epsilon}$ is chosen to give distances of less than $\epsilon/2,$ rather than less than $\epsilon.$ Let $\epsilon$ be sufficiently small so that $m\paren{\epsilon}>N_1.$ The events in Equations~\ref{eqn_bootstrap1} and~\ref{eqn_bootstrap2} both occur with probability greater than $\frac{1}{2}$ so there exists at least one point set in the intersection. That is, there exists a finite point set  $x_1,\ldots,x_{m\paren{\epsilon}}$ of $X$ so that
\begin{equation}
\label{eqn_bootstrap3}
\abs{\PH_i\paren{x_1,\ldots,x_{m\paren{\epsilon}}}} \leq D_1 m\paren{\epsilon}
\end{equation}
and
\begin{equation}
\label{eqn_bootstrap4}
 d_H\paren{\set{x_1,\ldots,x_{m\paren{\epsilon}}},X}<\epsilon\,.
\end{equation}
Therefore,
\begin{align*}
p_i\paren{X,\epsilon} \leq &\;\; p_i\paren{\set{x_1,\ldots,x_{m\paren{\epsilon}}},0} &&\text{by stability and Eqn.~\ref{eqn_bootstrap4}} \\
\leq & \;\; D_1 m\paren{\epsilon}&&\text{by Eqn.~\ref{eqn_bootstrap3}} \\
= & \;\;  D_1 \ceil{ \frac{2^d}{A_0} \epsilon^{-d}\log\paren{2^{d+1} c \epsilon^{-d}}} &&\text{by Eqn.~\ref{eqn_bootstrap25}}\\
=& \;\; O\paren{\epsilon^{-d} \log\paren{1/\epsilon}}
\end{align*}
as $\epsilon\rightarrow 0.$
\end{proof}

Next, we use the previous lemma to control $F_\alpha^i(X,\epsilon),$ the truncated $\alpha$-weighted sum defined in Section~\ref{sec_truncated}:
\[F_\alpha^i\paren{X,\epsilon}=\sum_{I\in\PH_i^\epsilon\paren{X}}\abs{I}^\alpha\,.\]
where $\PH_i^\epsilon\paren{X}$ is the set of $\PH_i$ intervals of $X$ of length greater than $\epsilon.$

\begin{Proposition}
\label{prop_selfHomology}
If $X$ satisfies the hypotheses of the previous lemma and $0<\alpha<d,$ then there exist positive real numbers $A_2$ and $\epsilon_1$ so that 
\[F_\alpha^i\paren{X,\epsilon} \leq A_2 \epsilon^{\alpha - d}\log\paren{1/\epsilon}\]
for all $\epsilon<\epsilon_1.$
\end{Proposition}.
\begin{proof}
Without loss of generality, we may rescale $X$ so its diameter is less than one. By the previous lemma
\begin{equation}
\label{eqn_selfhomology_1}
p_i\paren{X,\epsilon}\leq f\paren{\epsilon}\coloneqq A_1 \paren{\epsilon}^{-d} \log\paren{\frac{1}{\epsilon}}
\end{equation}
for all $\epsilon<\epsilon_0.$
Applying Lemma~\ref{integral_lemma} yields
\begin{equation}
\label{eqn_selfhomology_2}
F_\alpha^i\paren{Y,\epsilon}\leq \epsilon^\alpha f\paren{\epsilon}+\alpha \int_{t=\epsilon}^{\epsilon_0} f\paren{t} t^{\alpha-1}\;d t + F_\alpha^i\paren{Y,\epsilon_0}\,.
\end{equation}

By Equation~\ref{eqn_selfhomology_1}, the first term of Equation~\ref{eqn_selfhomology_2} equals
\[A_1 \epsilon^{\alpha-d}\paren{\log\paren{1/\epsilon}}\,,\]
which has the desired asymptotics as $\epsilon\rightarrow 0.$ The second term equals
\begin{align*}
\alpha\int_{t=\epsilon}^{\epsilon_0} A_1 t^{\alpha-d-1}\log\paren{1/t}\;d t=&\\
&\;\; A_1 \brac{-\frac{1}{d-\alpha} t^{\alpha-d} \log\paren{1/t}-\frac{1}{\paren{d-\alpha}^2} t^{\alpha-d}}_{\epsilon}^{\epsilon_0}\\
&=\;\; A_1 \paren{\frac{1}{d-\alpha} \epsilon^{\alpha-d} \log\paren{1/\epsilon}+\frac{1}{\paren{d-\alpha}^2} \epsilon^{\alpha-d}}\\
& - \;\; A_1 \paren{\frac{1}{d-\alpha} \epsilon_0^{\alpha-d} \log\paren{1/\epsilon_0}+\frac{1}{\paren{d-\alpha}^2} \epsilon_0^{\alpha-d}}\\
&=\;\;O\paren{\epsilon^{\alpha - d}\log\paren{1/\epsilon}}\,.
\end{align*}
Therefore, $p_i(X,\epsilon)=O\paren{\epsilon^{\alpha - d}\log\paren{1/\epsilon}},$ as desired. 

\end{proof}

Finally, we can bootstrap the previous result to control $E_\alpha^i\paren{\textbf{x}_n}$ and prove the upper bound in Theorem~\ref{thm_probabalistic}. For clarity, we restate that upper bound as a proposition.

\begin{Proposition}
\label{prop_upper_prob}
Let $\mu$ be a $d$-Ahlfors regular measure on a bounded metric space. If
\[\mathbb{E}\paren{\abs{\PH_i\paren{\textbf{x}_n}}}=O\paren{n}\]
and
\[\textit{Var}\paren{\abs{\PH_i\paren{\textbf{x}_n}}}/n^2\rightarrow 0\,,\]

then there is a $\Lambda >0$ so that
\[E_\alpha^i\paren{\textbf{x}_n}\leq \Lambda  n^{\frac{d-\alpha}{d}}\log\paren{n}^{\frac{\alpha}{d}}\]
with high probability as $n\rightarrow\infty.$
\end{Proposition}

\begin{proof}
Let 
\begin{equation}
\label{eqn_pup0}
G_\alpha^i\paren{\textbf{x},\epsilon}=\sum_{I \in\PH_i\paren{\textbf{x}} \setminus \PH_i^\epsilon\paren{\textbf{x}}}\abs{I}^\alpha\,.
\end{equation}
Our  strategy is to write
\[E_\alpha^i\paren{\textbf{x}_n}=G_\alpha^i\paren{\textbf{x}_n,\epsilon}+F_\alpha^i\paren{\textbf{x}_n,\epsilon}\]
for a well-chosen $\epsilon.$ The former term can be interpreted as ``noise,'' and the latter approximates the persistent homology of the support of $\mu.$ 

Let $0<p<1,$ and let $D$ be a positive real number so that 
\[\mathbb{E}\paren{\abs{\PH_i\paren{\textbf{x}_n}}}\leq \paren{D/2} n\]
for all sufficiently large $n.$ By Chebyshev's inequality,
\begin{align*}
\mathbb{P}\paren{\abs{\PH_i\paren{\textbf{x}_n}}>D n} \leq & \\
& \;\; \mathbb{P}\paren{\Big | \abs{\PH_i\paren{\textbf{x}_n}}-\mathbb{E}\paren{\abs{\PH_i\paren{\textbf{x}_n}}}\Big |> Dn/2}\\
\leq & \;\; \text{Var}\paren{\abs{\PH_i\paren{\textbf{x}_n}}} \frac{4}{D^2 n^2}
\end{align*}
which converges to $0$ as $n\rightarrow \infty,$ by hypothesis. Therefore, there is a $M$ so that 

\begin{equation}
\label{eqn_pup1}
\mathbb{P}\paren{\abs{\PH_i\paren{\textbf{x}_n}}> D n}< p/2
\end{equation}
for all $n>M.$

Solving for $\epsilon$ in Equation~\ref{eqn_lemma_hausdorff_probability} gives that
\begin{equation}
\label{eqn_pup2}
\mathbb{P}\paren{d_H\paren{\set{\textbf{x}_n},X}>\epsilon\paren{n}/4}< p/2
\end{equation}
if 
\[\epsilon\paren{n}=4 A_0^{-1/d} n^{-1/d} W\paren{\frac{2 c A_0 n}{p}}^{1/d}\,,\]
where $W$ is the Lambert W function. $W\paren{m}\sim \log\paren{m}$ as $m\rightarrow \infty,$ and $W\paren{m}\leq \log\paren{m}$ for $m\geq e$~\cite{2008hoorfar}. Therefore, there are positive real numbers $A_3$ and $N_1\paren{p},$ where the former does not depend on $p$ but the latter does, so that
\begin{equation}
\label{eq:epn}
 \frac{A_3}{2} n^{-1/d} \log\paren{n}^{1/d} \leq \epsilon\paren{n} \leq A_3 n^{-1/d} \log\paren{n}^{1/d}
\end{equation}
for all $n>N_1\paren{p}.$ The right hand side goes to zero as $n\rightarrow\infty$ so we can choose $N_2\paren{p}>N_1\paren{p}$ to be sufficiently large so that $\epsilon\paren{n}<\epsilon_1$ for all $n>N_2\paren{p},$ where $\epsilon_1$ is given in Proposition~\ref{prop_selfHomology}. Let $n>N_2\paren{p}$

 and suppose that  $\textbf{x}_n$ satisfies 
\begin{equation}
\label{eqn_pup3}
\abs{\PH_i\paren{\textbf{x}_n}}< D n \qquad \text{and} \qquad d_H\paren{\textbf{x}_n,X}<\epsilon\paren{n}/4\,,
\end{equation}
an event which occurs with probability greater than $1-p$ by Equations~\ref{eqn_pup1} and~\ref{eqn_pup2}.

 Write
\begin{equation}
\label{eqn_pup4}
E_\alpha^{i}\paren{\textbf{x}_n}=F_\alpha^i\paren{\textbf{x}_n,\epsilon\paren{n}}+G_\alpha^i\paren{\textbf{x}_n,\epsilon\paren{n}}\,.
\end{equation}

We consider the two terms separately.

\begin{align*}
G_\alpha^i\paren{\textbf{x}_n,\epsilon\paren{n}}\leq & \;\; D \abs{\textbf{x}_n} \epsilon\paren{n}^\alpha &&\text{by Eqns.~\ref{eqn_pup0} and~\ref{eqn_pup3}}\\
\leq & \;\; 2^{\alpha} D A_3^\alpha  n^{\frac{d-\alpha}{d}} \log\paren{n}^{\alpha/d}\\
= &\;\; A_4 n^{\frac{d-\alpha}{d}} \log\paren{n}^{\frac{\alpha}{d}}\,,
\end{align*}
where $A_4=2^{\alpha} D A_3^\alpha$ is a positive constant that does not depend on $n$ or $p.$

To bound the second term in Equation~\ref{eqn_pup4}, we apply Lemma~\ref{lemma_HB1} to find
\begin{align*}
F_\alpha^i\paren{\textbf{x},\epsilon\paren{n}}\leq & \;\; 2^\alpha F_\alpha^i\paren{X,\frac{\epsilon\paren{n}}{2}}\\
\leq & \;\; A_2 \paren{\epsilon\paren{n}}^{\alpha - d}\log\paren{\frac{1}{\epsilon\paren{n}}} &&\text{by Prop.~\ref{prop_selfHomology}}\\
\leq & \;\; A_2 A_3 ^{\alpha-d} n^{\frac{d-\alpha}{d}} \log\paren{n}^{-\frac{d-\alpha}{d}}\log\paren{\frac{1}{2A_3} n^{1/d} \log\paren{n}^{-1/d}} &&\text{by Eqn.~\ref{eq:epn}}\\
=&\;\;  A_2 A_3^{\alpha-d} n^{\frac{d-\alpha}{d}}\log\paren{n}^{-\frac{d-\alpha}{d}}\paren{\frac{1}{d}\log\paren{n}-\log\paren{2 A_3 \log\paren{n}^{1/d}}}\\
\leq & \;\; \frac{1}{d} A_2 A_3^{\alpha-d} n^{\frac{d-\alpha}{d}}\log\paren{n}^{\frac{\alpha}{d}}\\
=&\;\; A_5 n^{\frac{d-\alpha}{d}}\log\paren{n}^{\frac{\alpha}{d}}\,,
\end{align*}
where $A_5=\frac{1}{d} A_2 A_3^{\alpha-d}$ is a positive constant that does not depend on $n$ or $p.$

In summary, if $\Lambda=A_4+A_5$ and $0<p<1,$ then there exists an $N_2(p)>0$ so that 

\[\mathbb{P}\paren{E_\alpha^i\paren{\textbf{x}_n}\leq \Lambda  n^{\frac{d-\alpha}{d}}\log\paren{n}^{\frac{\alpha}{d}}}>1-p\]
for all $n>N_2(p).$
\end{proof}

\section{The Lower Bound}
In this section, we prove the lower bound in Theorems~\ref{thm_extremal} and~\ref{thm_probabalistic}. While our proofs of the upper bounds work for Ahlfors regular measures on arbitrary bounded metric spaces, here we restrict our attention to Ahlfors regular measures on Euclidean space. This will allow us to use the structure of the cubical grid on $\R^m.$ 

We remind the reader of the occupancy indicators defined in Section~\ref{sec:occupancy}. If $\textbf{x}$ is point set, $A$ is a set, and $\mathcal{B}$ of a collection of sets then 
\[\Xi\paren{\textbf{x},A,\mathcal{B}}=\begin{cases}
 1 & \delta\paren{A,\textbf{x}}=0 \quad \quad \text{and}\quad \quad \delta\paren{B,\textbf{x}}=1\quad \forall\, B\in\mathcal{B}\\
 0 & \text{otherwise}
\end{cases}\,,\]
where for any set $C$
\[\delta\paren{C,\textbf{x}}=\begin{cases}
0 & C\cap \textbf{x}=\varnothing\\ 
1 & C\cap \textbf{x} \neq \varnothing\\
\end{cases}
\,.\]

To prove the lower bound, we modify the approach in our paper on extremal $\PH$-dimension~\cite{2018schweinhart} to work in a probabilistic context. The outline of the argument is similar to that in Section~\ref{sec_MST_lower}, but more care is required to construct occupancy indicators implying the existence of persistent homology intervals. We work on two different length-scales: we divide the ambient Euclidean space into cubes of width $n^{-1/d}$, and divide each of these cubes into $k^m$ sub-cubes of width $n^{-1/d}/k.$ Using the non-triviality constants defined in~\cite{2018schweinhart} (Definition~\ref{defn_nontrivial}), we show that if a cube contains sufficiently many sub-cubes that overlap with the support of the measure, then we can define an occupancy indicator guaranteeing the existence $\PH_i$ interval of a certain length. We count the number of cubes with sufficiently many occupied sub-cubes in Lemma~\ref{lemma_lower_1}, and apply Lemma~\ref{lemma_occupancy_2} to the corresponding occupancy indicators to give
\[p_i\paren{\textbf{x}_n, \epsilon_0 \, n^{-1/d}} \geq \Omega_1 n \]
with high probability as $n\rightarrow\infty$ for some $\Omega_1>0$ (Lemma~\ref{lemma_lower_2}). Summing over intervals of length greater than $\epsilon_0 \, n^{-1/d}$ proves the desired lower bound (Proposition~\ref{prop_lower}).

The proof is complicated, so we warm up with the special case of an $m$-Ahlfors regular measure on $\mathbb{R}^m.$ The approach is more straightforward, but contains some of the same elements. These arguments also appear in our unpublished manuscript~\cite{2018schweinhart_b}, which has largely been subsumed into the current work. First, we find an occupancy indicator defined in terms of sub-cubes of a larger cube that guarantees the existence of a persistent homology interval of a given length, regardless of what happens outside of the cube in Lemma~\ref{lemma_euc_1}. Then, we assemble a collection of these occupancy indicators to show the lower bound in Propostition~\ref{prop_AC}.

\subsection{The Absolutely Continuous Case}
Note that if $\mu$ is an $m$-Ahlfors regular measure on $\mathbb{R}^m,$ then $\mu$ is comparable to the Lebesgue measure on its support and is thus absolutely continuous with respect to it. 

\begin{Proposition}
\label{prop_AC}
Let $\mu$ be an $m$-Ahlfors regular measure on $\mathbb{R}^m.$ There exist a constant $\Psi>0$ so that
\[\lim_{n\rightarrow\infty} n^{-\frac{d-\alpha}{d}} E_\alpha^i\paren{x_1,\ldots,x_n}\geq  \Psi \]
with high probability.
\end{Proposition}
Before proving the proposition, we state and prove a preliminary lemma to find an occupancy indicator guaranteeing the existence of a persistent homology interval of a given length, regardless of what happens outside of the cube (in a sense made precise in Equation~\ref{eq_euc_1} below). The idea is to take $\mathcal{B}$ to be the set of sub-cubes that intersect an $i$-dimensional sphere sphere that is separated from the boundary of a cube, as shown in Figure~\ref{fig:circleFig}. Let 
\[J_{\epsilon,i}\paren{X}=\set{\paren{b,d}\in\PH_i\paren{X}:d\leq \epsilon}\,.\]

\begin{Lemma}
\label{lemma_euc_1}
Let $0\leq i < m,$  and $0<b<d<1/8.$  There exists a $\lambda_0>0$ so that if $C\subset\R^m$ is an $m$-dimensional cube of width $R$ and $\lambda>\lambda_0,$  there exists a collection $\mathcal{B}$ of disjoint, congruent cubes of width $R\lambda^{-\frac{1}{m}}$ so that if $A=C\setminus \cup_{B\in \mathcal{B}}B$ and $\Xi\paren{\textbf{x},A,\mathcal{B}} =1$ then $\PH_i\paren{\textbf{x}\cap C}$ contains an interval $\paren{\hat{b},\hat{d}}$ with
\[0<\hat{b}<R b<R d<\hat{d}\,.\]
Furthermore,
\begin{equation}
\label{eq_euc_1}
J_{\hat{d},i}\paren{\textbf{x}}=J_{\hat{d},i}\paren{\textbf{x}\cap C}\cup J_{\hat{d},i}\paren{\textbf{x}\setminus\paren{\textbf{x}\cap C}}\,.
\end{equation}

\end{Lemma}
\begin{proof}

\begin{figure}
\centering
\includegraphics[width=.45\textwidth]{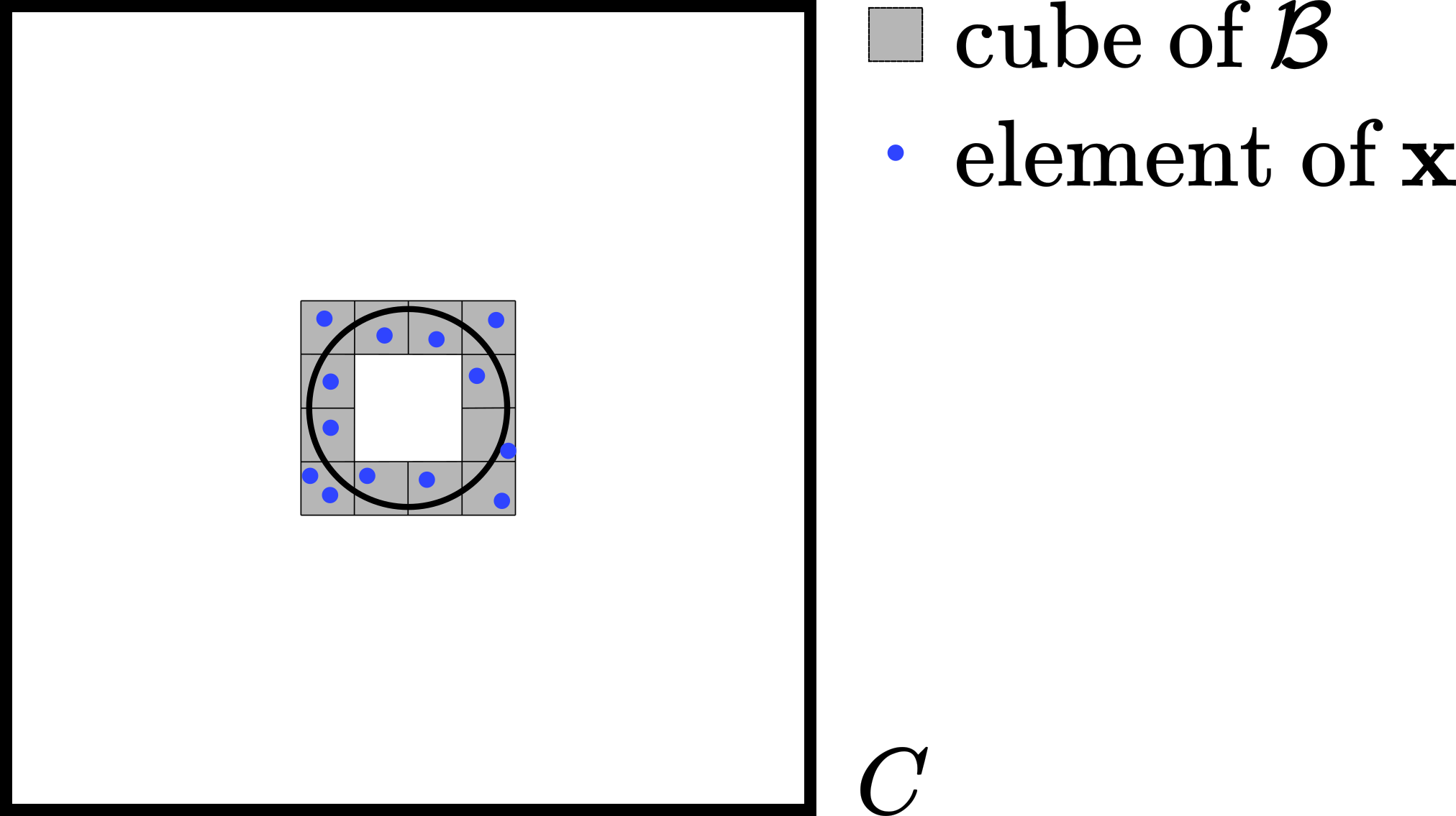}
\caption{\label{fig:circleFig} The setup in the proof of Lemma~\ref{lemma_euc_1}.}
\end{figure}

We may assume without loss of generality that $R=1$ and $C$ is centered at the origin. Let $S^i\subset \R^m$ be an $i$-dimensional sphere of diameter $1/4$ centered at the origin; note that $\PH_i\paren{S^i}$ consists of a single interval $\paren{0,1/8}$ for the \v{C}ech complex (a slightly different argument is required for the Rips complex).

Let 
\begin{equation}
\label{eq_acKappa}
\kappa=\min\paren{b,\frac{1}{8}-d,\frac{1}{24}}
\end{equation}
and
\[\lambda_0=\frac{m^{m/2} }{\kappa^m}\,.\]

Choose $\lambda>\lambda_0$ and set $\delta=\lambda^{-\frac{1}{m}}.$ Let $\mathcal{C}$ be the cubes in the standard tesselation of $\mathbb{R}^m$ by cubes of width $\delta$ and let
\[\mathcal{B}=\set{c\in\mathcal{C}:S^i\cap c\neq \varnothing}\,.\]
See Figure~\ref{fig:circleFig}.

If $\Xi\paren{\textbf{x},A,\mathcal{B}}=1,$ then
\[d_H\paren{\textbf{x}\cap C, S^i}<\kappa\]
where we used the fact that the length of the diagonal of an $m$-dimensional cube of width $\delta$ is $\delta\sqrt{m}.$  Stability and Equation~\ref{eq_acKappa} imply that $\PH_i\paren{\textbf{x}\cap C}$ includes an interval $\paren{\hat{b},\hat{d}}$ so that
\[\hat{b}<\kappa \leq b < d \leq \frac{1}{8}-\kappa < \hat{d} < \frac{1}{8}+\kappa \leq 1/6 \,.\]

By construction,
\[\frac{1}{2} d\paren{\textbf{x}\cap C,C^c} >\frac{1}{2}\paren{ d\paren{S^i,C^c}-\kappa}=\frac{3}{16}-\kappa/2\geq 1/6 >\hat{d}\,,\]
where $C^c$ is the complement of $C.$ Therefore, the the $\epsilon$-neighborhoods of $\textbf{x}\cap C'$ and $C^c$ are disjoint for all $\epsilon\leq \hat{d}$ and Equation~\ref{eq_euc_1} holds.
\end{proof}

\begin{proof}[Proof of Propostion~\ref{prop_AC}]
We will construct a set of bounded occupancy indicators of the form defined in the previous lemma, and apply Lemma~\ref{lemma_occupancy_2}. The reader may want to remind themselves of the definitions in Section~\ref{sec:occupancy}.

$\mu$ is absolutely continuous so its support contains a cube $C.$ Without loss of generality, we may assume that $C$ is a unit cube. Let $b_0=1/16,d_0=1/8,$ and $\lambda>\lambda_0,$ where $\lambda_0$ is as in the previous lemma. Set $\delta=n^{-1/m},$ and let  $\set{D_1,\ldots,D_{s}}$ be the sub-cubes in the cubical tessellation of width $\delta$ which are fully contained within $C.$ There is a constant $\kappa>0$ depending only on $m$ so that 
\begin{equation}
\label{eqn_propAC1}
s\geq \kappa\delta^{-m}\geq \kappa n\,.
\end{equation}

Assume $\delta$ is sufficiently small so that $k>\delta^m/2.$  Let $l\in\set{1,\ldots,s},$ and let $A_l$ and $\mathcal{B}_l$ be the set and collection of disjoint sub-cubes of  width  $\delta\lambda^{-\frac{1}{m}}$ contained in $C_l$ given by the previous lemma. It follows from the statement of that lemma that 
\begin{equation}
\label{eqn_propAC2}
p_i\paren{\textbf{x}_n,\frac{1}{16}n^{-1/m}}\geq \sum_{j=1}^{s}\Xi\paren{\textbf{x}_n,A_l,\mathcal{B}_l}\,.
\end{equation}

Let $c$ be the constant appearing the definition of Ahlfors regularity, $v_0$ be the volume of a unit ball in $\mathbb{R}^m,$ and 
\[q=\frac{c m^{m/2}v_0}{2^m}\qquad \text{and}\qquad p=\frac{v_0}{\lambda 2^m c}\,.\]
 If $E_0$ is a ball of radius $\delta\sqrt{m}/2$ containing $C,$

\[\mu\paren{A_l}\leq \mu\paren{E_0}\leq c\,\vol\paren{E_0}=c v_0 \paren{\delta\sqrt{m}/2}^m=q \delta^m=\frac{q}{n}\,.\]

Similarly, if $B$ is a cube of $\mathcal{B}_l,$ and $E_1$ is a ball of radius $\delta\lambda^{-\frac{1}{m}}/2$ contained in $B$
\[\mu\paren{B}\geq \mu\paren{E_1} \geq \frac{1}{c} \paren{\delta\lambda^{-\frac{1}{m}}}^m v_0=\frac{p}{n}\,.\]
Let $r=\abs{\mathcal{B}_l}$ and note that $r$ depends only on $\lambda_0, b_0,$ and $d_0.$ $\Xi\paren{\textbf{x}_n,A_l,\mathcal{B}_l}$ is a $n,p,q,r$-bounded occupancy indicator for each $l$, so by Lemma~\ref{lemma_occupancy_2} there exists a $\gamma_0>0$ so that
\begin{equation}
\label{eqn_propAC3}
\frac{1}{s}\sum_{j=1}^{s}\Xi\paren{\textbf{x}_n,A_l,\mathcal{B}_l}\geq \gamma_0
\end{equation}
with high probability as $n\rightarrow\infty.$ We have that
\begin{align*}
p_i\paren{\textbf{x}_n,\frac{1}{16}n^{-1/m}}\geq & \;\; \sum_{j=1}^{s}\Xi\paren{\textbf{x}_n,A_l,\mathcal{B}_l} &&\text{by Eqn.~\ref{eqn_propAC2}}\\
\geq & \;\; \gamma_0 s &&\text{by Eqn.~\ref{eqn_propAC3}}\\
\geq & \;\; \gamma_0 \kappa n &&\text{by Eqn.~\ref{eqn_propAC1}}
\end{align*}
with high probability as $n\rightarrow\infty.$

Then, by counting intervals of length greater than $n^{-1/m}/16,$
\begin{align*}
\lim_{n\rightarrow\infty}  n^{-\frac{m-\alpha}{m}} E_\alpha^i\paren{\textbf{x}_n}\geq& \;\; \lim_{n\rightarrow\infty}   n^{-\frac{m-\alpha}{m}} p_i\paren{\textbf{x}_n,\frac{1}{16} n^{-1/m}}\paren{\frac{1}{16}\paren{n^{-1/d}}}^\alpha\\
=& \;\; 16^{-\alpha} \lim_{n\rightarrow\infty}   \frac{1}{n}  p_i\paren{\textbf{x}_n,\frac{1}{16} n^{-1/d}}\\
\geq & \;\; 16^{-\alpha} \kappa \gamma_0 \\
\coloneqq & \;\; \Omega_0
\end{align*}
with high probability as $n\rightarrow\infty.$
\end{proof}
It is straightforward to modify the previous argument to work for any metric space $X$ with a subset $Y$ that is the bi-Lipschitz image of a cube in Euclidean space. In particular, if the bi-Lipschitz constant is $L$, it would suffice to take $b_0$ to be $\frac{1}{16 L^2}$ and $d_0$ to be $\frac{1}{8}$ and argue that an interval $\paren{b,d}\in\PH_i\paren{C}$ with $b<\frac{\delta}{16 L^2}<\frac{\delta}{8}$ corresponds to an interval $\paren{b_1,d_1}\in \PH_i\paren{X}$ with $b_1<\frac{\delta}{16 L} < \frac{\delta}{8 L} <d_1.$ 

\subsection{Non-triviality Constants}
\label{sec_nontriviality}

To prove the lower bound, we modify the approach in our paper on extremal $\PH$-dimension~\cite{2018schweinhart} to work in a probabilistic context. If $\mu$ is a $d$-Ahlfors regular measure on $\R^m$ and $\delta>0,$ let $\mathcal{C}_{\delta}\paren{\mu}$ be the cubes in the grid of mesh $\delta$ that intersect the support of $\mu.$  The basic idea is to sub-divide the grid of mesh $\delta$ so each cube contains $k^m$ sub-cubes. If $k$ is chosen carefully, we can find a positive fraction of cubes in $\mathcal{C}_\delta\paren{\mu}$ that contain enough cubes of $\mathcal{C}_{\delta/k}\paren{\mu}$ to guarantee a stable $\PH_i$ class. In fact, we can require that the sub-cubes have probability exceeding a certain threshold. We then control the number of stable $\PH_i$ classes realized by a random sample $\textbf{x}_n$ with Lemma~\ref{lemma_occupancy_2}.

In previous work~\cite{2018schweinhart}, we raised the question of how large a subset of the integer lattice can be without having a subset with ``stable'' $i$-dimensional persistent homology.
\begin{Definition}
\label{defn_nontrivial}
For $x\in \Z^m,$ let the cube corresponding to $x$ --- $C\paren{x}$ --- be the cube of width $1$ centered at $x.$ A subset $X$ of $\Z^m$ has a \textbf{stable} $i$-dimensional persistent homology class if there is a $c>0$ so that if $Y$ is any subset of $\cup_{x\in X}C\paren{x}$ satisfying
\[Y\cap C\paren{x}\neq \varnothing \quad \forall \; x\in X\,,\]
 then there is an $I\in PH_i\paren{Y}$ so that $\abs{I}>c$ (see Figure~\ref{fig:stable}). The supremal such $c$ is called the \textbf{size} of the stable persistence class.
\end{Definition}
Note that this notion depends on whether persistent homology is taken with respect to the Rips complex or the \v{C}ech complex, but is defined for both.

\begin{figure}
\centering  
\subfigure[]{\includegraphics[width=0.2\linewidth]{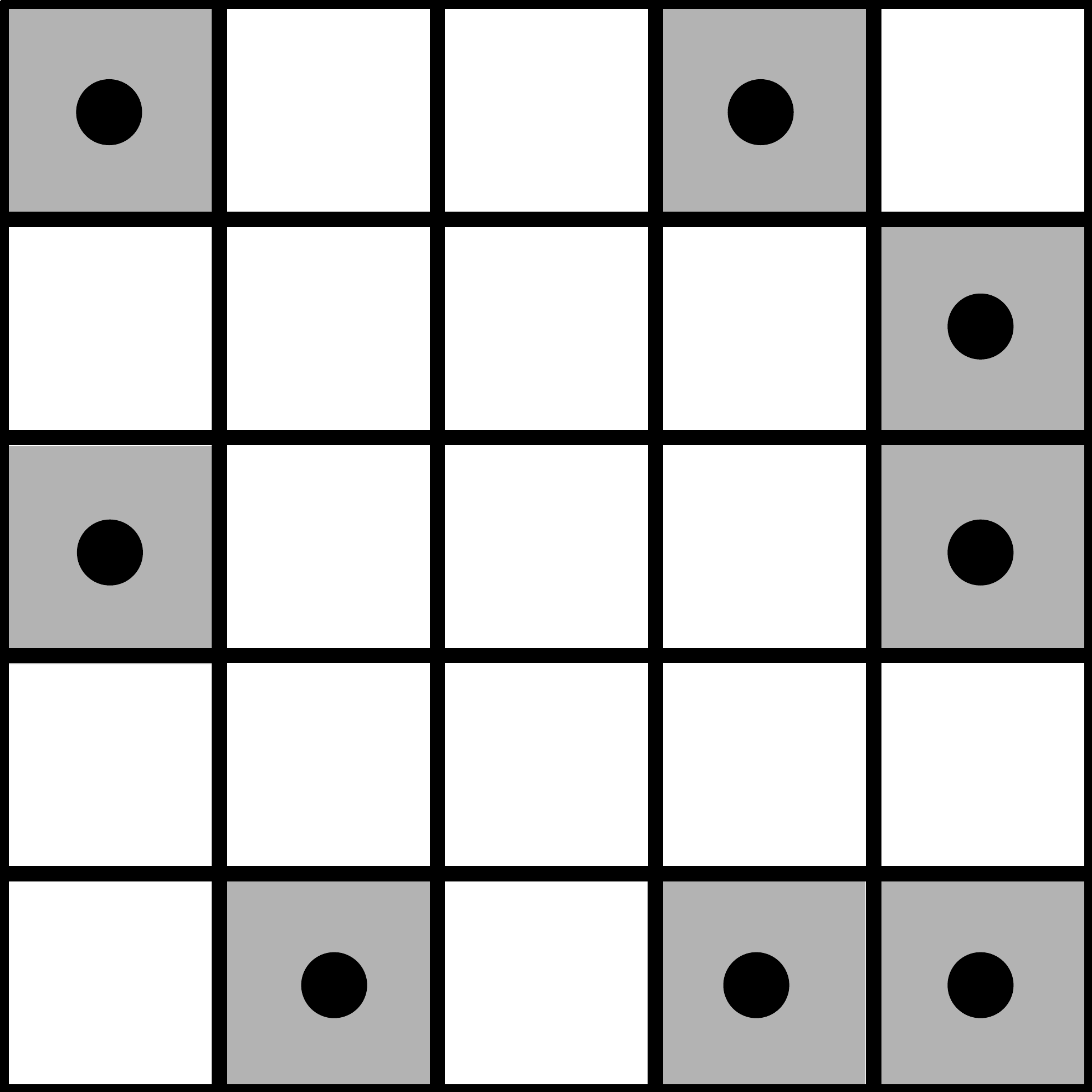}}
\subfigure[]{\includegraphics[width=0.2\linewidth]{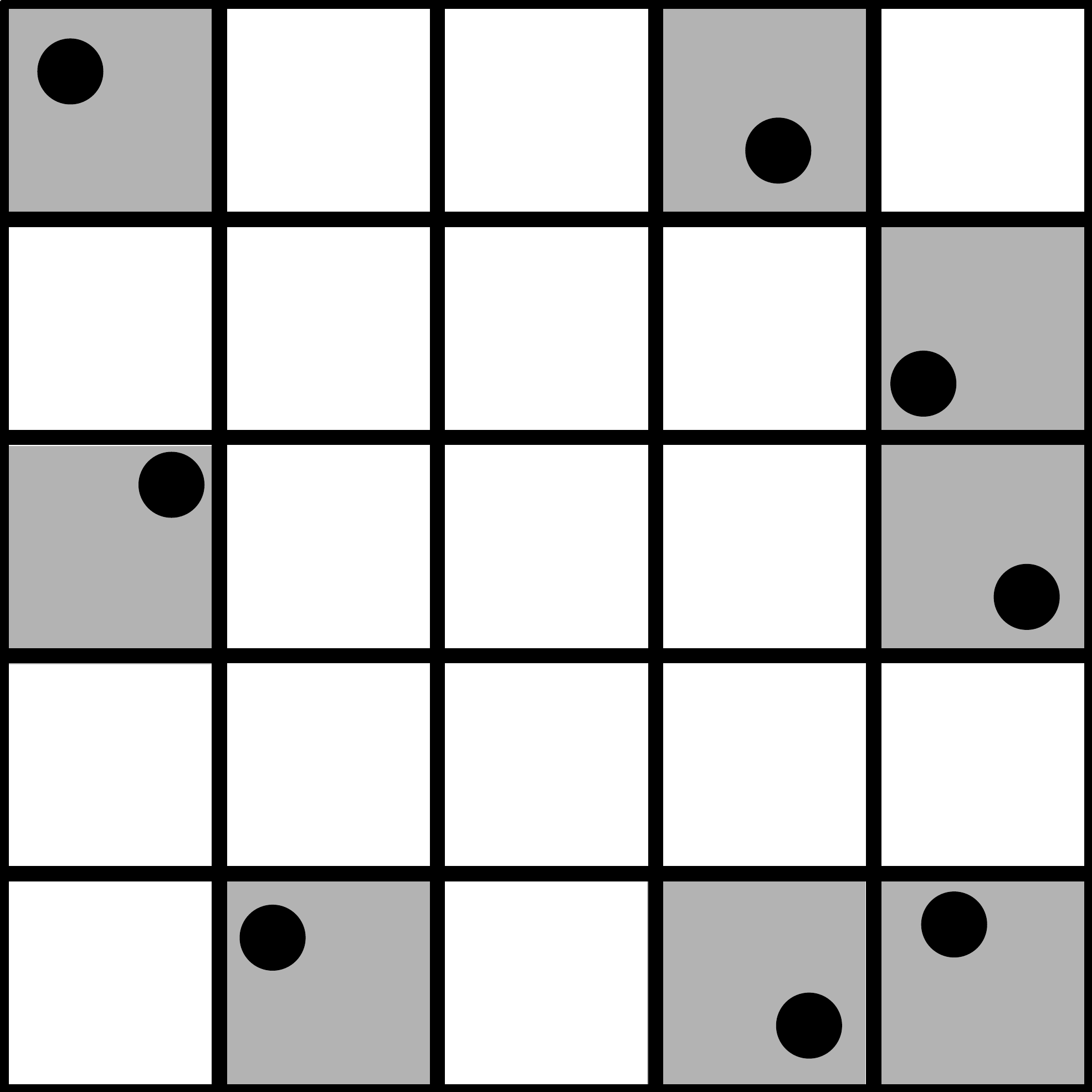}}
\subfigure[]{\includegraphics[width=0.2\linewidth]{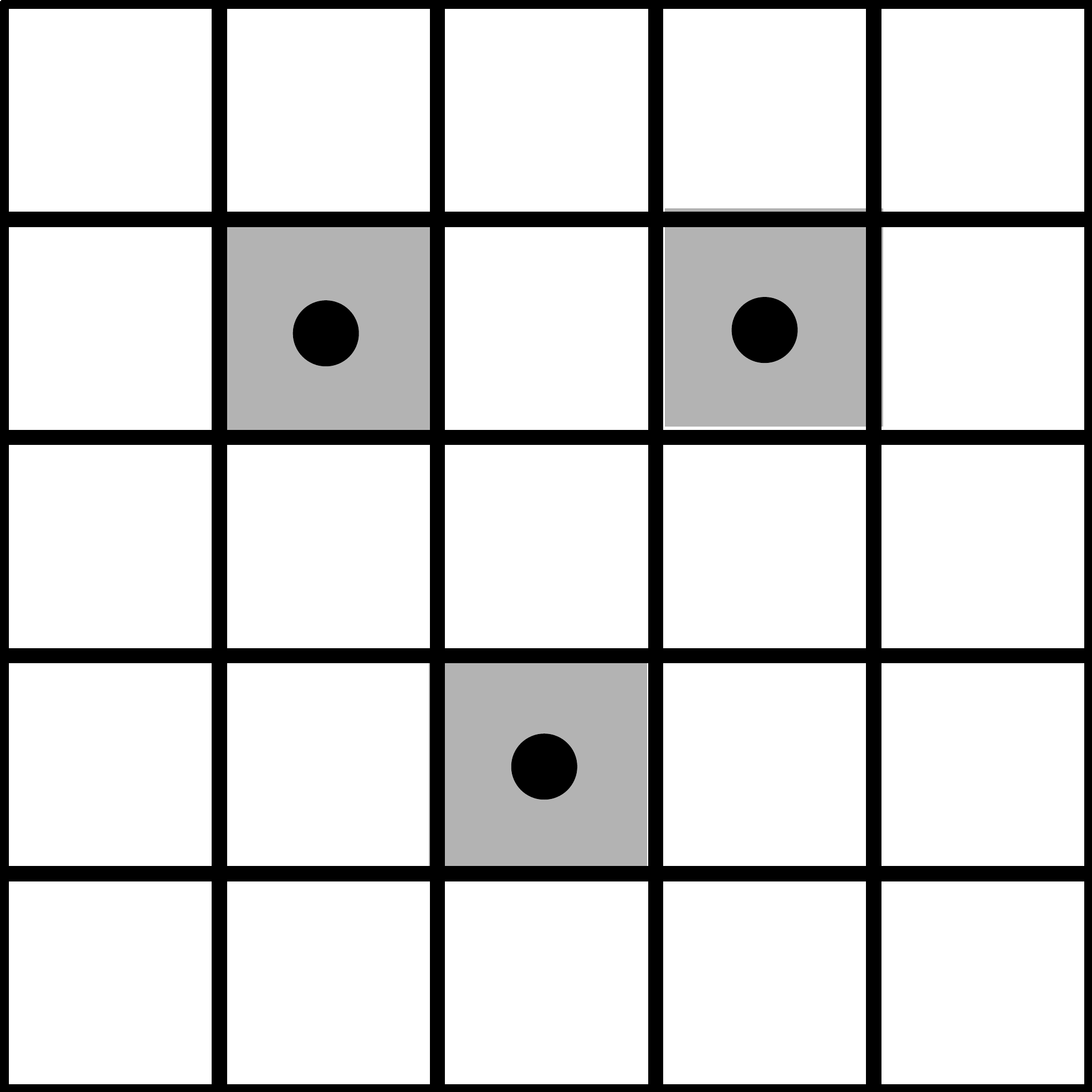}}
\subfigure[]{\includegraphics[width=0.2\linewidth]{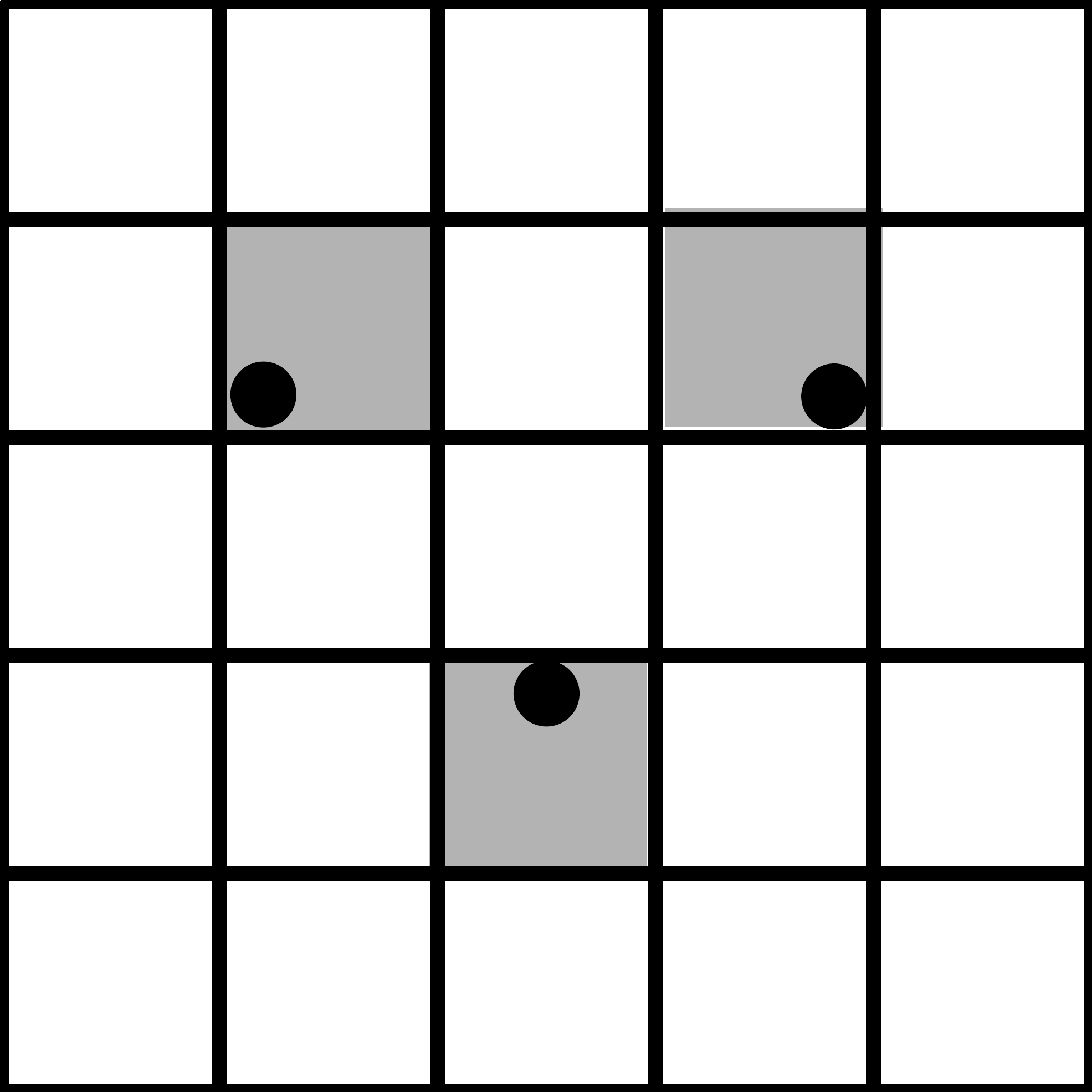}}
	\caption{\label{fig:stable}The \v{C}ech $\PH_1$ class of the lattice points corresponding to the gray cubes in (a) and (b)is stable --- any choice of one point in each cube will give the vertices of an acute triangle, and therefore a set with non-trivial $\PH_1.$ The one in (c) and (d) is not, because the points in (d) form an obtuse triangle so the persistent homology of that set is trivial.~\cite{2018schweinhart}}
\end{figure}

\begin{Definition}
\label{defn_gamma}
Let $\xi_i^m\left(N\right)$ be the size of the largest subset $X$ of $\set{1,\ldots,N}^m\subset \mathbb{Z}^m$ so that no subset $Y$ of $X$ has a stable $\PH_i$-class. Define
\[\gamma_i^m=\liminf_{N\rightarrow\infty} \frac{\log\paren{\xi_i^m\paren{N}}}{\log\paren{N}}\,.\]
\end{Definition}
$\gamma_i^m$ may depend on whether persistent homology is taken with respect to the Rips complex or the \v{C}ech complex, but we suppress the dependence here. $\gamma_0^m=0$ for all $m\in\N:$ any subset of $\Z^m$ with more than $3^m$ points has a minimum spanning tree edge of length at least $1$ and thus a $\PH_0$ interval of at length at least $1/2$ for the \v{C}ech complex and one of length at least $1$ for the Rips complex (at least $3^m$ points are necessary to rule out point sets with points from neighboring cubes). In~\cite{2018schweinhart}, we proved that $\gamma_1^m\leq m-\frac{1}{2}$ if persistent homology is taken with respect to the \v{C}ech complex. Note that Definition~\ref{defn_gamma} does not include the same restriction on the size as in~\cite{2018schweinhart}.

\subsection{Ahlfors Regular Measures and Box Counting}
\label{sec:AhlforsCubeCount}
Before proceeding to the proof of the lower bound, we prove two technical lemmas about the asymptotics of the number of cubes that intersect the support of a $d$-Ahlfors regular measure. The first is similar to Lemma~\ref{lemma_ballcount} on the asymptotic number of disjoint balls. Let $\mathcal{C}_\delta$ be the cubes in the cubical grid of mesh $\delta$ in $\mathbb{R}^m$ centered at the origin, and for $\delta,a>0$ define 
\begin{equation}
\label{eq:cube0}
\mathcal{C}_{\delta,a}\paren{\mu}=\set{C\in \mathcal{C}_\delta: \mu\paren{C}\geq a \delta^{d}}
\end{equation}
and
\begin{equation}
\label{eq:cube1}
N_{\delta,a}\paren{\mu}=\abs{\mathcal{C}_{\delta,a}\paren{\mu}}\,.
\end{equation}
(The upper and lower box dimensions of a subset of Euclidean space can be defined in terms of the asymptotic properties of $N_{\delta,0}\paren{X},$ analogously to Definition~\ref{defn:upperbox}.)

\begin{Lemma}
\label{lemma_respect}
If $\mu$ is a is $d$-Ahlfors regular measure with support $X\subset \R^m$, then there exist real numbers $0<c_0\leq c_1<\infty$ depending on $m$ and the constants $c$ and $d$ appearing in the definitions of Ahlfors regularity so that
\begin{equation}
\label{eqn_respect0}
c_0\delta^{-d} \leq N_{\delta,\hat{c}}\paren{\mu}\leq c_1 \delta^{-d}
\end{equation}
for all $\delta<\delta_0,$ where $\hat{c}=\frac{1}{c 2^m}.$ Similarly, there exist real numbers $0<c'_0\leq c'_1<$ depending on $c,$ $d,$ and $m$ so that
\begin{equation}
\label{eqn_respect1}
c'_0\delta^{-d} \leq N_{\delta,0}\paren{\mu} \leq  c'_1 \delta^{-d}
\end{equation}
for all $\delta<\delta_0.$
\end{Lemma}
\begin{proof}

Let $C$ be a cube in the grid of mesh $\delta$ that intersects $X,$ and $x\in C \cap X.$ First, we show that bounds for $N_{\delta,0}\paren{\mu}$ imply bounds for $N_{\delta,\hat{c}}\paren{\mu},$ and vice versa. By Ahlfors regularity, 
\[\mu\paren{B_{\delta}\paren{x}}> \frac{1}{c} \delta^{d}\,.\]
Also, $B_{\delta}\paren{x}$ intersects at most $2^m$ cubes in the grid of mesh $\delta,$ so at least one cube adjacent to $C$ has measure exceeding $\hat{c}\delta^{d}$ (where two cubes are adjacent if they share at least one point). Each cube of $\mathcal{C}_{\delta,\hat{c}}\paren{\mu}$ is adjacent to at most $3^m$ cubes of $\mathcal{C}_{\delta}\paren{\mu},$ so we can find a lower bound for $N_{\delta,\hat{c}}\paren{\mu}$ in terms of the number of cubes that intersect the support:
\begin{equation}
\label{eqn_respect2}
\frac{1}{3^m}N_{\delta,0}\paren{\mu} \leq N_{\delta,\hat{c}}\paren{\mu}\leq N_{\delta,0}\paren{\mu}\,,
\end{equation}
where the upper bound is trivial. 

To show the lower bounds in the statement, we compute
\begin{align*}
1=&\;\;\mu\paren{X}\\
\leq & \;\; \sum_{C\in\mathcal{C}_{\delta,0}\paren{\mu}}\mu\paren{C} \\
\leq & \;\; c \delta^d m^{d/2} N_{\delta,0}\paren{\mu}\\
\implies & \;\;  N_{\delta,0}\paren{\mu}\geq \frac{1}{c}m^{d/2} \delta^{-d}
\end{align*}
which is the lower bound in Equation~\ref{eqn_respect1} with $c_0'=\frac{1}{c}m^{d/2}.$ Then, by Equation~\ref{eqn_respect2}, the lower bound in Equation~\ref{eqn_respect0} holds with $c_0'=3^{-m}\frac{1}{c}m^{d/2}.$

For the upper bounds, note that the intersection of two cubes may have positive measure, but a cube can share measure with only $3^m-1$ adjacent cubes. It follows that
\begin{align*}
1=&\;\;\mu\paren{X}\\
\geq & \;\; \frac{1}{3^m} \hat{c} \delta^d N_{\delta,\hat{c}}\paren{\mu}\\
\implies & \;\;  N_{\delta,\hat{c}}\paren{\mu} \leq c \, 6^m \delta^{-d}\,,
\end{align*}
which is the upper bound in Equation~\ref{eqn_respect0} with $c_1=c \, 6^m.$ Then, upper bound in Equation~\ref{eqn_respect1} holds with $c_1'=c_1=c \, 18^m,$ using Equation~\ref{eqn_respect2}.
\end{proof}

\begin{figure}
\centering
\includegraphics[width=.7\textwidth]{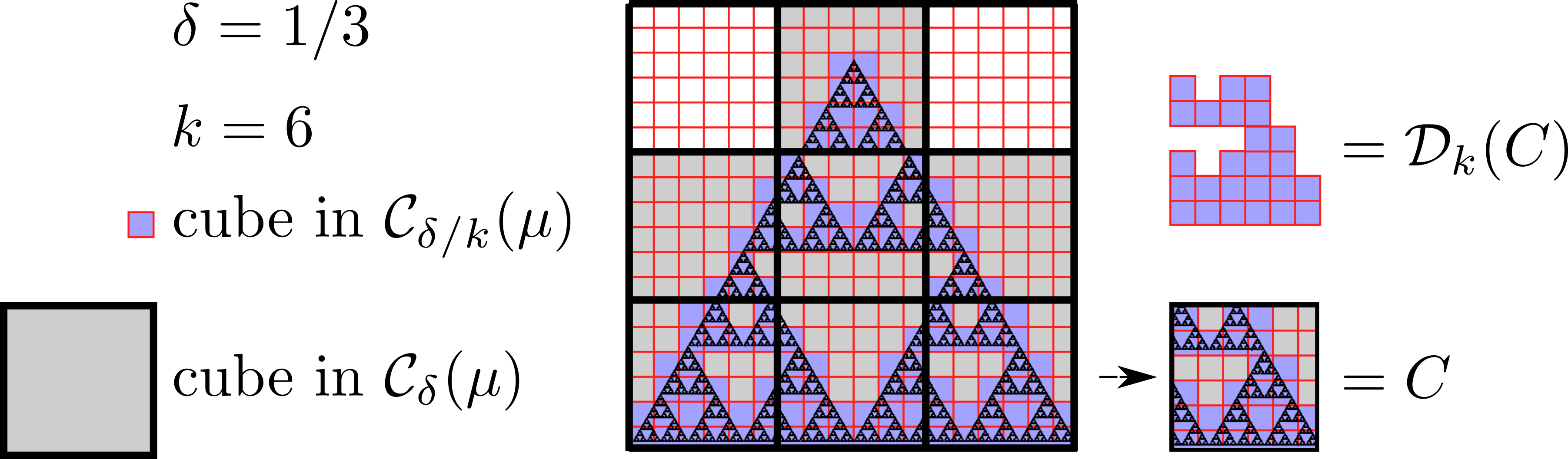}
\caption{Collections of cubes associated with the natural measure $\mu$ on the Sierpi\'{n}ski triangle.}
\label{fig_cdelta}
\end{figure}

We consider cubes at two different scales: cubes of width $\delta,$ and smaller cubes obtained by divding each cube of width $\delta$ into $k^m$ sub-cubes. Our eventual goal is to count the number of cubes of width $\delta$ which contain sufficiently many positive measure sub-cubes of width $\delta/k$ to define an occupancy event implying a persistent homology class. The next definition introduces collections of cubes corresponding to a measure $\mu,$ some of which are illustrated in Figure~\ref{fig_cdelta}. 

\begin{Definition}[Cube Collections Corresponding to a Measure]
\label{cube_defn}
Let $\mu$ be a $d$-Ahlfors regular measure on $\mathbb{R}^m$ and let $\mathcal{C}_{\delta,a}\paren{\mu}$ be as defined in Equation~\ref{eq:cube0}. For $k\in\N,$ $\delta>0,$ and $C\in\mathcal{C}_{\delta,0}\paren{\mu},$ define 
\begin{equation}
\label{eq:cube2}
\mathcal{D}_{k}\paren{C}=\set{D \in \mathcal{C}_{\delta/k,\hat{c}}\paren{\mu}: D\subset C}
\end{equation}
where  $\hat{c}=\frac{1}{c 2^m}$ as in Lemma~\ref{lemma_respect} above, and set
\begin{equation*}
D_{k}\paren{C}=\abs{\mathcal{D}_{k}\paren{C}}\,.
\end{equation*}
Next, we define a collection of cubes in $\mathcal{C}_{\delta,0}\paren{\mu}$ which contain sufficiently many sub-cubes. For $\delta>0$ and $0<\beta<d,$ let
\begin{equation*}
\mathcal{C}_\delta^{k,\beta}=\set{C\in \mathcal{C}_{\delta,0}\paren{\mu}: D_{k}\paren{C}>k^{\beta}}
\end{equation*}
and
\begin{equation*}
M\paren{\delta,k,\beta}=\abs{\mathcal{C}_\delta^{k,\beta}}\,.
\end{equation*}
\end{Definition}

Next, we prove a technical lemma establishing a lower bound for $M\paren{\delta,k,\beta}.,$ The argument is similar to that of Lemma~26 in~\cite{2018schweinhart}.

\begin{Lemma}
\label{lemma_lower_1}
If $\mu$ is a $d$-Ahlfors regular measure supported on $\R^m$ and $0<\beta<d,$ then there exists a $K$ so that for any $k>K$ there exit positive constants $\delta_1$ and $c_2$ so that 
\begin{equation}
\label{eq_lower_statement}
M\paren{\delta,k,\beta}> c_2 \delta^{-d}
\end{equation}
for all $\delta<\delta_1.$ 
\end{Lemma}
\begin{proof}

Let $c_0,$ $c_1',$ and $\delta_0$ be the constants from Lemma~\ref{lemma_respect} so 
\begin{equation}
\label{eq_lower_0}
N_{\delta,0}\paren{\mu}\leq c_1'\delta^{-d} \qquad \text{and} \qquad N_{\delta,\hat{c}}\paren{\mu}\geq c_0\delta^{-d}
\end{equation}
for all $\delta<\delta_0.$

A cube in $\mathcal{C}_{\delta,0}\paren{\mu}$ is either an element of $\mathcal{C}_\delta^{k,\beta}$ and contains between $k^\beta$ and $k^m$ sub-cubes of  $\mathcal{C}_{\delta/k,\hat{c}}\paren{\mu},$ or is contained in $C_{\delta,0}\paren{\mu}\setminus \mathcal{C}_\delta^{k,\beta}$ and can contain at most $k^\beta$  sub-cubes in that set. On the other hand, each sub-cube in $\mathcal{C}_{\delta/k,\hat{c}}\paren{\mu}$ is contained in exactly one larger cube in $\mathcal{C}_{\delta,0}\paren{\mu}.$ Therefore,
\begin{align*}
N_{\delta/k,\hat{c}}\paren{\mu}\leq &\;\; k^m M\paren{\delta,k,\beta}+k^{\beta} \abs{C_{\delta,0}\paren{\mu}\setminus \mathcal{C}_\delta^{k,\beta}}\\
 \leq & \;\; k^m M\paren{\delta,k,\beta}+k^{\beta} N_{\delta,0}\paren{\mu}\,.
\end{align*}

Re-arranging terms, we have that
\begin{align*}
M\paren{\delta,k,\beta}\geq &\;\; \frac{N_{\delta/k,\hat{c}}\paren{\mu} - k^{\beta} N_{\delta,0}\paren{\mu}}{k^m}\\
\geq &\;\; \frac{c_0 k^d \delta^{-d}-k^{\beta}c_1'\delta^{-d}}{k^m} &&\text{by Eqn.~\ref{eq_lower_0}}\\
=& \;\; \paren{c_0 k^{d-m} -c_1 ' k^{\beta-m}} \delta^{-d}\,.
\end{align*}

As $\beta<d,$ we can choose $K$ sufficiently large so that if $k>K$ then the coefficient
\[c_2\coloneqq \paren{c_0 k^{d-m} -c_1 ' k^{\beta-m}}\,,\]
is positive, and Equation~\ref{eq_lower_statement} holds for all $\delta<\delta_0,$ as desired. 

\end{proof}

\subsection{Proof of the Lower Bound in Theorems~\ref{thm_extremal} and~\ref{thm_probabalistic}}

We require one more lemma before proving the lower bound. The idea is similar to that of Lemma~\ref{lemma_dominated_mst}: we assemble a collection of occupancy indicators which each imply the existence of a persistence interval of a given length using Definition~\ref{defn_gamma} and Lemma~\ref{lemma_lower_1}, and apply Lemma~\ref{lemma_occupancy_2} to bound the total number of intervals in probability.

\begin{Lemma}
\label{lemma_lower_2}
If $\mu$ be a $d$-Ahlfors regular measure on $\R^m$ with $d>\gamma_i^m,$ then there exist positive real numbers $\epsilon_0$ and $\Omega_1$ so that
\[\lim_{n\rightarrow\infty} \frac{1}{n} p_i\paren{\textbf{x}_n, \epsilon_0 \, n^{-1/d}} \geq  \Omega_1\]
with high probability.
\end{Lemma}
\begin{proof}
Let $\gamma_i^m<\beta<d.$ By Definition~\ref{defn_gamma}, we can find a $K_0$ so that $k^{\beta}>\xi_i^m\paren{k}$ for all $k>K_0.$ Let $k>\min\paren{K,K_0},$ where $K_0$ is as given in the previous lemma, and let $\delta_1$ and $c_2$ also be as in that lemma. There are only finitely many collections of sub-cubes of $\brac{k}^m,$ so there are only finitely many  stable $\PH_i$ classes of subsets of $\brac{k}^m.$ Let $\epsilon_0$ be the minimum of the sizes of these stable classes.  

Let $\delta=n^{-1/d}$ and choose $n$ large enough so that $\delta<\delta_1.$ We will define a collection of occupancy indicators in terms of subsets of cubes in $\mathcal{C}_\delta^{k,\beta}$ which imply the length of a $\PH_i$ interval of length at least $\epsilon_0 \delta.$ To ensure that the indicators do not interfere with each other, let $\set{D_1,\ldots, D_s}$ be a maximal collection of cubes in $\mathcal{C}_\delta^{k,\beta}$ so that 
\begin{equation}
\label{eqn_lower2_0}
d\paren{D_j,D_l}> \paren{\delta+1} \sqrt{m}
\end{equation}
 for all $j,l\in\set{1,\ldots,s}$ so that $j\neq l.$ See Figure~\ref{fig_cdelta2}. There is a constant $0<\kappa <1$ that depends only on $d$ so that $s\geq \kappa M\paren{\delta,k,\beta}.$ Furthermore, by the previous lemma,
\[s \geq \kappa M\paren{\delta,k,\beta}> \kappa c_2 \delta^{-d}=\kappa c_2 n\,.\]

\begin{figure}
\centering
\includegraphics[width=.7\textwidth]{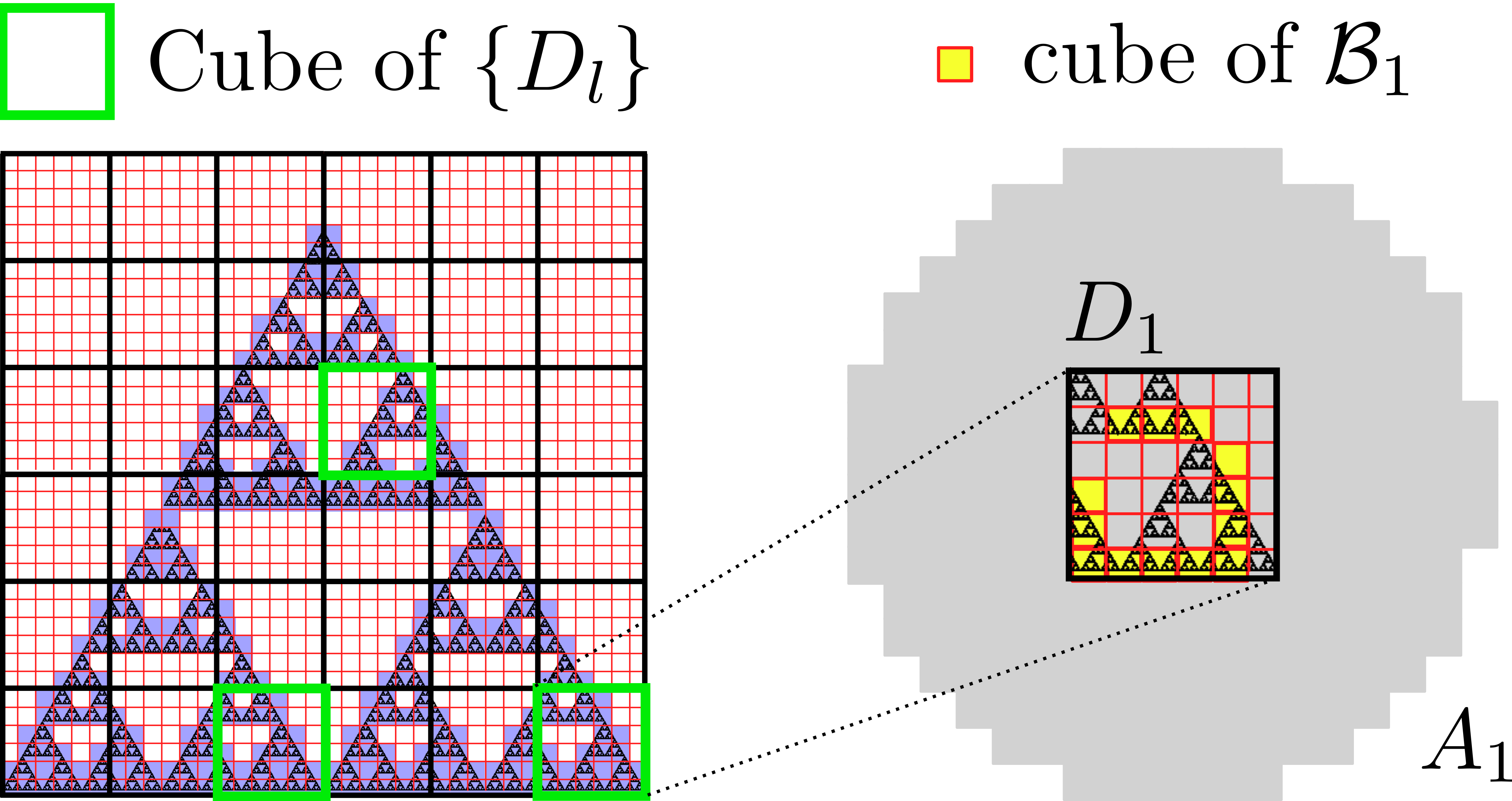}
\caption{The setup in the proof of Lemma~\ref{lemma_lower_2}.}
\label{fig_cdelta2}
\end{figure}

Let $l\in\set{1,\ldots,s}.$ By construction, $\mathcal{D}_{k}\paren{D_l}$ (defined in Equation~\ref{eq:cube2}) contains at least $k^\beta$ sub-cubes. $k>K_0,$ so $k^{\beta}>\xi_i^m\paren{k}$ and there is a collection of sub-cubes $\mathcal{B}_l\subset \mathcal{D}_{k}\paren{D_l}$ with a stable $\PH_i$ class (using Definition~\ref{defn_gamma}). Let
\[A_l=\hat{B}_{\delta\sqrt{m}}\paren{C}\setminus \cup_{B\in\mathcal{B}_l} B\]
where $\hat{B}_{\delta\sqrt{m}}\paren{D_l}$ is the union of all cubes in the grid of mesh $\delta/k$ within distance $\delta\sqrt{m}$ of $D_l$ (see Figure~\ref{fig_cdelta2}). Also, let $\mathcal{B}'_l$ be collection of the interiors of the sub-cubes $\mathcal{B}_l.$ Note that Equation~\ref{eqn_lower2_0} implies that the sets $A_l$ and $\mathcal{B}'_l$ are disjoint for different values of $l.$ It follows from property (2) in Section~\ref{sec_properties} that
\begin{equation*}
p_i\paren{\textbf{x}_n, \epsilon_0 \, n^{-1/d}}\geq \sum_{j=1}^{s}\Xi\paren{\textbf{x}_n,A_l,\mathcal{B}'_l}\,.
\end{equation*}

$A_l$ is contained in a ball of radius $\delta\sqrt{m}+\delta$ so if $q=c\paren{\sqrt{m}+1}^d$  then, by Ahlfors regularity,
\[\mu\paren{A_l} \leq c\delta^d{\sqrt{m}+1}^d=\frac{q}{n}\]
for all $l\in\set{1,\ldots,s}.$ Also, each $B\in \mathcal{B}_l$ is an cube of width $\delta/k$ in $\R^m$ so
\[\mu\paren{B}\geq \frac{1}{c}\paren{\frac{\delta\sqrt{m}}{2 k}}^d=\frac{p}{n}\,,\]
where $p=2^{-d}k^{-d}m^{d/2}/c.$ Therefore, $\Xi\paren{\textbf{x}_n,A_l,\mathcal{B}'_l}$ is a $n,p,q,k^m$-bounded occupancy indicator for each $l$, and the desired result follows from Lemma~\ref{lemma_occupancy_2}.
\end{proof}

The proof of the lower bound in Theorems~\ref{thm_extremal} and~\ref{thm_probabalistic} is now straightforward.
\begin{Proposition}
\label{prop_lower}
Let $\mu$ be a $d$-Ahlfors regular measure on $\R^m$ with $d>\gamma_i^m.$ Then there is an $\Omega>0$ so that 

\[\lim_{n\rightarrow\infty} n^{-\frac{d-\alpha}{d}} E_\alpha^i\paren{x_1,\ldots,x_n}\geq  \Omega \]
with high probability.
\end{Proposition}
\begin{proof}
Let $\epsilon_0$ be as in the previous lemma. By counting intervals of length greater than $\epsilon_0 n^{-1/d},$ we have that
\begin{align*}
\lim_{n\rightarrow\infty}  n^{-\frac{d-\alpha}{d}} E_\alpha^i\paren{\textbf{x}_n}\geq& \;\; \lim_{n\rightarrow\infty}   n^{-\frac{d-\alpha}{d}} p_i\paren{\textbf{x}_n,\epsilon_0 n^{-1/d}}\paren{\epsilon_0 n^{-1/d}}^\alpha\\
=& \;\; \epsilon_0^\alpha \lim_{n\rightarrow\infty}   \frac{1}{n}  p_i\paren{\textbf{x}_n,\epsilon_0 n^{-1/d}}\\
\geq & \;\; \epsilon_0^\alpha \Omega_1 &&\text{by Lemma~\ref{lemma_lower_2}}\\
\coloneqq & \;\; \Omega
\end{align*}
with high probability as $n\rightarrow\infty.$ 
\end{proof}

\section{Acknowledgments}
We would like to thank H. Adams and M. Kahle for interesting discussions. In particular, the computational experiments of Adams et al.~\cite{2019adams} helped to inspire this work. We would also like to thank J. Jaquette and M. Kahle for their helpful comments. Moreover, we are grateful to the anonymous referee for their valuable comments, which have led to an improved manuscript.

Funding for this research was provided by a NSF Mathematical Sciences Postdoctoral Research Fellowship under award number DMS-1606259.

\nocite{2017hiraoka}
\bibliographystyle{plain}

\bibliography{Bibliography}

\appendix

\section{Construction of Counterexample~\ref{prop:noLimit}}
\label{sec:noLimit}

We will construct a $d$-Ahlfors regular measure $\sigma$ with $d=\frac{\log\paren{2}}{\log\paren{3}}$ so that if $\set{z_j}_{j\in\mathbb{N}}$ are i.i.d. samples from $\mu$ and $0<\alpha<d$ then the quantity
\[n^{-\frac{d-\alpha}{d}}E_0^\alpha\paren{z_1,\ldots,z_n}\]
oscillates with high probability as $n\rightarrow\infty.$ Our example will be constructed as the intersection of a nested sequence of closed subsets $Y_1\supset Y_2 \supset Y_3 \ldots$ of $\brac{0,1},$ where each $Y_j$ is the union of finitely many congruent, disjoint intervals. At some scales the set will resemble the Cantor set, while at others it will resemble the Cantor set scaled by a factor $\frac{5}{7}.$ As described at the end of this section, the construction can easily be modified to produce a counterexample of dimension $d$ for any $d\in\paren{0,1}.$

We introduce notation and shorthand related to sets of intervals.  Call a finite set of disjoint intervals $\mathcal{I}$ an ``interval collection.'' We will abuse notation, and use $\mathcal{I}$ to refer to both the collection $\mathcal{I}$ and the union $\cup_{I\in\mathcal{I}}I.$  Let $\abs{\mathcal{I}}$ be the number of intervals in the collection, and $\norm{\mathcal{I}}$ be the minimum length of an interval.

Before proving Counterexample~\ref{prop:noLimit}, we prove three technical lemmas. The first one shows that if two point sets have the same ``interval membership'' in a fine enough interval collection, then the length of the corresponding minimum spanning trees is close. If $\mathcal{I}=\set{I_j}_{j=1}^k$ and $\set{x_1,\ldots,x_n}\subset \mathcal{I}$ let  $\phi_{\mathcal{I}}\paren{x_1,\ldots,x_n}$ record the interval membership of the points $x_1,\ldots,x_n:$
\[\phi_{\mathcal{I}}\paren{x_1,\ldots,x_n}=\paren{l_1,\ldots,l_n}\text{ if } x_1\in I_{l_1},\ldots, x_n\in I_{l_n}\,.\]

\begin{Lemma}
Let $n\in \N,\epsilon_0>0,$ and $\alpha>0.$ There exists a $\delta>0$ so that if $\mathcal{I}$ is an interval collection with $\norm{\mathcal{I}}<\delta$ and $\set{x_1,\ldots,x_n}, \set{y_1,\ldots,y_n}\subset \mathcal{I}$ satisfy 
\[\phi_{\mathcal{I}}\paren{x_1,\ldots,x_n}=\phi_{\mathcal{I}}\paren{y_1,\ldots,y_n}\,\]
then 
\[\abs{E_\alpha^0\paren{x_1,\ldots,x_n}-E_\alpha^0\paren{y_1,\ldots,y_n}}<\epsilon_0\,.\]
\end{Lemma}
\begin{proof}
The function $x\rightarrow x^\alpha$ is $\alpha$-Holder continuous on $\brac{0,1}$ so there exists a $C>0$ so that 
\begin{equation}
\label{eqn_nolimit0_1}
\abs{x^\alpha-y^\alpha}<C\abs{x-y}^\alpha
\end{equation}
for all $x,y\in\brac{0,1}.$ Let 
\begin{equation}
\label{eqn_nolimit0_2}
\delta=\frac{1}{2}\paren{\frac{\epsilon_0}{Cn}}^{1/\alpha}\,.
\end{equation}

If $\mathcal{I},\set{x_1,\ldots,x_n},$ and $\set{y_1,\ldots,y_n}$ satisfy the hypotheses then $\abs{x_i-y_i}<\delta$ for $i=1,\ldots, n,$ because $x_i$ and $y_i$ are contained in an interval whose length is less than $\delta.$

 It follows that
\begin{align*}
\abs{E_\alpha^0\paren{x_1,\ldots,x_n}-E_\alpha^0\paren{y_1,\ldots,y_n}}=&\;\;\abs{\sum_{i=1}^{n-1}\paren{x_{i+1}-x_i}^\alpha-\paren{y_{i+1}-y_i}^\alpha}\\
\leq & \;\; \sum_{i=1}^{n-1}\abs{\paren{x_{i+1}-x_i}^\alpha-\paren{y_{i+1}-y_i}^\alpha}\\
\leq & \;\; \sum_{i=1}^n C \abs{\paren{x_{i+1}-y_{i+1}}+\paren{x_i-y_i}}^\alpha &&\text{by Eqn.~\ref{eqn_nolimit0_1}}\\
< & \;\; n C 2^\alpha\delta^\alpha &&\text{by $\norm{\mathcal{I}}<\delta$}\\
= & \;\; \epsilon_0  &&\text{by Eqn.~\ref{eqn_nolimit0_2}}\,.
\end{align*}
\end{proof}

The next lemma shows that probability measures supported on a fine enough interval collection which induce the same distribution of interval membership have random minimum spanning trees with similar lengths. If $\mu$ is supported on an interval collection $\mathcal{I},$  let $\mu^n_{\mathcal I}$ be the discrete random variable $\phi_\mathcal{I}\paren{x_1,\ldots,x_n}.$

\begin{Lemma}\label{lemma:noLimit2}
Let $\epsilon>0$ and $0<\alpha<d.$ Let $S_1\supset S_2 \ldots $ be a nested sequence of interval collections with  $\norm{S_j}\rightarrow 0,$ and let $\mu$ be  a probability measure supported on $\cap_{j} S_j$ so that
\begin{equation}
\label{eqn_noLimit2_0}
\abs{n^{-\frac{d-\alpha}{d}} E_\alpha^0\paren{x_1,\ldots,x_n} -c}<\epsilon/2
\end{equation}
with probability greater than $1-\epsilon$ for an integer $n>0.$ Then there exists an $M\paren{\epsilon,n}$ so that if $j>M\paren{\epsilon,n}$ and $\nu$ is any probability measure supported on $S_j$ satisfying $\nu^n_{S_j}=\mu^n_{S_j}$ (that is, $\mu$ and $\nu$ induce the same discrete probability distribution on the intervals in the set $S_j$) then 
\[\abs{c-n^{-\frac{d-\alpha}{d}} E_\alpha^0\paren{y_1,\ldots,y_n}}<\epsilon\]
with probability greater than $1-\epsilon,$ where $\set{y_k}_{k\in\mathbb{N}}$ are i.i.d. points sampled from $\nu.$
\end{Lemma}
\begin{proof}
For convenience, define
\[F\paren{x_1,\ldots,x_n}=\abs{c-n^{-\frac{d-\alpha}{d}} E_\alpha^0\paren{x_1,\ldots,x_n}}\,.\]
Let  $\delta>0$ be as given in by the previous lemma for $\epsilon_0=n^{\frac{d-\alpha}{d}}\epsilon/2.$ Choose $M\paren{\epsilon,n}$ sufficiently large so that $\norm{S_j}<\delta$ for $j>M\paren{\epsilon,n}.$

Let $j>M\paren{\epsilon,n}$ and define
\[V=\phi_{S_j}\paren{F^{-1}\paren{\epsilon/2}}\,.\]
That is,  $\paren{l_1,\ldots,l_n}\in V$ if there exist points $x_1\in I_{l_1},\ldots,x_n\in I_{l_n}$ so that $F\paren{x_1,\ldots,x_n}<\epsilon/2$ (where $S_j=\set{I_s})$. By Equation~\ref{eqn_noLimit2_0}
\[\mu^n_{S_j}\paren{V}>1-\epsilon\]
and, because the discrete random variables $\mu^n_{S_j}$ and $\nu^n_{S_j}$ coincide by hypothesis,
\[\nu^n_{S_j}>1-\epsilon\,.\]
Therefore, with probability greater than $1-\epsilon,$ $\phi_{\mathcal{I}}\paren{y_1,\ldots,y_n}\in V$ so there exist $z_1,\ldots,z_n$ satisfying $F\paren{z_1,\ldots,z_n}<\epsilon/2$ and $\phi_{\mathcal{I}}\paren{y_1,\ldots,y_n}=\phi_{\mathcal{I}}\paren{z_1,\ldots,z_n}.$ $\paren{z_1,\ldots,z_n}$ and $\paren{y_1,\ldots,y_n}$ satisfy the hypotheses of the previous lemma, so
\begin{align*}
F\paren{E_\alpha^0\paren{y_1,\ldots,y_n}}\leq & \;\; \abs{n^{-\frac{d-\alpha}{d}}\paren{ E_\alpha^0\paren{z_1,\ldots,z_n}- E_\alpha^0\paren{y_1,\ldots,y_n}}}+F\paren{z_1,\ldots,z_n}\\
\leq &\;\; n^{-\frac{d-\alpha}{d}}\epsilon_0+\epsilon/2\\
= &\;\; \epsilon\,.
\end{align*}
\end{proof}

The third lemma compares the behavior of random minimum spanning trees on an interval collection with that of one on another interval collection formed by translating its intervals. A natural map between two interval collections $\mathcal{I}$ and $\mathcal{J}$ is an order-preserving homeomorphism $f:\mathcal{I}\rightarrow\mathcal{J}$ so that, for any $I\in\mathcal{I},$ $f\mid_I $ is a translation and $f\paren{I}$ is an interval in $\mathcal{J}.$ Note that if $\mathcal{I}$ and $\mathcal{J}$ are sets of disjoint, congruent intervals so that $\abs{\mathcal{I}}=\abs{\mathcal{J}}$ and $\norm{\mathcal{I}}=\norm{\mathcal{J}},$ then there is a unique natural map between them.

\begin{Lemma}\label{lemma:noLimit1} Let $\mathcal{I}$ and $\mathcal{J}$ be interval collections contained in $\brac{0,1}$, and suppose that there is a natural map $f:\mathcal{I}\rightarrow\mathcal{J}.$ Let $0<\alpha<d,$ and let $\mu$ be a probability measure supported on $\mathcal{I}$ so that
\[n^{-\frac{d-\alpha}{d}} E_\alpha^0\paren{\textbf{x}_n}\rightarrow c\]
in probability as $n\rightarrow\infty$, for some real number $c.$ Then 
\[n^{-\frac{d-\alpha}{d}} E_\alpha^0\paren{f\paren{x_1},\ldots,f\paren{x_n}}\rightarrow c\]
in probability as $n\rightarrow\infty.$
\end{Lemma}

\begin{figure}
\centering
\includegraphics[width=.8\textwidth]{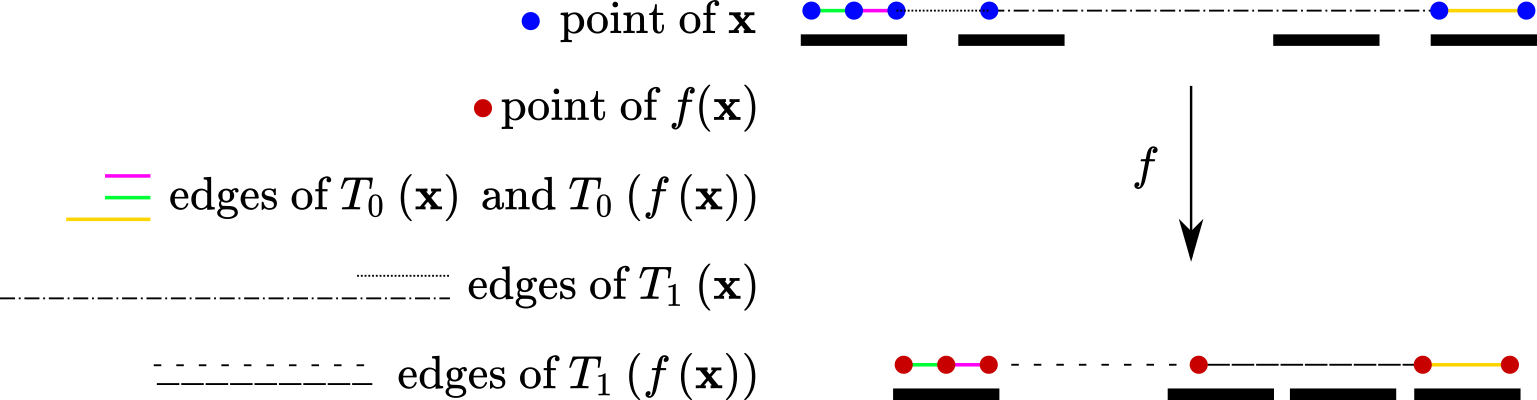}
\caption{The setup in the proof of Lemma~\ref{lemma:noLimit1}.}
\label{fig:lemmaNoLimit1}
\end{figure}

\begin{proof}
First, note that if  $\set{y_1,\ldots,y_n}$ is an ordered set of points in $\mathbb{R},$ the edges of the minimum spanning tree $T\paren{y_1,\ldots,y_n}$ are the intervals $\brac{y_1,y_2}, \ldots \brac{y_{n-1},y_n}.$ For a finite point set $\textbf{x}\subset\mathcal{I}$ let $T_0\paren{\textbf{x}}$ be the set of edges of $T\paren{\textbf{x}}$ that are contained in an interval of $\mathcal{I}$:
 
\[T_0\paren{\textbf{x}}=\set{e\in T\paren{\textbf{x}}:e\subseteq I\text{ for some }I\in\mathcal{I}}\]
and let $T_1\paren{\textbf{x}}$ consist of the remaining edges:
\[T_1\paren{\textbf{x}}=T\paren{\textbf{x}}\setminus T_0\paren{\textbf{x}}\,.\]
See Figure~\ref{fig:lemmaNoLimit1}. Let $k=\abs{\mathcal{I}}$ and note that $\abs{T_1\paren{\textbf{x}}}<k.$

Recall that $\textbf{x}_n$ is shorthand for $\set{x_1,\ldots,x_n}.$ If the edge $\brac{x_{j},x_{j+1}}$ is contained in $T_0\paren{\textbf{x}_n},$ then $\brac{x_{j},x_{j+1}}\subseteq I$ for some $I\in\mathcal{I},$ and by the definition of a natural map there is a $J\in\mathcal{J}$ so that $f\paren{I}=J$ and $f_{\mid I}$ is a translation. Therefore $\brac{f\paren{x_{j}},f\paren{x_{j+1}}}$ is an interval in $T_0\paren{f\paren{\textbf{x}_n}}$ of the same length as $\brac{x_{j},x_{j+1}}.$ It follows that there is a length-preserving bijection between the edges of $T_0\paren{\textbf{x}_n}$ and  $T_0\paren{f\paren{\textbf{x}_n}},$ and
\begin{equation}
\label{eqn_here0}
\sum_{e\in T_0\paren{\textbf{x}_n}}\abs{e}^\alpha=\sum_{e\in T_0\paren{f\paren{\textbf{x}_n}} }\abs{e}^\alpha\,.
\end{equation}

 For $\epsilon>0,$ choose $N$ sufficiently large so that for all $n>N$  
\begin{equation}
\label{eqn_here1}
k n^{-\frac{d-\alpha}{d}} <\epsilon/4
\end{equation}
and
\begin{equation}
\label{eqn_here2}
\abs{n^{-\frac{d-\alpha}{d}} E_\alpha^0\paren{\textbf{x}_n}-c}<\epsilon/2
\end{equation}
with probability greater than $1-\epsilon.$

Let $n>N.$ We have that
\begin{align*}
\abs{E_\alpha^0\paren{f\paren{\textbf{x}_n}}- E_\alpha^0\paren{\textbf{x}_n}} \leq & \;\; \abs{\sum_{e\in T_0\paren{f\paren{\textbf{x}_n}}}\abs{e}^\alpha-\sum_{e\in T_0\paren{\textbf{x}_n}}\abs{e}^\alpha} \\
& +  \sum_{e\in T_1\paren{f\paren{\textbf{x}_n}}}\abs{e}^\alpha+\sum_{e\in T_1\paren{\textbf{x}_n}}\abs{e}^\alpha\\
= & \;\; 0+  \sum_{e\in T_1\paren{f\paren{\textbf{x}_n}}}\abs{e}^\alpha+\sum_{e\in T_1\paren{\textbf{x}_n}}\abs{e}^\alpha &&\text{using Equation~\ref{eqn_here0}}\\
< &\;\; 2k && \text{all edges are contained in $\brac{0,1}$}\\
< &\;\; \epsilon/2 n^{\frac{d-\alpha}{d}} &&\text{using Equation~\ref{eqn_here1}\,.}
\end{align*}

Therefore, 
\begin{align*}
\abs{n^{-\frac{d-\alpha}{d}} E_\alpha^0\paren{f\paren{\textbf{x}_n}}-c} \leq &  \;\; n^{-\frac{d-\alpha}{d}}\abs{ E_\alpha^0\paren{f\paren{\textbf{x}_n}}- E_\alpha^0\paren{\textbf{x}_n}} + \abs{ n^{-\frac{d-\alpha}{d}}E_\alpha^0\paren{\textbf{x}_n}-c} \\ 
 < & \;\; \epsilon/2 + \epsilon/2\\
 = & \;\; \epsilon
\end{align*}
with probability greater than $1-\epsilon,$ where we used the previous computation and Equation~\ref{eqn_here2}.
\end{proof}

We are now ready to construct the counterexample. 

\begin{figure}
\centering
\includegraphics[width=.5\textwidth]{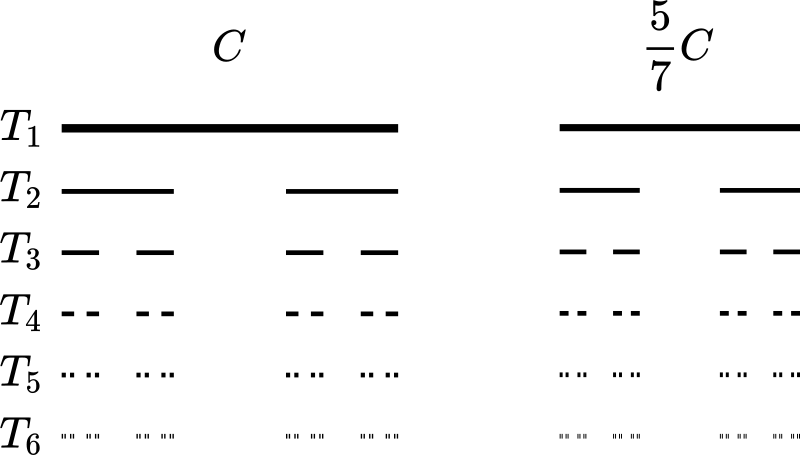}
\caption{The nested sequences of closed sets defining $C$ and $\frac{5}{7}C.$}
\label{fig:Cantor}
\end{figure}

As described in the introduction, our counterexample is related to the Cantor set. The Cantor set can be defined in terms of a middle thirds operation on interval collections. If $I$ is an interval, let $K\paren{I}$ be the set of two intervals of length $1/3\abs{I}$
\[K\paren{I}=\set{1/3 I, 1/3 I + 2/3\abs{I}}\,,\]
 and if $\mathcal{I}$ is an interval collection let 
\[K\paren{\mathcal{I}}=\set{K\paren{I}:I\in \mathcal{I}}\,.\]
 Define $T_m$ to be the set of intervals obtained by applying $K$ to $\set{\brac{0,1}}$ $m-1$ times. $T_m$ consists of $2^{m-1}$ intervals of length $\paren{1/3}^{m-1}.$  Then the Cantor set is
\[C=\cap_{m\in\mathbb{N}}T_m\,.\]
See Figure~\ref{fig:Cantor}.

\begin{figure}
\centering
\includegraphics[width=.4\textwidth]{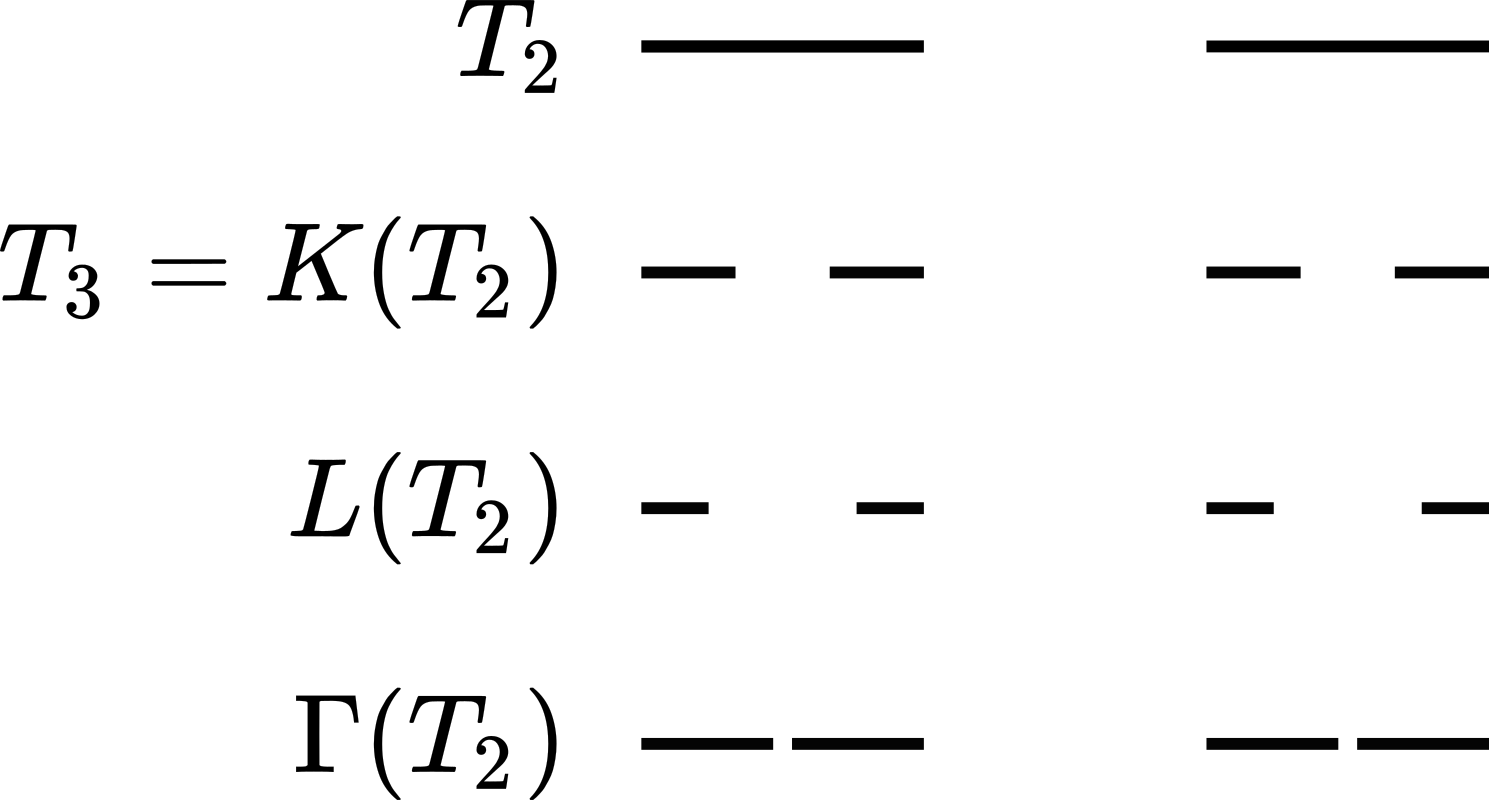}
\caption{The interval operations $L$ and $\Gamma$ applied  $T_2,$ the second interval collection used in the definition of the Cantor set.}
\label{fig:intervalOps}
\end{figure}

We define two more operations on interval collections. One produces slightly thinner intervals than $K$ does, and the other produces slightly thicker intervals. If $I$ is an interval, let 
\[L\paren{I}=\set{\paren{\frac{5}{7}}\paren{\frac{1}{3}}I,\paren{\frac{5}{7}}\paren{\frac{1}{3}}\abs{I}+\paren{1-\paren{\frac{5}{7}}\paren{\frac{1}{3}}}I}\]
and 
\[\Gamma\paren{I}=\set{\paren{\frac{7}{5}}\paren{\frac{1}{3}}I,\paren{\frac{7}{5}}\paren{\frac{1}{3}}\abs{I}+\paren{1-\paren{\frac{7}{5}}\paren{\frac{1}{3}}}I}\,.\]
If $\mathcal{I}$ is an interval collection, define $L\paren{\mathcal{I}}$ and $\Gamma\paren{\mathcal{I}}$ by performing the operation on each interval in the collection. See Figure~\ref{fig:intervalOps}. The scaling factor $\frac{7}{5}$ was chosen so that the intervals of $\Gamma\paren{S_k}$ are disjoint.

If $S_1\supset S_2 \supset S_3\ldots $ is a nested sequence of interval collections, there is natural probability measure $\mu_{S}$ on $\cap_n S_n$ that assigns equal probability to each interval of $S_j,$ for each value of $j.$ That is,
\[\mu_{S}\paren{I}=\frac{1}{\abs{S_j}}\text{ for }I\in S_j\,.\] 
For example, the natural measure on the Cantor set is $\mu_T,$ where $T_1\supset T_2 \supset T_3\ldots $ is the nested sequence of interval collections defined above (and depicted in Figure~\ref{fig:Cantor}). It assigns probability $1/2$ to each interval in $T_2,$ probability $1/4$ to each interval in $T_3,$ and so on.

\begin{proof}[Proof of Counterexample~\ref{prop:noLimit}] Set $d=\frac{\log\paren{2}}{\log\paren{3}}$ and choose $0<\alpha<d.$ Also, let $\mu$ and $\nu$ be the natural probability measures on $C$ and $\frac{5}{7}C,$ where $C$ is the Cantor set. Let $\set{x_j}_{j\in\N}$ and $\set{y_j}_{j\in\N}$ be i.i.d. samples from $\mu$ and $\nu,$ respectively.  Assume that there is a real number $c$ so that
\begin{equation}
\label{eqn_proofcounter0}
n^{-\frac{d-\alpha}{d}}E_\alpha^0\paren{x_1,\ldots,x_n}\rightarrow c
\end{equation}
in probability as $n\rightarrow \infty.$ If this was false, $C$ would be our desired example. Theorem~\ref{thm_mst} implies that $c>0.$ By rescaling the Cantor set, we have that  
\[n^{-\frac{d-\alpha}{d}}E_\alpha^0\paren{y_1,\ldots,y_n}\rightarrow \paren{\frac{5}{7}}^\alpha c \]
in probability as $n\rightarrow \infty.$

We will construct a nested sequence of interval collections $S_1\supset S_2 \supset S_3 \ldots $ mostly by applying the middle thirds operation $K$, but infrequently and alternately applying the operations $L$ and $\Gamma.$

Let $\set{\epsilon_i}$ be a sequence of real numbers converging to zero. We proceed by induction. Let $n_1$ be large enough so 
\[\abs{c-n_1^{-\frac{d-\alpha}{d}} E_\alpha^0\paren{x_1,\ldots,x_{n_1}}}<\epsilon_1/2\]
with probability greater than $1-\epsilon_1,$ choose $m_1> M\paren{\epsilon_1,n_1},$ where $M\paren{\epsilon_1,n_1}$ is as given in Lemma~\ref{lemma:noLimit2}. Let $S_1=T_{m_1},$ where $T_{m_1}$ is the $m_1$-th interval collection from the construction of the Cantor set. Then, by the definition of $M\paren{\epsilon_1,n_1},$ if $\sigma$ is any probability measure $\sigma$ satisfying
\[\sigma\paren{I}=\frac{1}{\abs{S_1}}\text{ for $I\in S_1$}\]
(that is, $\sigma$ assigns the same probabilities to intervals in $S_1$ as $\mu$ does) then
\[\abs{c-n_{1}^{-\frac{d-\alpha}{d}} E_\alpha^0\paren{x_1,\ldots,x_{n_{1}}}}<\epsilon_{1}\]
with probability greater than $1-\epsilon_{1}.$ (Note that we are indexing the sets $S_j$ differently than described in the introduction but the resulting example is the same.)

By way of induction suppose that there are integers $n_1,\ldots, n_{k-1}$ and a nested sequence of interval collections $S_1 \supset \ldots \supset S_{k-1}$ so that
\begin{itemize}
\item For odd $i=1,\ldots,k-1,$ $S_{i}$ consists of $2^{b_i}$ disjoint intervals of length $\paren{\frac{1}{3}}^{b_i}$ for some integer $b_i.$
\item For even $i=2,\ldots,k-1,$ $S_{i}$ consists of $2^{b_i}$ disjoint intervals of length $\frac{5}{7}\paren{\frac{1}{3}}^{b_i}$ for some integer $b_i.$ 
\end{itemize}
and, furthermore, if $\sigma$ is any probability measure on $S_{k-1}$ satisfying
\[\sigma\paren{I}=\frac{1}{\abs{S_j}}\text{ for $I\in S_j,$ $j=1,\ldots,k-1$}\]
then if $\set{z_j}_{j\in\mathbb{N}}$ are i.i.d. samples from $\sigma,$
\begin{itemize} 
\item for odd $i=1,\ldots,k-1,$ 
\[\abs{c-n_{i}^{-\frac{d-\alpha}{d}} E_\alpha^0\paren{x_1,\ldots,x_{n_{i}}}}<\epsilon_{i}\]
 with probability greater than $1-\epsilon_{i}.$
\item for even $i=2,\ldots,k-1,$ 
\[\abs{\paren{\frac{5}{7}}^\alpha c -n_{i}^{-\frac{d-\alpha}{d}} E_\alpha^0\paren{x_1,\ldots,x_{n_{i}}}}<\epsilon_{i}\]
 with probability greater than $1-\epsilon_{i}.$
\end{itemize}

If $k$ is odd, let $\mathcal{I}_k=\Gamma\paren{S_{k-1}},$ so $\mathcal{I}_k$ consists of $2^{b_{k-1}+1}$ disjoint intervals of length $\paren{\frac{1}{3}}^{b_{k-1}+1}.$ It follows that there is a natural map $f_k:T_{b_{k-1}+1}\rightarrow \mathcal{I}_k,$ where $T_{b_{k-1}+1}$ is the $\paren{b_{k-1}+1}$-st interval collection in the construction of the Cantor set.  Let $\mu_k$ be the pushforward of $\mu$ by $f_k$ (recall that $\mu$ is the natural measure on the Cantor set)  and let $\set{w_j}_{j\in\mathbb{N}}$ be i.i.d. samples from $\mu_{k}.$ $f_k$ is a natural map, so Lemma~\ref{lemma:noLimit1} and Equation~\ref{eqn_proofcounter0} imply that
\[n^{-\frac{d-\alpha}{d}} E_\alpha^0\paren{w_1\ldots w_n}\rightarrow c\]
in probability as $n\rightarrow \infty.$ It follows that there exists an $n_k$ so that
\begin{equation}
\label{eqn_construction1}
\abs{n_k^{-\frac{d-\alpha}{d}}E_\alpha^0\paren{w_1,\ldots, w_{n_k}}-c}<\epsilon_k/2\,,
\end{equation}
with probability greater than $1-\epsilon_k.$ $\mu_k$ is supported on intersection of the nested interval collections 
\[\mathcal{I}_k = f_k\paren{T_{b_{k-1}+1}}  \supset f_k\paren{T_{b_{k-1}+2}}\supset f_k\paren{T_{b_{k-1}+3}}\ldots\]
 and $\norm{ f_k\paren{T_{b_{k-1}+j}}}\rightarrow 0$ as $j\rightarrow \infty$ so the hypotheses of Lemma~\ref{lemma:noLimit2} are met; choose $j>M\paren{\epsilon_{k},n_{k}},$ where $M\paren{\epsilon_{k},n_{k}}$ is as defined in that Lemma and set $S_k=f_k\paren{T_{b_{k-1}+j}}.$ Then, by the definition of $M\paren{\epsilon_k,n_k},$ if $\sigma$ is any probability measure so that
\[\sigma\paren{I}=\frac{1}{\abs{S_{k}}}\text{ for $I\in S_{k}$}\]
(that is, $\sigma$ assigns the same probabilities to intervals in $S_{k}$ as $\mu_k$ does) then
\begin{equation}
\label{eqn_construction2}
\abs{n_{k}^{-\frac{d-\alpha}{d}}E_\alpha^0\paren{z_1,\ldots, z_{n_{k}}}-c}<\epsilon_{k}
\end{equation}
with probability greater than $1-\epsilon_k,$ as desired.

The argument for even $k$ is very similar, except we set $\mathcal{I}_k=L\paren{S_{k-1}},$ the intervals of $\mathcal{I}_k$ have length $\frac{5}{7}\paren{\frac{1}{3}}^{b_{k-1}+1},$ $f_k$ is a natural map from $\frac{5}{7}T_{b_{k-1}+1}$ to $\mathcal{I}_k,$ $\mu_k$ is the pushforward of $\nu$ by $f_k,$ and $c$ is replaced by $\paren{\frac{5}{7}}^\alpha c$ in Equations~\ref{eqn_construction1} and~\ref{eqn_construction2}.

Let $\sigma$ be the natural probability measure on $S=\cap_{j}S_j$ (the one that assigns equal probability to the intervals of $S_j$ for all values of $j$), and let $\set{z_j}_{j\in\mathbb{N}}$ be i.i.d. samples from $\sigma.$ By construction
\[n_{2k}^{-\frac{d-\alpha}{d}} E_\alpha^0\paren{z_1,\ldots,z_{n_{2k}}}\rightarrow c\]
but
\[n_{2k+1}^{-\frac{d-\alpha}{d}} E_\alpha^0\paren{z_1,\ldots,z_{n_{2k+1}}}\rightarrow \paren{\frac{5}{7}}^\alpha c\]
in probability as $k\rightarrow\infty.$

To complete the proof, we will show that $\sigma$ is $d$-Ahlfors regular. Let $x\in S.$ As a first case, let $m\in\mathbb{N}$ and consider the ball of radius $\frac{1}{3^m} $ centered at $x,$ an interval of length $\frac{2}{3^m}.$ $S_{m+1}$ contains $2^m$ intervals whose lengths are either $\frac{1}{3^m}$ or $\frac{5}{7}\frac{1}{3^m}.$ $B_{3^{-m}}\paren{x}$ contains at least $1$ interval of $S_{m+1}$ (the one that has $x$ as an element) and intersects at most $4$ such intervals. Therefore,
\[\sigma\paren{ B_{3^{-m}}\paren{x}}\geq \frac{1}{\abs{S_m}}=2^{-m}=\paren{3^{-m}}^{d}\]
and
\[\sigma\paren{ B_{3^{-m}}\paren{x}}\leq \frac{4}{\abs{S_m}}=4\paren{2^{-m}}=4\paren{3^{-m}}^{d}\,.\]

Let $0<\delta<1,$ $\epsilon_0=3^{\floor{\log_3\paren{\delta}}},$ and $\epsilon_1=3^{\ceil{\log_3\paren{\delta}}},$ so
\[\epsilon_0 \leq  \delta < 3 \epsilon_0 \text{ and } \epsilon_1/3 < \delta \leq \epsilon_1\,.\]

By our previous computations,

\[\sigma\paren{ B_{\delta}\paren{x}}\geq \sigma\paren{B_{\epsilon_0}\paren{x}} \geq \epsilon_0^{d} \geq 3^{-d} \delta^d\]
and
\[\sigma\paren{ B_{\delta}\paren{x}}\leq \sigma\paren{B_{\epsilon_1}\paren{x}} \leq 4\epsilon_1^{d} \leq 4\paren{3^{d} \delta^d}.\]
Therefore, $\sigma$ is $d$-Ahlfors regular with $\delta_0=1$ and $c=4\paren{3^d}.$ 
 
\end{proof}

The construction for general $d\in\paren{0,1}$ is nearly identical, but is based on the middle-$\beta$ Cantor set rather than the middle thirds Cantor set. Let $0<\beta<1$ and $\gamma=\frac{1-\beta}{2}.$ For an interval $I$ define
\[K_{\beta}\paren{I}=\set{\gamma I, \gamma I + \paren{1-\gamma}\abs{I}}\,,\]
so $K_\beta\paren{I}$ consists of two intervals of length $\gamma\abs{I}$ obtained by removing an interval of length $\beta\abs{I}$ from the middle of $I.$ Let $T_1^\beta=\brac{0,1},$ and inductively define $T_k^\beta=K\paren{T_{k-1}^\beta}.$ If $T^\beta=\cap_{k} T_k^\beta,$ then $T$ is the union of two separated copies of itself rescaled by $\gamma,$ and the natural measure on $T$ is Ahlfors regular of dimension 
\[d\coloneqq \frac{\log\paren{2}}{\log\paren{1/\gamma}}=\frac{\log\paren{2}}{\log\paren{2}-\log\paren{1-\beta}}\,.\]
Note that $d$ ranges between $1$ and $0$ as $\beta$ ranges from $0$ to $1.$ To finish the construction, repeat the previous argument verbatim except replace $\frac{5}{7}$ with a scaling factor $\eta$ so that 
\[1-\beta<\eta<1\]
and $\frac{7}{5}$ with $\frac{1}{\eta}.$ This will produce a $d$-Ahlfors regular measure so that if $0<\alpha<d$ then $n^{-\frac{d-\alpha}{d}} E_\alpha^0\paren{x_1,\ldots,x_{n}}$ oscillates between a positive constant and $\eta^\alpha$ times that constant.

\section{Scaling of Persistent Homology}

\begin{figure}
\centering  
\subfigure[]{\includegraphics[width=0.3\linewidth]{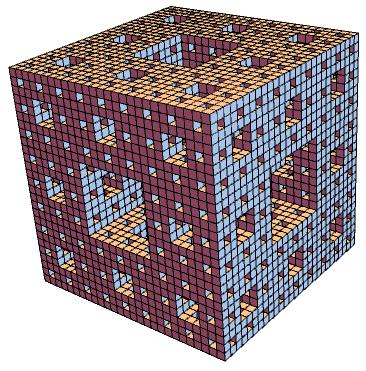}}
\subfigure[]{\includegraphics[width=0.3\linewidth]{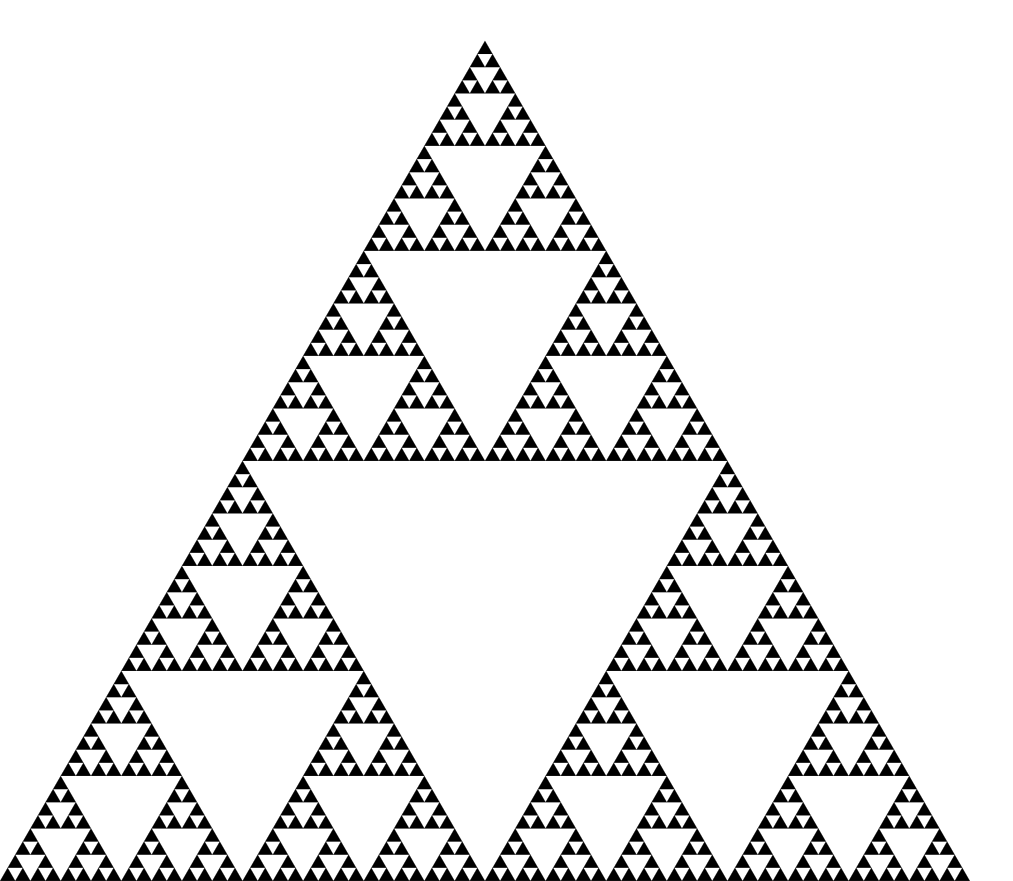}}
\subfigure[]{\includegraphics[width=0.3\linewidth]{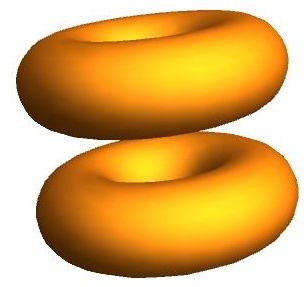}}
	\caption{\label{fig:appendix_examples} The (a) Menger sponge, (b) Sierpi\'{n}ski triangle, and (c) two stacked tori. Figures were generated in Mathematica.}
		\end{figure}

We provide computational evidence that the hypotheses of Theorem~\ref{thm_probabalistic} hold in many cases. We examine four examples in $\mathbb{R}^3$ --- the natural measures on the Menger sponge and the Sierpi\'{n}ski triangle cross an interval, the uniform measure on two tori stacked one above the other, and empirical data from earthquake hypocenters. See Figure~\ref{fig:appendix_examples}. The first three are Ahflors regular measures, with dimensions of $\frac{\log\paren{20}}{\log\paren{3}}\approx 2.727,$ $1+\frac{\log\paren{3}}{\log\paren{2}}\approx 2.585,$ and $2,$ respectively. Note that $\gamma_1^2 \leq 2.5$~\cite{2018schweinhart}, so the first two examples are known to meet all requirements of Theorem~\ref{thm_probabalistic} for $i=1$ except for perhaps the scaling of the expectation and variance of the number of intervals.

\begin{figure}
\centering  
\subfigure[]{\includegraphics[width=0.4\linewidth]{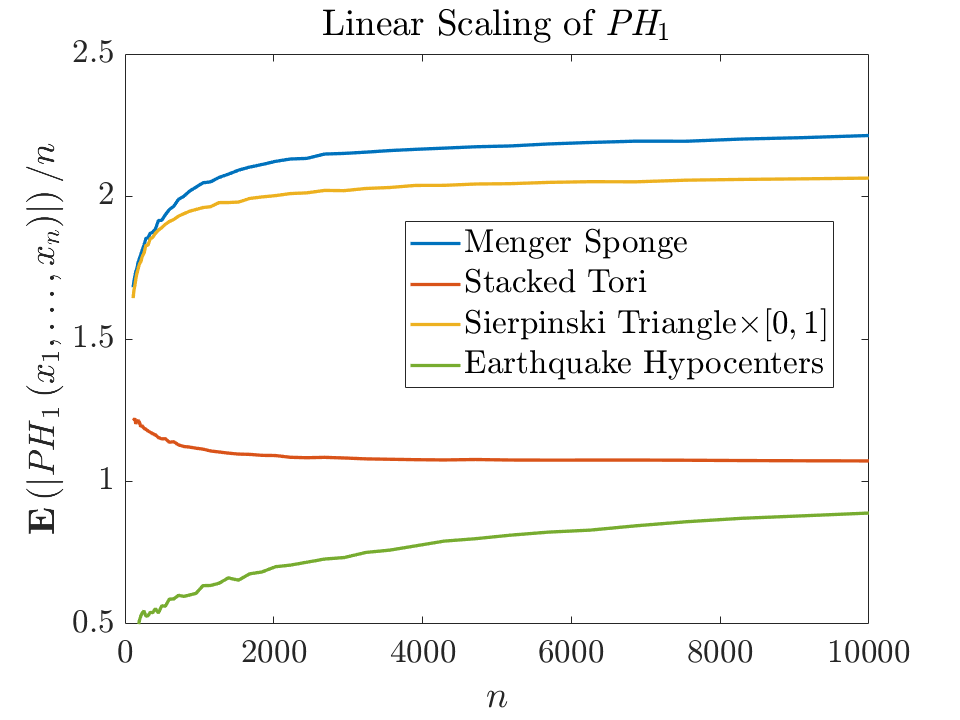}}
\subfigure[]{\includegraphics[width=0.4\linewidth]{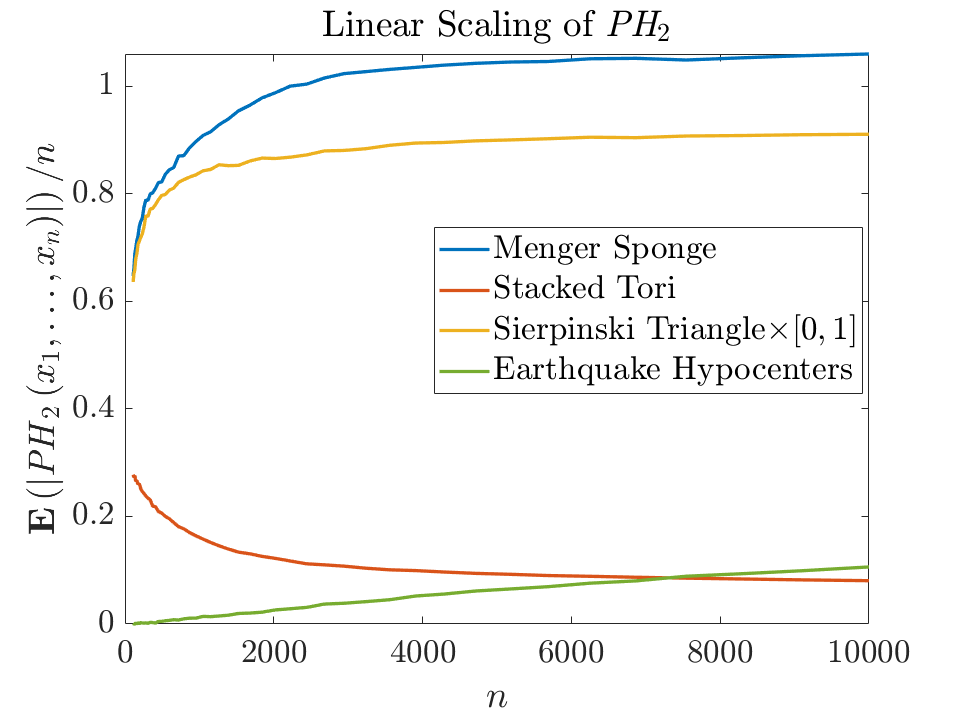}}
	\caption{\label{fig:expectation} Scaling of $\abs{\PH_i\paren{x_1,\ldots,x_n}}/n$ for four examples, and $i=1,2.$}
	\end{figure}
	
	\begin{figure}
\centering  
\subfigure[]{\includegraphics[width=0.4\linewidth]{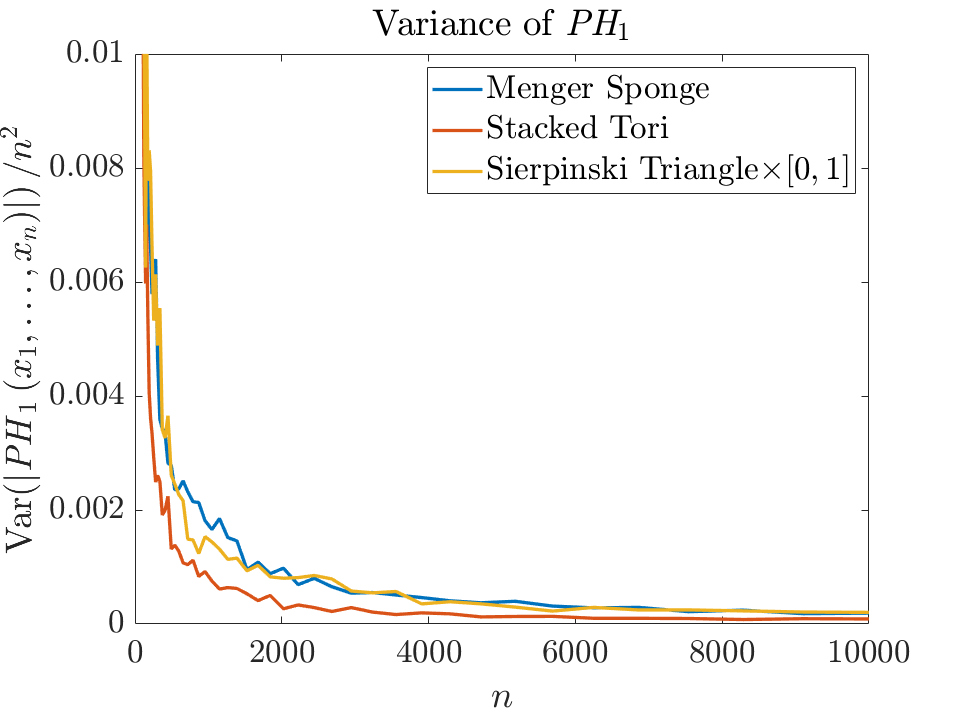}}
\subfigure[]{\includegraphics[width=0.4\linewidth]{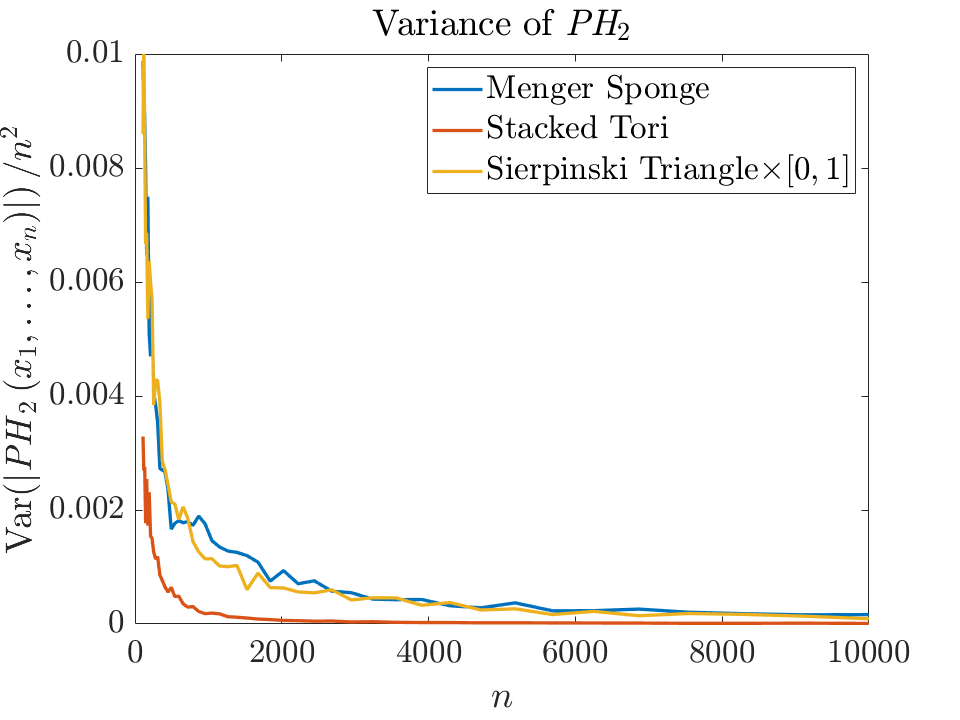}}
	\caption{\label{fig:variance} Variance of $\abs{\PH_i\paren{x_1,\ldots,x_n}}$ divided by $n^2$ for three examples, and $i=1,2.$}
	\end{figure}

We sample points from the natural measures on the Menger sponge and the Sierpi\'{n}ski triangle using the procedures described in~\cite{2019jaquette}. The rejection sampling algorithm developed in~\cite{2013diaconis} was used to sample points from the uniform distribution of the torus $\left({\sqrt {x^{2}+y^{2}}}-R\right)^{2}+z^{2}=r^{2}$ with $R=2$ and $r=1.$ The $z$-coordinate was translated by $3$ with probability $\frac{1}{2}.$ The earthquake hypocenter data comes from the  Hauksson--Shearer Waveform Relocated Southern California earthquake catalog~\cite{2012hauksson,2007lin,2013scedc}; data was processed as in~\cite{2019jaquette}. Persistent homology was computed using the implementation of the Alpha complex in GUDHI~\cite{GUDHI_persistence, GUDHI_alpha}.

Figure~\ref{fig:expectation} shows the empirical expectation of $\PH_i\paren{x_1,\ldots,x_n}$ divided by $n$ for each of the four examples, and $i=1,2.$ The expectation was averaged over $100$ trials for each example except for the earthquake data, which was averaged over $7.$ In each case, the quantity appears to limit to a constant with $n,$ indicating linear scaling of the number of intervals. Figure~\ref{fig:variance} shows the empirical variance of $\PH_i\paren{x_1,\ldots,x_n}$ divided by $n^2$ for the three regular examples. This quantity decreases toward zero for all examples.

\end{document}